\documentclass[12pt]{article}
\usepackage[utf8]{inputenc}
\usepackage[all]{xy}
\usepackage{amsthm, amsmath, amssymb, enumerate,graphicx}
\usepackage{fullpage}
\usepackage{tikz-cd}
\usepackage{mathtools}
\usepackage{colonequals}
\usepackage{enumitem}
\usepackage{mathrsfs}
\usepackage[colorlinks=true,backref=page]{hyperref}

\begin{document}
\title{A counting argument for the geometric Bombieri--Lang conjecture on ramified covers of abelian varieties}
\author{Guoquan Gao}
\maketitle

\begin{abstract}
We prove the geometric Bombieri--Lang conjecture for projective varieties which have finite maps to abelian varieties over function fields of characteristic 0. This generalizes the recent results of Xie--Yuan, which require either the hyperbolicity assumption or the non-isotriviality assumption. The proof builds upon their strategy for constructing entire curves, yet hinges crucially on a new counting argument and draws substantially on tools from algebraic geometry and Nevanlinna theory to overcome various technical difficulties.
\end{abstract}

\theoremstyle{plain}
\newtheorem{thm}{Theorem}[section]
\newtheorem{theorem}[thm]{Theorem}
\newtheorem{cor}[thm]{Corollary}
\newtheorem{corollary}[thm]{Corollary}
\newtheorem{lem}[thm]{Lemma}
\newtheorem{lemma}[thm]{Lemma}
\newtheorem{pro}[thm]{Proposition}
\newtheorem{proposition}[thm]{Proposition}
\newtheorem{prop}[thm]{Proposition}
\newtheorem{definition}[thm]{Definition}
\newtheorem{assumption}[thm]{Assumption}
\newtheorem{conjecture}[thm]{Conjecture}
\newtheorem{sublemma}[thm]{Sub-lemma}

\theoremstyle{remark}
\newtheorem{remark}[thm]{Remark}
\newtheorem{example}[thm]{Example}
\newtheorem{remarks}[thm]{Remarks}
\newtheorem{problem}[thm]{Problem}
\newtheorem{exercise}[thm]{Exercise}
\newtheorem{situation}[thm]{Situation}
\newtheorem{Question}[thm]{Question}

\numberwithin{equation}{subsection}

\newcommand{\ZZ}{\mathbb{Z}}
\newcommand{\CC}{\mathbb{C}}
\newcommand{\QQ}{\mathbb{Q}}
\newcommand{\RR}{\mathbb{R}}
\newcommand{\HH}{\mathcal{H}}     

\newcommand{\ad}{\mathrm{ad}}            
\newcommand{\NT}{\mathrm{NT}}
\newcommand{\nonsplit}{\mathrm{nonsplit}}
\newcommand{\Pet}{\mathrm{Pet}}
\newcommand{\Fal}{\mathrm{Fal}}

\newcommand{\cs}{{\mathrm{cs}}}

\newcommand{\ZZn}{\mathbb{Z}[1/n]}
\newcommand{\ZZN}{\mathbb{Z}[1/N]}

\newcommand{\pair}[1]{\langle {#1} \rangle}
\newcommand{\wpair}[1]{\left\{{#1}\right\}}
\newcommand{\wh}{\widehat}
\newcommand{\wt}{\widetilde}

\newcommand\Spf{\mathrm{Spf}}

\newcommand{\lra}{{\longrightarrow}}

\newcommand{\matrixx}[4]
{\left( \begin{array}{cc}
  #1 &  #2  \\
  #3 &  #4  \\
 \end{array}\right)}        


\newcommand{\BA}{{\mathbb {A}}}
\newcommand{\BB}{{\mathbb {B}}}
\newcommand{\BC}{{\mathbb {C}}}
\newcommand{\BD}{{\mathbb {D}}}
\newcommand{\BE}{{\mathbb {E}}}
\newcommand{\BF}{{\mathbb {F}}}
\newcommand{\BG}{{\mathbb {G}}}
\newcommand{\BH}{{\mathbb {H}}}
\newcommand{\BI}{{\mathbb {I}}}
\newcommand{\BJ}{{\mathbb {J}}}
\newcommand{\BK}{{\mathbb {K}}}
\newcommand{\BL}{{\mathbb {L}}}
\newcommand{\BM}{{\mathbb {M}}}
\newcommand{\BN}{{\mathbb {N}}}
\newcommand{\BO}{{\mathbb {O}}}
\newcommand{\BP}{{\mathbb {P}}}
\newcommand{\BQ}{{\mathbb {Q}}}
\newcommand{\BR}{{\mathbb {R}}}
\newcommand{\BS}{{\mathbb {S}}}
\newcommand{\BT}{{\mathbb {T}}}
\newcommand{\BU}{{\mathbb {U}}}
\newcommand{\BV}{{\mathbb {V}}}
\newcommand{\BW}{{\mathbb {W}}}
\newcommand{\BX}{{\mathbb {X}}}
\newcommand{\BY}{{\mathbb {Y}}}
\newcommand{\BZ}{{\mathbb {Z}}}

\newcommand{\CA}{{\mathcal {A}}}
\newcommand{\CB}{{\mathcal {B}}}
\newcommand{\CD}{{\mathcal{D}}}
\newcommand{\CE}{{\mathcal {E}}}
\newcommand{\CF}{{\mathcal {F}}}
\newcommand{\CG}{{\mathcal {G}}}
\newcommand{\CH}{{\mathcal {H}}}
\newcommand{\CI}{{\mathcal {I}}}
\newcommand{\CJ}{{\mathcal {J}}}
\newcommand{\CK}{{\mathcal {K}}}
\newcommand{\CL}{{\mathcal {L}}}
\newcommand{\CM}{{\mathcal {M}}}
\newcommand{\CN}{{\mathcal {N}}}
\newcommand{\CO}{{\mathcal {O}}}
\newcommand{\CP}{{\mathcal {P}}}
\newcommand{\CQ}{{\mathcal {Q}}}
\newcommand{\CR }{{\mathcal {R}}}
\newcommand{\CS}{{\mathcal {S}}}
\newcommand{\CT}{{\mathcal {T}}}
\newcommand{\CU}{{\mathcal {U}}}
\newcommand{\CV}{{\mathcal {V}}}
\newcommand{\CW}{{\mathcal {W}}}
\newcommand{\CX}{{\mathcal {X}}}
\newcommand{\CY}{{\mathcal {Y}}}
\newcommand{\CZ}{{\mathcal {Z}}}

\newcommand{\ab}{{\mathrm{ab}}}
\newcommand{\Ad}{{\mathrm{Ad}}}
\newcommand{\an}{{\mathrm{an}}}
\newcommand{\Aut}{{\mathrm{Aut}}}

\newcommand{\Br}{{\mathrm{Br}}}
\newcommand{\bs}{\backslash}
\newcommand{\bbs}{\|\cdot\|}

\newcommand{\cod}{{\mathrm{cod}}}
\newcommand{\cont}{{\mathrm{cont}}}
\newcommand{\cl}{{\mathrm{cl}}}
\newcommand{\criso}{{\mathrm{criso}}}

\newcommand{\dR}{{\mathrm{dR}}}
\newcommand{\disc}{{\mathrm{disc}}}
\newcommand{\Div}{{\mathrm{Div}}}
\renewcommand{\div}{{\mathrm{div}}}

\newcommand{\Eis}{{\mathrm{Eis}}}
\newcommand{\End}{{\mathrm{End}}}

\newcommand{\Frob}{{\mathrm{Frob}}}

\newcommand{\Gal}{{\mathrm{Gal}}}
\newcommand{\GL}{{\mathrm{GL}}}
\newcommand{\GO}{{\mathrm{GO}}}
\newcommand{\GSO}{{\mathrm{GSO}}}
\newcommand{\GSp}{{\mathrm{GSp}}}
\newcommand{\GSpin}{{\mathrm{GSpin}}}
\newcommand{\GU}{{\mathrm{GU}}}
\newcommand{\BGU}{{\mathbb{GU}}}

\newcommand{\Hom}{{\mathrm{Hom}}}
\newcommand{\Hol}{{\mathrm{Hol}}}
\newcommand{\HC}{{\mathrm{HC}}}

\renewcommand{\Im}{{\mathrm{Im}}}
\newcommand{\Ind}{{\mathrm{Ind}}}
\newcommand{\inv}{{\mathrm{inv}}}
\newcommand{\Isom}{{\mathrm{Isom}}}

\newcommand{\Jac}{{\mathrm{Jac}}}
\newcommand{\JL}{{\mathrm{JL}}}

\newcommand{\Ker}{{\mathrm{Ker}}}
\newcommand{\KS}{{\mathrm{KS}}}

\newcommand{\Lie}{{\mathrm{Lie}}}

\newcommand{\new}{{\mathrm{new}}}
\newcommand{\NS}{{\mathrm{NS}}}

\newcommand{\ord}{{\mathrm{ord}}}
\newcommand{\ol}{\overline}

\newcommand{\rank}{{\mathrm{rank}}}

\newcommand{\PGL}{{\mathrm{PGL}}}
\newcommand{\PSL}{{\mathrm{PSL}}}
\newcommand{\Pic}{\mathrm{Pic}}
\newcommand{\Prep}{\mathrm{Prep}}
\newcommand{\Proj}{\mathrm{Proj}}

\newcommand{\Picc}{\mathcal{P}ic}

\renewcommand{\Re}{{\mathrm{Re}}}
\newcommand{\red}{{\mathrm{red}}}
\newcommand{\sm}{{\mathrm{sm}}}
\newcommand{\sing}{{\mathrm{sing}}}
\newcommand{\reg}{{\mathrm{reg}}}

\newcommand{\tor}{{\mathrm{tor}}}
\newcommand{\tr}{{\mathrm{tr}}}

\newcommand{\ur}{{\mathrm{ur}}}

\newcommand{\vol}{{\mathrm{vol}}}

\newcommand{\ds}{\displaystyle}

\tableofcontents

\section{Introduction}
The far-reaching Bombieri--Lang conjecture is a higher-dimensional generalization of the celebrated Mordell conjecture proved by Faltings \cite{zbMATH03944027}. Beyond the Mordell conjecture, an important known case of the Bombieri--Lang conjecture is the case of subvarieties of abelian varieties by Faltings \cite{zbMATH04214190, zbMATH00721811}.

The geometric Bombieri--Lang conjecture is an analogue of the Bombieri--Lang conjecture over function fields. There are several versions of it. In the case of characteristic 0, \cite[Conjecture 1.1]{arXiv:2305.14789} formulates a version that gives a very specific description of the structure of the rational points on a projective variety (over the function field). This version also implies most of other versions (see \cite[\S2.4]{arXiv:2305.14789} for more details). Thus we will use this version and focus on the case of characteristic 0 throughout this article.

The conjecture (or its suitable versions) is proved for curves by \cite{zbMATH03266293, zbMATH03223435}, for smooth projective varieties with ample cotangent bundles by \cite{zbMATH03719282, zbMATH03849398, zbMATH06923798}, for subvarieties of abelian varieties by \cite{zbMATH03832101, zbMATH00016118, zbMATH00224044, zbMATH00957009}, and for constant hyperbolic varieties by \cite[Corollary 4.2]{zbMATH00090953}. See \cite{zbMATH03645250} (or \cite[Theorem 2.3]{MR4457668}), \cite[Theorem E, Theorem 6.3]{arXiv:2310.06784} for further results on constant varieties.

Recently, Xie and Yuan considered the case that $X$ has a finite map $f\colon X\to A$ to an abelian variety $A$ over $K$. Here $K$ is a finitely generated field over a field $k$ of characteristic 0. They proved in \cite{arXiv:2305.14789} the case when $X$ is hyperbolic, and in \cite{arXiv:2308.08117} the case when the $K/k$-trace $A^{(K/k)}$ of $A$ is 0. We refer to \cite{zbMATH05233949} for more details on Chow's trace of abelian varieties. It is worth mentioning that Parshin's method in \cite{zbMATH00016118} can also be used to prove the intersection of these two cases; see \cite{zbMATH08143088} for more details. Besides, \cite{zbMATH07629510, arXiv:2507.21468} also provides some results on the sparsity of rational points for ramified covers of abelian varieties.

The goal of this article is to remove both of the conditions above. Namely, we neither assume that $X$ is hyperbolic nor that $A$ is traceless. Unlike most problems over function fields, a non-trivial trace creates essential difficulty here; see the discussion in \S\ref{setup}. Our approach starts from the strategy for constructing entire curves in \cite{arXiv:2308.08117}, but hinges on a new counting argument, which leads to several new technical difficulties. To overcome these difficulties, we require extensive use of tools from algebraic geometry and Nevanlinna theory. We will state the main result in \S\ref{1.1} and introduce the ideas in \S\ref{setup}, \S\ref{idea} and \S\ref{jishu}. Specifically, \S\ref{setup} introduces some basic reductions and setups, \S\ref{idea} then presents the key counting argument. Finally \S\ref{jishu} addresses three technical difficulties and their resolutions.

\subsection{The main result}\label{1.1}
Now we state the main result. By a variety, we mean an integral separated scheme of finite type over a field. Recall that the \emph{algebraic special set} ${\rm Sp}(X)$\footnote{In \cite{arXiv:2305.14789,arXiv:2308.08117}, they used ${\rm Sp_{alg}}(X)$ for the algebraic special set, we omit the subscript ``alg'' in this article.} of  a projective variety $X$ over a field $K$ is the Zariski closure in $X$ of the union of the images of all non-constant rational maps from abelian varieties over $\ol{K}$ to $X_{\ol{K}}$.
\begin{theorem}\label{main}
Let $K$ be a finitely generated field over a field $k$ of characteristic \emph{0} such that $k$ is algebraically closed in $K$. Let $X$ be a projective variety over $K$ with a finite morphism $f\colon X\to A$ to an abelian variety $A$ over $K$. Let $Z$ be the Zariski closure of $(X\backslash{\rm Sp}(X))(K)$ in $X$.

Then there is a finite set $\{Z_{1},\dots, Z_{r}\}$ of distinct closed subvarieties of $Z$ containing all irreducible components of $Z$ and satisfying the following conditions:

\begin{enumerate}[label=\normalfont(\arabic*),leftmargin=*]

\item For each $i$, the normalization $Z_{i}'$ of $Z_{i}$ is constant, i.e., there is a projective variety $T_{i}$ over $k$ and a $K$-isomorphism $\rho_{i}\colon T_{i}\times_{k}K\to Z_{i}'$.

\item The set $(X\backslash{\rm Sp}(X))(K)$ is contained in the union over $i=1,\dots, r$ of the images of the composition
$$
T_{i}(k)\ \lra\ (T_{i}\times_{k}K)(K)\stackrel{\rho_{i}}{\lra} Z'_{i}(K)\ \lra\ Z_{i}(K)\ \lra\ Z(K)\ \lra\ X(K).
$$
\end{enumerate}

\end{theorem}
Roughly speaking, the theorem says that all the rational points of $X$ outside the special locus come from the $k$-points of the $T_{i}$'s. Note that the set $\{Z_{1},\dots, Z_{r}\}$ contains all irreducible components of $Z$, but it may contain more elements to cover all rational points as in $(2)$; see \cite[Example 2.12]{arXiv:2305.14789} for an example.

Theorem \ref{main} confirms the geometric Bombieri--Lang conjecture formulated in \cite[Conjecture 1.1]{arXiv:2305.14789} when $X$ has a finite map $f\colon X\to A$ to an abelian variety $A$. Furthermore, it generalizes \cite[Theorem 1.2]{arXiv:2305.14789} and \cite[Theorem 1.1]{arXiv:2308.08117} where one requires $X$ to be hyperbolic (i.e. ${\rm Sp}(X)=\emptyset$) and the other requires the $K/k$-trace of $A$ is 0. We refer to \cite{zbMATH05233949} for more details on Chow's trace of abelian varieties. Note that we do not assume that $f$ is surjective in the theorem. In particular, $X$ is allowed to be a closed subvariety of $A$, and in this case, versions of the geometric Bombieri--Lang conjecture were proved by \cite{zbMATH00721811, zbMATH03832101, zbMATH00016118, zbMATH00224044, zbMATH00957009}.

We have the following corollary, which recovers \cite[Theorem 1.1]{arXiv:2308.08117}.
\begin{corollary}\label{recover}
Let $K$ be a finitely generated field over a field $k$ of characteristic \emph{0} such that $k$ is algebraically closed in $K$. Let $X$ be a projective variety over $K$ with a finite morphism $f\colon X\to A$ to an abelian variety $A$ over $K$. Then $($after suitable finite extensions of $K$ and $k)$ the image of $(X\backslash{\rm Sp}(X))(K)$ in $A(K)$ is contained in a finite union in $A$ of some translations of $(A^{(K/k)})_{K}$. Here $A^{(K/k)}$ is the Chow's trace of $A$ over $k$.

In particular, if $X$ is of \emph{general type} and $A$ is not defined over $k$, then $X(K)$ is not Zariski dense in $X$. If the $K/k$-trace of $A$ is trivial, then $(X\backslash{\rm Sp}(X))(K)$ is finite.
\end{corollary}
\begin{proof}
Let $T_{i}, Z_{i}$ be as in the statement of Theorem \ref{main}. By \cite[Lemma 4.9]{arXiv:2305.14789}, after a finite extension of $K$, the image of $T_{i}\times_{k}K\to Z_{i}\to X\to A$ is contained in some translation $(A^{(K/k)})_{K}+a_{i}\subseteq A$ where $a_{i}\in A(K)$ for all $i$. This proves the first statement. Now if $X$ is of general type and $A$ is not defined over $k$, then ${\rm Sp}(X)\subsetneqq X$ by \cite[Corollary 4.3]{arXiv:2305.14789} (which is essentially due to Ueno and Kawamata) and $(A^{(K/k)})_{K}\subsetneqq A$, which implies that $X(K)$ is not Zariski dense in $X$. The third statement is obvious from the first one.
\end{proof}
To prove Theorem \ref{main}, we will first prove (in \S\ref{3}) the following variant, which generalizes \cite[Corollary 1.2]{arXiv:2308.08117}. Recall that a normal projective $K$-variety $X$ is said to have a \emph{maximal Albanese dimension} if the Albanese morphism $X_{\ol{K}}\to{\rm Alb}(X_{\ol{K}})$ (via a base point of $X(\ol{K})$) is generically finite onto its image. For convenience, we say a projective $K$-variety $X$ has a \emph{maximal Albanese dimension} if the normalization of $X$ has.

\begin{theorem}[maximal Albanese dimension]\label{albanese}
Let $K$ be a finitely generated field over a field $k$ of characteristic \emph{0} such that $k$ is algebraically closed in $K$. Let $X$ be a projective variety of \emph{general type} over $K$. Assume that $X$ has a maximal Albanese dimension. If $X(K)$ is Zariski dense in $X$, then the Albanese variety of $($the normalization of\emph{)} $X$ is defined over $k$ and $X$ is birational to $T_{K}$ for some projective variety $T$ over $k$. Moreover, the complement of the image of ${\rm Im}(T(k)\to T_{K}(K))$ in $X(K)$ $($via a birational map from $T_{K}$ to $X)$ is not Zariski dense in $X$.

If in addition the Albanese morphism of $($the normalization of\emph{)} $X$ is finite, then the normalization of $X$ is isomorphic to the base change $T_{K}=T\times_{k}K$ for a normal projective variety $T$ over $k$.
\end{theorem}

\subsection{Basic setups}\label{setup}
Now we present the reductions and setups for the proof of Theorem \ref{main}. The essential thing to prove is the following proposition, which is the goal of \S\ref{5}.

\begin{proposition}[Proposition \ref{generic1}]\label{red-introduction}
Let $K$ be the function field of a smooth projective curve $B$ over $\CC$. Let $X$ be a projective variety of \emph{general type} over $K$. Let $A$ be an abelian variety over $K$ and let $T$ be a projective variety over $\CC$. Let $f\colon X\to A\times_{K} T_{K}$ be a finite morphism over $K$. Let $\{x_{n}\}_{n\geq1}\subseteq X(K)$ be a \emph{generic} sequence of rational points in $X$; that is, any infinite subsequence of $\{x_{n}\}_{n\geq1}$ is Zariski dense in $X$. If the image of $x_{n}$ in $T_{K}(K)$ lies in ${\rm Im}(T(\CC)\to T_{K}(K))$ for each $n\geq1$, then the height of $\{x_{n}\}_{n\geq1}$ is bounded.
\end{proposition}
Here ``the height of $\{x_{n}\}_{n\geq1}$ is bounded'' means that the Weil height of $\{x_{n}\}_{n\geq1}$ with respect to an ample line bundle $L$ on $X$ is bounded. This definition is independent of the choice of $L$ and the Weil height associated to $L$.\medskip

The derivation of the main Theorem \ref{main} from Proposition \ref{red-introduction} will be given in \S\ref{3}, involving reduction steps similar to those in \cite{arXiv:2305.14789}.

\begin{remark}\quad\medskip

\noindent (1) We explain that $T$ in the proposition plays the role of the ``constant part''. Even if we are only concerned with finite covers of the abelian variety $A$ (i.e. the case where $T$ is trivial), when the $K/\CC$-trace of $A$ is nontrivial, we sometimes transform the problem into the situation where $T$ is nontrivial. In this transformation, $T_{K}$ becomes a constant quotient abelian variety of $A$; see \S\ref{5.3}, ``Basic setups'' for more details.\medskip

\noindent (2) In Proposition \ref{red-introduction}, if we take $T$ to be trivial and $f=(f_{0})_{K}$ arises as the base change of a morphism $f_{0}\colon X_{0}\to A_{0}$ of complex varieties, then $x_{n}$ can be seen as a holomorphic map $x_{n}\colon B\to X_{0}$ and the boundedness of the height of $\{x_{n}\}_{n\geq1}$ is equivalent to the boundedness of the degree of $\{x_{n}(B)\}_{n\geq1}$ in $X_{0}$ (with respect to some ample line bundle on $X_{0}$). This is known by \cite[Corollary 1(3)]{zbMATH06463522}. Therefore, Proposition \ref{red-introduction} is a generalization of this result of Yamanoi.
\end{remark}

\noindent Now we focus on the proof of Proposition \ref{red-introduction}. For convenience, we assume that $f$ is surjective and $A$ has trivial $K/\CC$-trace. This actually captures the most essential case. Let $f(x_{n})=(s_{n},t_{n})$ be the image of $x_{n}$ in $A(K)\times T_{K}(K)$. Then by our assumption, $\{s_{n}\}_{n\geq1}$ is generic in $A$ and $\{t_{n}\}_{n\geq1}$ is generic in $T$. By the condition in the proposition, we see that $t_{n}\in T(\CC)$ for all $n\geq1$. By compactness, after taking a subsequence we may assume that $t_{n}$ tends to $t_{\infty}\in T(\CC)$ in the Euclidean topology.

Take integral models $\CX,\CA$ of $X,A$ over $B$. Then $\CA\times_{\CC}T$ is an integral model of $A\times_{K}T_{K}$ over $B$. Replace $\CX$ by the Zariski closure of ${\rm Im}(X\to X\times_{K}(A\times_{K}T_{K})\to\CX\times_{B}(\CA\times_{\CC}T))$, we can assume that $f\colon X\to A\times T_{K}$ extends to a $B$-morphism $\CX\to\CA\times_{\CC}T$. For convenience we still denote this morphism by $f$. By abuse of notation, we often view $x_{n}$ as a section $x_{n}\colon B\to\CX$ of $\CX\to B$ and similarly for $s_{n}$.

We prove the proposition by contradiction. Assume to the contrary that the height of $\{x_{n}\}_{n\geq1}$ tends to infinity as $n\to\infty$. Then the height of $\{(s_{n},t_{n})\}_{n\geq 1}$ also tends to infinity. As $t_{n}$ is constant for all $n$, we deduce that the height of $\{s_{n}\}_{n\geq1}$ tends to infinity as $n\to\infty$.\medskip\medskip

\noindent\textbf{Xie--Yuan's idea}\medskip

\noindent When $T$ is trivial, the image of the special set $Z={\rm Sp}(X)$ in $A$ is properly contained in $A$. As $\{x_{n}\}_{n\geq1}$ is generic (in $X$), $\{s_{n}\}$ is also generic (in $f(X)$). After removing finitely many points, we can assume that $s_{n}\notin f(Z)$ for all $n$. Based on the unboundedness of the height of $\{s_{n}\}_{n\geq1}$, in \cite{arXiv:2308.08117} Xie and Yuan constructed a linear entire curve in some fibre $\CA_{b} \,(b\in B(\CC))$ which is the limit of some re-parametrization of the sections $\{s_{n}\}_{n\geq1}\colon B\to\CA$ such that this entire curve is not contained in $f(\CZ_{b})$, where $\CZ$ is the Zariski closure of $Z$ in $\CX$. Then they used Brody lemma to lift this entire curve to an entire curve on $\CX_{b}$ which is not contained in $\CZ_{b}$. On the other hand, \cite{arXiv:2305.14789} showed that the $b$ above can be chosen so that the analytic special set\footnote{For a projective variety $X$ over $\CC$, the \emph{analytic special set} of $X$ is defined to be the Zariski closure of the union of the images of all nonconstant holomorphic maps $\phi\colon \CC\to X$ in $X$.} of $\CX_{b}$ equals to $\CZ_{b}$. This leads to a contradiction.

However, when $T$ is not trivial, the situation gets more complicated. In this case, the image of the special set $Z={\rm Sp}(X)$ in $A\times_{K}T_{K}$ may contain $A\times\{t_{\infty}\}$. If this happens, any limit point of the sections $(s_{n},t_{n})\colon B\to\CA\times_{\CC}T(n\geq1)$ is contained in $\CA\times\{t_{\infty}\}\subseteq f(\CZ)$. Thus the strategy in \cite{arXiv:2308.08117} is not sufficient.\medskip\medskip

\noindent\textbf{New idea}\medskip

\noindent To treat the general case, we will employ a rather different approach---a counting argument. We still use the construction of linear entire curves and control their positions. But instead of studying the positions of the entire curves relative to ${\rm Sp}(X)$, we focus on the ramification divisor $R$ of $f$. We will derive a contradiction by estimating the intersection numbers of the sections $\{x_{n}\}_{n\geq1}$ with the Zariski closure $\CR$ of $R$ in $\CX$. Notably, this approach also yields a new proof for \cite[Theorem 1.1]{arXiv:2308.08117}.

\begin{remark}
In fact, we will not use anything about the special set ${\rm Sp}(X)$ in the proof of Proposition \ref{red-introduction} (and also Proposition \ref{generic1}). The only thing we will use for it is in the reduction steps in \S\ref{3}, especially the Green--Griffiths--Lang conjecture for finite covers of abelian varieties, which is purely algebraic geometry. Therefore, although both \cite{arXiv:2308.08117} and the present article use the construction of linear entire curves and require controlling their positions, their underlying philosophies are quite different.
\end{remark}

\subsection{Key idea: a counting argument}\label{idea}
Now we begin to explain the crucial counting argument. We will first present the geometric intuition and then describe the method for its rigorous formulation.\medskip\medskip

\noindent\textbf{Tangency}\medskip

\noindent The idea originates from a ``tangency'' observation. To illustrate it succinctly and intuitively, we temporarily assume that $X$ is smooth. In general we cannot make this assumption, but there is no big difference in this step. We will come back to the difficulty of singular case in \S\ref{jishu}.

Denote by $R\subseteq X$ the ramification divisor of the finite surjective morphism $f\colon X\to A\times_{K}T_{K}$ (see \S\ref{1.3} for the definition). Let $\CR$ be the Zariski closure of $R$ in $\CX$. Denote $D=f(R)$ and $\CD=f(\CR)$. For $b\in B, t\in T$, denote by $\CX_{b,t}$ the fibre of the natural morphism $\CX\to B\times T$ over $(b,t)$. Similarly we can define $\CR_{b,t},\CD_{b,t}$. For any point $x\in X(K)$, let $(s,t)=f(x)$ be the image of $x$ in $A(K)\times T_{K}(K)$. View $x$ as a section $B\to\CX$ and view $(s,t)$ as a section $B\to\CA\times_{\CC}T$. We have the following proposition.

\begin{proposition}\label{qie0}
If $b\in B(\CC)$ such that $x(b)\in\CR$, then the section $(s,t)\colon B\to\CA\times_{\CC}T$ is tangent to $\CD_{\rm red}$ at the point $(s(b),t(b))$. More precisely, the image of the tangent map ${\rm d}(s,t)\colon T_{b}B\to T_{(s(b),t(b))}(\CA\times_{\CC}T)$ at $b$ is contained in $T_{(s(b),t(b))}(\CD_{\rm red})$.
\end{proposition}
Here $\CD_{\rm red}$ is the reduced induced closed subscheme of $\CX$ with underlying set $\CD$. This proposition will be proved (in a slightly more general setting) in \S\ref{5.5} (see Corollary \ref{qie}). \medskip\medskip

\noindent\textbf{Counting argument}\medskip

\noindent Now we assume Proposition \ref{qie0}. By Hurwitz's formula, $K_{X}\sim f^{*}K_{(A\times_{K}T_{K})}+R$. Here $K_{X}$ is a canonical divisor of $X$, and similarly for $K_{(A\times_{K}T_{K})}$. Since the restriction of $f^{*}K_{(A\times_{K}T_{K})}$ to any fibre $X_{t}\,(t\in T)$ is trivial, we get that the restrictions of $K_{X}$ and $R$ to any fibre $X_{t}$ are linearly equivalent. Since $X$ is of general type, $K_{X}$ is big on $X$, and thus $R$ is a big divisor on the fibre $X_{t}$ for general $t\in T$.

This positivity of $R$, combined with the fact that the height of $x_{n}$ tends to infinity as $n\to\infty$, allows us to ``intuitively'' imagine that $\CR$ intersects sections $x_{n}\colon B\to\CX$ at more and more points as $n\to\infty$. By Proposition \ref{qie0}, this implies that $\CD_{\rm red}$ is tangent to sections $(s_{n},t_{n})\colon B\to\CA\times_{\CC}T$ at more and more points as $n\to\infty$. Passing to the limit, we can imagine that the limit entire curve constructed in \cite{arXiv:2308.08117} is tangent to $\CD_{\rm red}$ at ``an excessive number of'' points. As the limit entire curve is contained in the fibre $\CA\times\{t_{\infty}\}$, it is tangent to $(\CD_{\rm red})_{t_{\infty}}$ at ``an excessive number of'' points.

Now if $(\CD_{\rm red})_{t_{\infty}}$ is reduced and not equal to the whole fibre $\CA\times\{t_{\infty}\}$, we get an entire curve on $\CA\times\{t_{\infty}\}$ which is tangent to the reduced divisor $(\CD_{t_{\infty}})_{\rm red}$ on $\CA\times\{t_{\infty}\}$ at ``an excessive number of'' points. Since $A$ has trivial trace, using \cite[Theorem 3.3]{arXiv:2308.08117}, we can construct this limit entire curve so that it is not contained in $\CD_{b,t_{\infty}}$. Moreover, we can choose $b\in B(\CC)$ such that the fibre of the morphism $(\CD_{t_{\infty}})_{\rm red}\to B$ at $b$ is reduced. In this case we can deduce that the limit entire curve is tangent to the reduced divisor $(\CD_{b,t_{\infty}})_{\rm red}$ on $\CA_{b}\times\{t_{\infty}\}$ at ``an excessive number of'' points. However, it is strange that a linear entire curve in a complex abelian variety is tangent at ``an excessive number of'' points to a reduced divisor which does not contain it. This leads to a contradiction.

In general, we have to handle the case when $(\CD_{\rm red})_{t_{\infty}}$ is non-reduced or it is the whole fibre $\CA\times\{t_{\infty}\}$. This leads to the other two technical difficulties and will be treated in the next subsection.
\medskip\medskip

\noindent\textbf{Rigorous realization via Nevanlinna theory}\medskip

\noindent The above is the rough idea, but it is far from rigorous. For example, the meaning of ``an excessive number of'' points is not clear. We will use Nevanlinna theory (see \S\ref{5.1}) to make the proof rigorous.

Specifically, we will employ the counting function in Nevanlinna theory to quantify the number of intersections between $x_{n}(B)$ and $\CR$, as well as the number of tangent points between $(s_{n},t_{n})(B)$ and $\CD_{\rm red}$. Eventually, we will show two things: on the one hand, the positivity of $R$ implies that the tangent points between $(s_{n},t_{n})(B)$ and $\CD_{\rm red}$ account for \textbf{a strictly positive proportion} of all intersection points (in the asymptotic sense; see Lemma \ref{xiajie-}); on the other hand, by Yamanoi's Second Main Theorem (see Theorem \ref{Nevanlinna}(1)) and some limit arguments (Lemma \ref{key1} below and Lemma \ref{limits}), we will prove that this proportion of tangent points \textbf{must be zero} (see Lemma \ref{subtle}), thereby derive a contradiction.

Note that $\{(s_{n},t_{n})\}_{n\geq1}$ are merely sections, not entire curves. The limit entire curve is obtained by re-parametrizing each $(s_{n},t_{n})$ by open discs of increasing radii and then passing to limit (see Theorem \ref{limit}). Therefore, unlike the classical Nevanlinna theory, we have to consider the counting functions (as well as proximity and characteristic functions) for holomorphic maps from an open disc to a complex variety, rather than only from the whole complex plane. Furthermore, we need to vary the holomorphic maps and study the limit behaviour of these three functions in Nevanlinna theory (see Lemma \ref{limits}).

\begin{remark}
In practice, it is more convenient to consider asymptotic behaviour rather than limiting behaviour. Specifically, in the two opposite estimates above, one is easier to handle for each individual section $(s_{n},t_{n})$ (the asymptotic behaviour), while the other fits naturally with the limit entire curve (the limiting behaviour). Hence, in order to get a contradiction, we have two possible routes: either pass from the asymptotic behaviour to the limiting behaviour, or conversely.

Classically, it is more natural to transform both estimates into bounds for the limit entire curve, since more tools are available for entire curves. In our setting, however, the second route---from the limit to the asymptotic---proves to be easier. The essential reason is that the limit entire curve exists only on $\CA\times_{\CC}T$ (it can also be lifted to $\CX$, but its behaviour becomes bad when the fibre $\CX_{t_{\infty}}$ lies in the singular locus of $\CX$), whereas the tangency comes from intersections of $x_{n}$ and $\CR$, which live on $\CX$. Hence we need to use the limit entire curve as a bridge to obtain information about the asymptotic behaviour.
\end{remark}

\subsection{Three technical difficulties}\label{jishu}
As we have mentioned in the last subsection, there are three difficulties in realizing the counting argument above. We now explain each in turn.\medskip \medskip

\noindent\textbf{Non-reducedness of $(\CD_{\rm red})_{t_{\infty}}$}\medskip

\noindent Firstly, $(\CD_{\rm red})_{t_{\infty}}$ might be non-reduced. In this case the tangency of $(\CD_{\rm red})_{t_{\infty}}$ and the entire curve in $\CA\times\{t_{\infty}\}$ is trivial. In order to overcome this problem, we prove the following (slightly surprising) lemma in \S\ref{5.2}.

\begin{lemma}[Lemma \ref{key}]\label{key1}
Let $S$ be a variety over $\CC$. Let $\pi\colon X\to S$ be a morphism of finite type between separated $\CC$-schemes. Fix a point $s_{\infty}\in S(\CC)$ and a sequence of points $s_{n}\in S(\CC)$ tending to $s_{\infty}$ in the Euclidean topology. Then there is a Zariski open dense subset $X_{s_{\infty}}^{\circ}\subset X_{s_{\infty}}$ such that for every point $x_{\infty}\in X_{s_{\infty}}^{\circ}(\CC)$ and every sequence of points $x_{n}\in X_{s_{n}}$ tending to $x_{\infty}$, if there exists a sequence of tangent vectors $v_{n}\in T_{x_{n}}(X_{s_{n}, {\rm red}})$ tending to a tangent vector $v_{\infty}\in T_{x_{\infty}}(X_{s_{\infty}})$ under the Euclidean topology of the tangent bundle $TX$, then we have $v_{\infty}\in T_{x_{\infty}}(X_{s_{\infty}, {\rm red}})$.
\end{lemma}
Roughly speaking, the lemma says that although the fibre $X_{s_{\infty}}$ may be non-reduced, the limit of the tangent bundles $T(X_{s_{n},{\rm red}})$ is ``almost'' $T(X_{s_{\infty},{\rm red}})$. Now we take $S=T$ and $X=\CD_{\rm red}$, and take $\pi$ to be the composition $\CD_{\rm red}\hookrightarrow\CA\times_{\CC}T\to T$. Since $\CD_{\rm red}$ is reduced and $\{t_{n}\}_{n\geq1}$ is Zariski dense in $T$, after taking a subsequence, we may assume that $(\CD_{\rm red})_{t_{n}}$ is reduced for all $n\geq1$. Take $\{t_{n}\}_{n\geq1}$ to be the sequence $\{s_{n}\}_{n\geq1}$ in the lemma. Then the lemma implies that the limit of the tangent bundles $T(\CD_{\rm red})_{t_{n}}=T(\CD_{t_{n},{\rm red}})$ is ``almost'' $T(\CD_{t_{\infty},{\rm red}})$.

In our case, since the section $(s_{n},t_{n})\colon B\to\CA\times_{\CC}T$ is contained in $\CA\times\{t_{n}\}$ for all $n$, we have $T[(s_{n},t_{n})(B)]\cap T\CD_{\rm red}\subseteq T(\CD_{\rm red})_{t_{n}}$. Therefore we can expect that the limit entire curve is indeed tangent to $(\CD_{t_{\infty}})_{\rm red}$ at ``an excessive number of'' points. This resolves the first difficulty.\medskip\medskip

\noindent\textbf{Non-flatness of $\CD$ over $T$}\medskip

\noindent The second difficulty is that $\CD$ may contain the whole fibre $\CA\times\{t_{\infty}\}$. In this case the tangency of $(\CD_{\rm red})_{t_{\infty}}$ and the entire curve in $\CA\times\{t_{\infty}\}$ is also trivial. To overcome this problem, we use a weak semistable reduction theorem (cf. \cite[Theorem 0.3]{zbMATH01443405}). Namely, there is a generically finite proper surjective map $T'\to T$ and a proper birational equidimensional morphism $\CX'\to\CX\times_{T}T'$ such that the fibre of $\CX'\to T'$ are all reduced. This reducedness essentially implies that the ramification divisor of $X'\to A\times_{K}T'_{K}$ is ``horizontal'' over $T'_{K}$, where $X'$ is the generic fibre of the composition $\CX'\to\CX\times_{T}T'\to\CX\to B$. This resolves the second difficulty. 

We emphasize that the essential reason why this operation transforms the problem is that $t_{n}\in T(\CC)$ is constant for each $n\geq1$. Therefore the section $(s_{n},t_{n})\colon B\to\CA\times_{\CC}T$ can be lifted to a section $(s_{n},t'_{n})\colon B\to\CA\times_{\CC}T'$. This is not always possible when $t_{n}\in T_{K}(K)$ is non-constant.
\medskip\medskip

\noindent\textbf{Problem caused by singularities}\medskip

\noindent However, the careful reader may notice that the operation in the previous paragraph may change the smoothness of $X$, this leads to the third difficulty. In general, we cannot assume that $X$ is smooth. When $X$ is normal but not smooth, the ramification divisor $R$ of $f\colon X\to A\times_{K}T_{K}$ may not be $\QQ$-Cartier, this makes an issue to define the intersection number of $\CR$ and $x_{n}(B)$. If we take $\widetilde{X}$  to be a resolution of singularity of $X$, we can talk about the intersection number. But in this case the base locus of  $K_{\widetilde{X}}$ may contain the fibre $\widetilde{X}_{t_{\infty}}$, which leads to difficulty in estimating the uniform bound of the intersection numbers $\widetilde{\CR}\cap x_{n}(B)$.

To overcome this problem, we will use the \emph{canonical model} $X_{\rm can}$ of $X$ in the minimal model program, which is a birational model of $X$ with ample $\QQ$-Cartier canonical class. Moreover, it has mild singularities. By the Hacon--Mckernan extension theorem (cf. \cite[Corollary 1.7]{zbMATH05166600}), we have a birational morphism $X_{\rm can}\to X$ whenever $X$ does not contain rational curves. We can use this canonical model to avoid the case that $R$ is not $\QQ$-Cartier or $K_{X}$ has base locus. See \S\ref{5.3} for more details.

\subsection{Notations and terminology}\label{1.3}
For any abelian group $M$ and any ring $R$ containing $\ZZ$, denote $M_{R}=M\otimes_{\ZZ}R$. This apply particularly to $R=\QQ, \RR, \CC$.

By a \emph{variety}, we mean an integral scheme, separated of finite type over the base field. A curve is a 1-dimensional variety.

We use $\sim$ to denote the linear equivalent relation between $\QQ$-Weil or $\QQ$-Cartier divisors.

For a closed subset $Z$ of a scheme $X$, denote by $Z_{\rm red}$ the reduced induced scheme structure on $Z$.

For a normal variety $X$ over a field $k$, define the \emph{canonical sheaf} $\omega_{X}\colonequals i_{*}\omega_{U}$, where $U$ is the smooth locus of $X$ and $i\colon U\hookrightarrow X$ is the natural open immersion. This is a reflexive sheaf of rank 1, so it corresponds to a Weil divisor class $K_{X}$ of $X$. We call $K_{X}$ a \emph{canonical divisor} of $X$.

For a generically finite surjective morphism $f\colon X\to Y$ between normal schemes, define the \emph{ramification divisor} of $f$ as
$$
R(f)\colonequals \sum\limits_{D\subseteq X}{\rm length}_{k(D)}(\Omega_{X/Y})_{D}[D].
$$
where the summation is over all prime Weil divisors of $X$ and $\Omega_{X/Y}$ denotes the sheaf of relative differentials. Moreover, if $X,Y$ are normal varieties over a field $k$ of characteristic 0 with $K_{X}, K_{Y}$ being $\QQ$-Cartier, then $f$ is separable and we have the Hurwitz's formula $K_{X}\sim f^{*}K_{Y}+R(f)$.

By a \emph{function field of one variable} over a field $k$, we mean a finitely generated field $K$ over $k$ of transcendence degree 1 such that $k$ is algebraically closed in $K$. We usually denote by $B$ a smooth projective curve over $k$ with function field $K$. For a projective variety $X$ over $K$, an \emph{integral model} of $X$ over $B$ is a projective variety $\CX$ over $k$ together with a projective flat morphism $\CX\to B$ whose generic fibre is isomorphic to $X$.

For a sequence (of any kind of object) we always mean an infinite sequence. For a subsequence of a given sequence, we also mean an infinite subsequence.

Given a function field $K$ of one variable over a field $k$, a projective variety $X$ over $K$ and a sequence of rational points $\{x_{n}\}_{n\geq1}\subseteq X(K)$, we say that $\{x_{n}\}_{n\geq1}$ have bounded (resp. unbounded) height if for some ample line bundle $L$ on $X$, the Weil height $h_{L}$ associated to $L$ is bounded (resp. unbounded) on $\{x_{n}\}_{n\geq1}$. Note that this definition is independent of the choice of $L$ and $h_{L}$.

All complex analytic varieties are assumed to be reduced and irreducible. For a complex analytic variety $X$ with a point $x\in X$, denote by $T_{x}X$ the complex analytic tangent space of $X$ at $x$ defined by holomorphic derivations. For a holomorphic map $f\colon X\to Y$ with $f(x)=y$, denote by ${\rm d}f\colon T_{x}X\to T_{y}Y$ the induced map between the tangent spaces.

For a variety $X$ over $\CC$, denote $X(\CC)$ its associated complex analytic space (endowed with Euclidean topology). For a point $x\in X(\CC)$, denote $T_{x}X$ for the tangent space $T_{x}(X(\CC))$ for simplicity. Denote by $TX\colonequals {\rm Spec}({\rm Sym}^{*}(\Omega_{X}))$ the (geometric) tangent bundle of $X$. The complex point of $TX$ is represented by a pair $(x,v)$ with $x\in X(\CC)$ and $v\in T_{x}X$ a tangent vector of $X$ at $x$. For a morphism $f\colon X\to Y$ between complex varieties, there is a natural tangent map ${\rm d}f\colon TX\to TY$ induced by $f$, which is algebraic. For $x\in X$, the restriction of ${\rm d}f$ to $T_{x}X\subseteq TX$ coincides with the tangent map mentioned above.

For any real number $r>0$, denote by $\BD_{r}\subseteq\CC$ the open disc of radius $r$ centred at $0$. Let $(Y, d)$ be a metric space. Let $\{r_{n}\}_{n\geq1}$ be a sequence of positive real numbers tending to infinity. Let $\phi_{n}\colon \BD_{r_{n}}\to Y$ be a sequence of continuous maps. We say that $\{\phi_{n}\}_{n\geq1}$ converges to a map $\phi\colon \CC\to Y$ uniformly on any compact subset of $\CC$ if for every $\varepsilon>0$ and every compact subset $\Omega\subseteq\CC$, there is $N\geq1$ such that for every $n\geq N$ and $z\in \Omega$, $d(\phi(z), \phi_{n}(z))\leq\varepsilon$. If such $\phi$ exists, it is unique and continuous. Moreover, if $Y$ is further a complex analytic variety and $\phi_{n}$ are holomorphic, then $\phi$ is holomorphic.

\medskip

\noindent\textbf{Acknowledgements.} I would like to express my sincere thanks to my advisors Junyi Xie and Xinyi Yuan for their valuable insights shared during discussions and for their consistent encouragement throughout the research. I am grateful to She Yang for many useful discussions in Lemma \ref{key}, and to Jihao Liu for providing the reference \cite{zbMATH01443405}. I am appreciate to Ariyan Javanpeykar for many useful communications on the first version of this article. I would also like to thank Ya Deng, Ziyang Gao, Carlo Gasbarri, Katsutoshi Yamanoi and Umberto Zannier for numerous helpful comments on this work.

This work is supported by the National Natural Science Foundation of China Grants No. 12271007, No. 12250004 and No. 12321001.
\section{Locations of limit points}
In this section we recall some results from \cite{arXiv:2308.08117}, which will be used in the next section. \S\ref{lec} is some basic on the linear entire curves (cf. \cite[\S 2]{arXiv:2308.08117}) and \S\ref{limit} explicitly describes the location of limit points of a given sequence of sections of an abelian scheme, which is related to the linear entire curves (cf. \cite[\S3]{arXiv:2308.08117}).
\subsection{Linear entire curves}\label{lec}
Let $A$ be a complex abelian variety. Denote by ${\rm Lie}(A)$ the group of translation-invariant holomorphic derivations on $\CO_{A}$. Let $x\in A$ be a point, and $v\in{\rm Lie}(A)$ be a nonzero vector in the Lie algebra. By restriction, we have a canonical isomorphism ${\rm Lie}(A)\simeq T_{x}A$. By definition, $v$ defines a translation-invariant vector field on $\CA$, and thus it is a foliation $\CF(v)$ on $A$. The leaf $\CF(v)_{x}=\CF(x, v)$ of this foliation through $x$ is called the \emph{linear entire curve through $x$ in the direction $v$}. By translation, we have
$$
\CF(x, v)=\CF(0, v)+x.
$$

Alternatively, write $A={\rm Lie}(A)/ H_{1}(A, \ZZ)$. Denote by $\tilde{x}$ be a lifting of $x$ in ${\rm Lie}(A)$. In the complex vector space ${\rm Lie}(A)$, there is a linear subset given by
$$
\widetilde{\phi}_{(\tilde{x}, v)}\colon \CC\ \lra\ {\rm Lie}(A),\qquad z\mapsto\tilde{x}+zv.
$$
Composing with ${\rm Lie}(A)\to A$, we obtain an entire curve
$$
\phi_{(x, v)}\colon\CC\ \lra\ A.
$$
The entire curve is independent of the choice of $\tilde{x}$, and its image is exactly the linear leaf $\CF(x, v)$. We call $\phi_{(x, v)}\colon\CC\to A$ the \emph{linear entire curve through $x$ in the direction $v$}.

\medskip\medskip

\noindent\textbf{Zariski closures of linear entire curves}\medskip

\noindent Linear entire curves are generally not algebraic. We recall a clear description in \cite{arXiv:2308.08117} of Zariski closures of linear leaves in fibres of abelian schemes.

Let $B$ be a smooth quasi-projective cure over $\CC$, and let $\pi\colon\CA\to B$ be an abelian scheme.  Let $K=\CC(B)$ be the function field of $B$, and $A$ be the generic fibre of $\pi\colon\CA\to B$. By the Lang--N\'{e}ron theorem (cf. \cite[Theorem 2.1]{zbMATH05233949}),
$$
V(A, K)=(A(K)/A^{(K/\CC)}(\CC))_{\RR}
$$
is a finite dimensional $\RR$-vector space. For an abelian subvariety $H$ of $A$, there is a natural inclusion
$$
V(H, K)=(H(K)/H^{(K/\CC)}(\CC))_{\RR}=(H(K)/H(K)\cap A^{(K/\CC)}(\CC))_{\RR}\subseteq V(A, K).
$$
Moreover, Xie--Yuan proved in \cite[Lemma 2.4]{arXiv:2308.08117} that for every $s\in V(A, K)$, there exists a \emph{unique minimal abelian subvariety} $G(s)$ of $A$ with $s\in V(G(s), K)$.

The following is the main proposition cited in this subsection.
\begin{proposition}[{\cite[Proposition 2.6]{arXiv:2308.08117}}]\label{closure}
Let $B$ be a smooth quasi-projective curve over $\CC$, and let $\pi\colon\CA\to B$ be an abelian scheme. Let $K=\CC(B)$ be the function field of $B$, and $A$ be the generic fibre of $\pi\colon\CA\to B$. Assume that the natural map ${\rm End}_{K}(A)\to\rm{End}_{\ol{K}}(A_{\ol{K}})$ is an isomorphism. Let $s$ be a nonzero element of $V(A, K)$. Denote by $\CG$ the Zariski closure of $G(s)$ in $\CA$. Then there is a subset $S(A)$ of $B(\CC)$ of measure $0$ such that for any $b\in B(\CC)\backslash S(A)$, the Zariski closure of the linear leaf $\CF(0, \delta(s, v_{b}))$ in $\CA_{b}$ equals to $\CG_{b}$. Here $v_{b}\in T_{b}B$ is a nonzero vector.
\end{proposition}
Here the transfer map (associated to $s$) $\delta(s,\cdot)\colon T_{b}B\to{\rm Lie}(\CA_{b})$ is defined in \cite[\S 2.2]{arXiv:2308.08117} using the Betti map. A key property of this transfer map is the following non-degeneracy theorem.
\begin{theorem}[{\cite[Theorem 2.2]{arXiv:2308.08117}}]\label{non-degeneracy}
$(B, \pi\colon\CA\to B, K, A)$ are as above. Let $s\in V(A, K)$ be a nonzero element. Let $\Sigma(s)$ be the set of $b\in B(\CC)$ such that the transfer map
$$
\delta(s, \cdot)\colon T_{b}B\ \lra\ {\rm{Lie}}(\CA_{b})
$$
is zero. Then $\Sigma(s)$ has measure $0$ in $B$ with respect to any K\"{a}hler form on $B$.
\end{theorem}
\subsection{Locations of limit points}\label{limit}
The following theorem gives a concrete construction and a precise description of linear entire curves obtained from sections of abelian schemes by a Brody-type convergence.

For a real number $r>0$, denote by $\BD_{r}\subseteq\CC$ the open disc of radius $r$ centred at $0$. Write $\BD=\BD_{1}$. Denote by $v_{\rm st}=\frac{\rm d}{{\rm d}z}$ the tangent vector of $\BD_{r}$ at 0 under the standard coordinate $z$.

\begin{theorem}[{\cite[Theorem 3.1(a)]{arXiv:2308.08117}}]\label{linear}
Let $B$ be a smooth quasi-projective curve over $\CC$, and let $\pi\colon\CA\to B$ be an abelian scheme. Let $K=\CC(B)$ be the function field of $B$, and $A$ be the generic fibre of $\pi\colon\CA\to B$. Let $\{s_{n}\}_{n\geq1}$ be a sequence in $\CA(B)$ and $\{b_{n}\}_{n\geq1}$ be a sequence in $B$ satisfying the following properties:

\begin{enumerate}[label=\normalfont(\arabic*),leftmargin=*]

\item $b_{n}$ converges to a point $b\in B$, and $s_{n}(b_{n})$ converges to a point $x\in \CA_{b}$.

\item There is a sequence $\ell_{n}$ of positive real numbers converging to infinity such that the image of $\ell_{n}^{-1}s_{n}$ in $V(A, K)$ converges to a nonzero element $s_{\infty}$ in $V(A, K)$.

\item The transfer map $\delta(s_{\infty}, \cdot)\colon T_{b}B\to {\rm{Lie}}(\CA_{b})$ is nonzero.

\end{enumerate}

\noindent Define a re-parametrization $\{\phi_{n}\colon \BD_{r_{n}}\to\CA\}_{n\geq1}$ of $\{s_{n}\colon B\to\CA\}_{n\geq1}$ as follows. Take an open unit disc $\BD$ in $B$ with centre $b$, and assume that $b_{n}\in\BD$ for all $n\geq1$.Let $z$ be the standard coordinate of $\BD$. Take the tangent vector $v_{b}\in T_{b}B$ to be the one represented by $v_{\rm{st}}=\frac{\rm{d}}{\rm{d}z}\in T_{0}\BD$. Define the re-parametrization map
$$
\phi_{n}\colon \BD_{r_{n}}\ \lra\ \CA, \quad\quad z\mapsto s_{n}(b_{n}+l_{n}^{-1}z)
$$
Here the sum $b_{n}+\ell_{n}^{-1}z$ is taken in $\BD$, and $\{r_{n}\}_{n\geq1}$ is a sequence of positive numbers satisfying
$$
{r_{n}}/{\ell_{n}}<1-|b_{n}|, \quad\quad r_{n}\ \lra\ \infty.
$$
Then $\{\phi_{n}\}_{n}$ converges $($uniformly on any compact subset of $\CC)$ to the linear entire curve $\phi_{(x, \delta(s_{\infty}, v_{b}))}$ of $\CA_{b}$ through $x$ in the direction of $\delta(s_{\infty}, v_{b})$ $($which can also be viewed as an entire curve on $\CA)$.
\end{theorem}
Combine Proposition \ref{closure}, Theorem \ref{non-degeneracy} and Theorem \ref{linear}, we have the following corollary.
\begin{corollary}\label{dim}
$(B, \pi\colon\CA\to B, K, A)$ are as above. Let $\{s_{n}\}_{n\geq1}$ be a sequence in $\CA(B)$ satisfying the condition $(2)$ in \emph{Theorem \ref{linear}}. Denote by $G$ the smallest abelian subvariety of $A$ such that $s_{\infty}$ lies in $G(K)_{\RR}$. Denote by $\CG$ the Zariski closure of $G$ in $\CA$. Assume that the natural map ${\rm End}_{K}(A)\to{\rm{End}}_{\ol{K}}(A_{\ol{K}})$ is an isomorphism. Then for all $b\in B(\CC)$ outside a subset of measure $0$, there is a linear entire curve in $\CA_{b}$ which is the unique limit of a re-parametrization of a subsequence $($depending on $b)$ of $\{s_{n}\}_{n\geq1}$, and the Zariski closure of this linear entire curve in $\CA_{b}$ is a translation of $\CG_{b}$.
\end{corollary}
\begin{proof}
By Proposition \ref{closure} and Theorem \ref{non-degeneracy}, there are two subsets $S(A), \Sigma(s_{\infty})$ of $B$ of measure 0 such that for all $b\in B(\CC)\setminus(S(A)\cup\Sigma(s_{\infty}))$, the Zariski closure of the linear leaf $\CF(0, \delta(s_{\infty}, v_{b}))$ in $\CA_{b}$ is equal to $\CG_{b}$, and the transfer map $\delta(s_{\infty},\cdot)\colon T_{b}(B)\to{\rm Lie}(\CA_{b})$ is nonzero. Fix one such $b$. By compactness, there is a limit point $x\in\CA_{b}$ of the sequence $\{s_{n}(b)\}_{n\geq1}$. Passing to a subsequence $\{n_{j}\}_{j\geq1}\subseteq\BZ_{>0}$, taking $b_{n_{j}}=b$ and $x=\lim_{j\to\infty}s_{n_{j}}(b)$ in Theorem \ref{linear}, we get a linear entire curve $\phi_{(x, \delta(s_{\infty}, v_{b}))}$ of $\CA_{b}$ which is the unique limit of a re-parametrization of $\{s_{n_{j}}\}_{j\geq1}$, and its Zariski closure in $\CA_{b}$ is equal to the translation $x+\CG_{b}$. This finishes the proof.
\end{proof}
Using the same method as in \cite[Theorem 3.3]{arXiv:2308.08117}, we can prove the following corollary which will be used in \S\ref{5.4}.
\begin{corollary}\label{location}
$(B, \pi\colon\CA\to B, K, A)$ are as above. Let $\{s_{n}\}_{n\geq1}$ be a sequence of points in $A(K)$. Assume that the image of $\{s_{n}\}_{n\geq1}$ in any nontrivial quotient abelian variety of $A$ has height tending to infinity as $n\to\infty$. And the natural map ${\rm End}_{K}(A)\to{\rm{End}}_{\ol{K}}(A_{\ol{K}})$ is an isomorphism. Then there is a subset $S$ of $B(\CC)$ of measure $0$ such that for every $b\in B(\CC)\backslash S$, there exists a Zariski dense subset $U_{b}\subseteq\CA_{b}$ such that each $y\in U_{b}$ is a limit point of $\{s_{n}(B)\}_{n\geq1}$, i.e. there is a sequence $\{b_{n}\}_{n\geq1}$ in $B(\CC)$ such that the sequence $\{s_{n}(b_{n})\}_{n\geq1}$ has a subsequence converging to $y$ in $\CA$.
\end{corollary}
\begin{proof}
We prove by induction on $\dim A$. When $\dim A=1$, any entire curve on $\CA_{b}$ is surjective, so the statement is trivial by Theorem \ref{linear}. Now we assume that the corollary is true for all abelian varieties with dimension less than $\dim A$.

As in Theorem \ref{linear}, choose $\ell_{n}\to+\infty$ such that the image of $\ell_{n}^{-1}s_{n}$ in $V(A,K)$ converges to a nonzero element $s_{\infty}\in V(A,K)$. Then $G(s_{\infty})$ is a nontrivial abelian subvariety of $A$. If $G(s_{\infty})=A$, then the conclusion holds by Corollary \ref{dim}. Now we assume that $G(s_{\infty})\subsetneqq A$. Let $A'=A/G(s_{\infty})$ be the quotient abelian variety and let $s'_{n}$ be the image of $s_{n}\in A(K)\cong\CA(B)$ in $A'(K)$. Moreover, take an integral model $\CA'$ of $A'$ over $B$. Shrinking $B$ if necessary, we may assume that $\CA'$ is an abelian scheme over $B$ and $A\to A'$ extends to a $B$-morphism $\CA\to\CA'$. Now we have data $(B,\CA'\to B,K,A',\{s'_{n}\}_{n\geq1})$. They also satisfy the assumptions in this corollary. Therefore by induction hypothesis, there is a subset $S'$ of $B(\CC)$ of Lebesgue measure 0 such that for every $b\in B(\CC)\backslash S'$, there exists a Zariski dense subset $U'_{b}\subseteq \CA'_{b}$ such that each $y\in U'_{b}$ is a limit point of $\{s'_{n}(B)\}_{n\geq1}$.

Let $S$ be the union of $S'$ and the exception set of Corollary \ref{dim}, which is also of Lebesgue measure 0. Now for any $b\in B(\CC)\backslash S$ and any point $y'\in U'_{b}$, we can write $y'=\lim_{k\to\infty}s'_{n_{k}}(b_{n_{k}})$, where $\{n_{k}\}$ is a subsequence of $\BZ_{>0}$. Now consider the points $\{s_{n_{k}}(b_{n_{k}})\}_{k\geq1}$. By properness of $\CA\to\CA'$, we know that these points have a limit point $y$ on $\CA$, whose image in $\CA'$ is $y'\in U_{b}$. Now by Corollary \ref{dim} and our assumption on $S$, we see that there is a linear entire curve in $\CA_{b}$ which is the unique limit of a re-parametrization of a subsequence of $\{s_{n_{k}}\}_{k\geq1}$, and the Zariski closure of this entire curve is exactly $y+\CG_{b}$. Here $\CG$ is the Zariski closure of $G(s_{\infty})$ in $\CA$. Note that $y+\CG_{b}$ is just the fibre of $\CA_{b}\to\CA'_{b}$ over the point $y'\in\CA'_{b}$. In other words, for any $y'\in U'_{b}$, there is an entire curve in $\CA_{b}$ which is Zariski dense in the fibre of $\CA_{b}\to\CA'_{b}$ and which is the limit of a re-parametrization of a subsequence of $\{s_{n_{k}}\}_{k\geq1}$.

Let $U_{b}\subseteq \CA_{b}$ be the set of all limit points of $\{s_{n}(B)\}_{n\geq1}$ in $\CA_{b}$. By the argument above, the image of $U_{b}$ in $\CA'$ contains the set $U'_{b}$ which is Zariski dense in $\CA'_{b}$, and the fibre of $U_{b}$ over each $y'\in U'_{b}$ is Zariski dense in the fibre of $\CA\to\CA'$ over $y'$. This just means that $U_{b}$ is Zariski dense in $\CA_{b}$. This finishes the proof.
\end{proof}

\section{Boundedness of height on generic sequences}\label{5}
The goal of this section is to prove Proposition \ref{red-introduction}, which asserts that the height of any generic sequence of rational points in $X(K)$ is bounded. Here $X$ can be a projective variety of general type which has a finite map to an abelian variety. This is the essential step in the proof of Theorem \ref{main}. Now we restate the proposition as follows.

\begin{proposition}[Proposition \ref{red-introduction}]\label{generic1}
Let $K$ be the function field of a smooth projective curve $B$ over $\CC$. Let $X$ be a projective variety of \emph{general type} over $K$. Let $A$ be an abelian variety over $K$ and let $T$ be a projective variety over $\CC$. Let $f\colon X\to A\times_{K} T_{K}$ be a finite morphism over $K$. Let $\{x_{n}\}_{n\geq1}\subseteq X(K)$ be a \emph{generic} sequence of rational points in $X$; that is, any infinite subsequence of $\{x_{n}\}_{n\geq1}$ is Zariski dense in $X$. If the image of $x_{n}$ in $T_{K}(K)$ lies in ${\rm Im}(T(\CC)\to T_{K}(K))$ for each $n\geq1$, then the height of $\{x_{n}\}_{n\geq1}$ is bounded.
\end{proposition}

In \S\ref{5.1} we recall some Nevanlinna theory which will be used in the proof. As mentioned in \S\ref{idea}, we will extend classical Nevanlinna theory to holomorphic maps on open discs and consider the limiting behaviour. In \S\ref{5.2} we prove a key lemma (Lemma \ref{key}), which was already mentioned in \S\ref{jishu} regarding the first difficulty.

The proof of Proposition \ref{generic1} is carried out in the next three subsections. In \S\ref{5.3} we perform an algebraic-geometric reduction, which resolves the latter two difficulties mentioned in \S\ref{jishu}. In \S\ref{5.4} and \S\ref{5.5} we realize the key counting argument in two stages: \S\ref{5.4} obtains an upper bound for the numbers of tangency points between the branch divisor and the sections corresponding to the images in $A$ of certain rational points of $X$, or equivalently, a lower bound for their transversal intersection numbers. Finally in \S\ref{5.5}, we derive a contradiction via a reverse estimate.

We need to mention that the arguments in \S\ref{5.3}, \S\ref{5.4} and \S\ref{5.5} are coherent; although each subsection mainly focuses on its title, certain assumptions needed later are occasionally introduced early and will be explicitly noted. The reader may start directly from \S\ref{5.3} and come back to the previous two subsections when they need.
\subsection{Preliminaries in Nevanlinna theory}\label{5.1}

\noindent\textbf{Counting function}\medskip

\noindent Let $X$ be a projective variety over $\CC$ and let $Z\subseteq X$ be a closed subscheme. Let $f\colon \CC\to X$ be a holomorphic map such that $f(\CC)\nsubseteq Z$. We recall the definition of $N(r,f,Z)$ in \cite{zbMATH02144170}.

Let $D_{1},\dots,D_{l}$ be effective Cartier divisors such that $Z=D_{1}\cap\dots\cap D_{l}$. This means that $\CI_{Z}=\CI_{D_{1}}+\dots+\CI_{D_{l}}$, where $\CI_{Z},\CI_{D_{1}},\dots,\CI_{D_{l}}$ are the defining ideal sheaves in $\CO_{X}$. Set
$$
{\rm ord}_{z}f^{*}Z\colonequals \min\{{\rm ord}_{z}f^{*}D_{1},\dots,{\rm ord}_{z}f^{*}D_{l}\}.
$$
When $f(\CC)\subseteq {\rm Supp}(D_{i})$, we set ${\rm ord}_{z}f^{*}D_{i}=+\infty$. The definition of ${\rm ord}_{z}f^{*}Z$ does not depend on the choice of the Cartier divisors $D_{1},\dots,D_{r}$. We define the counting function by
$$
N(r,f,Z)\colonequals \int_{1}^{r}\left(\sum\limits_{z\in\BD_{t}}{\rm ord}_{z}f^{*}Z\right)\frac{{\rm d}t}{t},
$$
$$
N^{(k)}(r,f,Z)\colonequals \int_{1}^{r}\left(\sum\limits_{z\in\BD_{t}}\min\{k,{\rm ord}_{z}f^{*}Z\}\right)\frac{{\rm d}t}{t}.
$$

Here $\BD_{t}=\{z\in\CC\big| |z|<t\}$ is the open disc of radius $t$ with centre 0. \medskip

\noindent Moreover, when $f\colon \BD_{R}\to X$ is a holomorphic map from an open disc $\BD_{R}\ (R>0)$ to $X$ with $f(\BD_{R})\nsubseteq Z$, we can define the counting function $N(r,f,Z),N^{(k)}(r,f,Z)$ for any $r<R$ in the same way as above.\medskip\medskip

\noindent\textbf{Characteristic function and proximity function}\medskip

\noindent Let $X$ be a projective variety over $\CC$. Let $L$ be a line bundle on $X$ with a smooth Hermitian metric $\left\|\cdot\right\|$ on it. Let $c_{1}(L,\left\|\cdot\right\|)$ be the Chern form. Let $f\colon \CC\to X$ be a holomorphic curve. Define the characteristic function by
$$
T(r,f,L,\left\|\cdot\right\|)\colonequals \int_{1}^{r}\left(\int_{\BD_{t}}f^{*}c_{1}(L,\left\|\cdot\right\|)\right)\frac{{\rm d}t}{t}.
$$
For an effective Cartier divisor $D$ on $X$ associated to a section $s_{D}\in \Gamma(X,L)$, if $f(\CC)\nsubseteq D$, define the proximity function by
$$
m(r,f,D,\left\|\cdot\right\|)\colonequals -\int_{0}^{2\pi}\log\left\|s_{D}(f(r{\rm e}^{{\rm i}\theta}))\right\|\frac{{\rm d\theta}}{2\pi}
$$
As $X$ is compact, $\left\|s_{D}(x)\right\|(x\in X)$ is bounded above by a finite constant. Thus $m(r,f,D,\left\|\cdot\right\|)$ is bounded below by a constant, which depends only on the metric $\left\|\cdot\right\|$. In particular, this lower bound is uniform for all $f$ and $r$.

We also set $T(r,f,D,\left\|\cdot\right\|)\colonequals T(r,f,L,\left\|\cdot\right\|)$. Moreover, for any holomorphic map\\
$f\colon \BD_{R}\to X$, we can define the characteristic function $T(r,f,L,\left\|\cdot\right\|)$ in the same way as above. If in addition $f(\BD_{R})\nsubseteq D$, we can also define $m(r,f,D,\left\|\cdot\right\|)$ for any $r<R$.

The above definitions can be extended by $\QQ$-linearity to the case when $D$ is a $\QQ$-Cartier divisor (or $L$ is a $\QQ$-line bundle with a smooth Hermitian metric on some multiple of $L$).

The next lemma assures that for two different metrics on $L$, the difference of the corresponding characteristic functions and proximity functions are bounded.
\medskip

\begin{lemma}\label{property}
$X,L,D$ as above. Let $\left\|\cdot\right\|, \left\|\cdot\right\|'$ be two smooth Hermitian metrics on $L$. Then there is a constant $C>0$ depending only on these two metrics such that for any holomorphic map $f\colon \BD_{R}\to X$ and any $0<r<R$,

\begin{enumerate}[label=\normalfont(\arabic*),leftmargin=*]

\item $\left|T(r,f,L,\left\|\cdot\right\|)-T(r,f,L,\left\|\cdot\right\|')\right|\leq C$.

\item If in addition $f(\BD_{R})\nsubseteq D$, then $\left|m(r,f,D,\left\|\cdot\right\|)-m(r,f,D,\left\|\cdot\right\|')\right|\leq C$.
\end{enumerate}

\end{lemma}
\begin{proof}
First consider $(1)$. Since $c_{1}(L,\left\|\cdot\right\|)-c_{1}(L,\left\|\cdot\right\|')$ is a smooth exact $(1,1)$-form, the $\partial\ol{\partial}$ lemma yields that there is a smooth function $\varphi$ on $X$ such that $c_{1}(L,\left\|\cdot\right\|)-c_{1}(L,\left\|\cdot\right\|')={\rm dd^{c}}\varphi$. Hence we have
$$
\int_{1}^{r}\frac{{\rm d}t}{t}\int_{\BD_{t}}f^{*}c_{1}(L,\left\|\cdot\right\|)=\int_{1}^{r}\frac{{\rm d}t}{t}\int_{\BD_{t}}f^{*}c_{1}(L,\left\|\cdot\right\|')+\int_{1}^{r}\frac{{\rm d}t}{t}\int_{\BD_{t}}{\rm dd^{c}}\varphi(f(z))
$$
By Stokes' formula, we get
\begin{align*}
\int_{1}^{r}\frac{{\rm d}t}{t}\int_{\BD_{t}}{\rm dd^{c}}\varphi(f(z))&=\int_{1}^{r}\frac{{\rm d}t}{t}\int_{\partial\BD_{t}}{\rm d^{c}}\varphi(f(z))\\
&=\frac{1}{2}\int_{1}^{r}\frac{{\rm d}t}{t}\int_{0}^{2\pi}t\frac{\partial}{\partial t}\varphi(f(t{\rm e}^{{\rm i}\theta}))\frac{{\rm d}\theta}{2\pi}\\
&=\frac{1}{2}\int_{0}^{2\pi}\varphi(f(r{\rm e}^{{\rm i}\theta}))\frac{{\rm d}\theta}{2\pi}-\frac{1}{2}\int_{0}^{2\pi}\varphi(f({\rm e}^{{\rm i}\theta}))\frac{{\rm d}\theta}{2\pi}.
\end{align*}
Since $\varphi$ is bounded, the right-hand side is bounded by a constant $C_{1}>0$ which depends only on $\left\|\cdot\right\|,\left\|\cdot\right\|'$.

For (2), as $X$ is compact, there exists $C_{2}>1$ (depending only on $\left\|\cdot\right\|,\left\|\cdot\right\|'$) such that $C_{2}^{-1}\left\|\cdot\right\|\leq\left\|\cdot\right\|'\leq C_{2}\left\|\cdot\right\|$. Thus $-\log\left\|s_{D}(x)\right\|$ and $-\log\left\|s_{D}(x)\right\|'$ are bounded by $\log C_{2}$. By definition we see that $\left|m(r,f,D,\left\|\cdot\right\|)-m(r,f,D,\left\|\cdot\right\|')\right|\leq \log C_{2}$. Now take $C=\max\{C_{1},\log C_{2}\}$, we finish the proof.
\end{proof}

\begin{remark}
Because of Lemma \ref{property}, most references use $T(r,f,L)$ and $m(r,f,D)$ to denote a real-valued function class on $0<r<R$ (two functions are in the same class if and only if their difference is a bounded function). However, in our application, we need to control the error term uniformly for all $f$ and $r$ (not just for $r$). Therefore, it is more convenient to fix the metrics on all the line bundles we need to use. After doing so, we can make the Nevanlinna First Main Theorem more precisely in Theorem \ref{fmt} below.
\end{remark}

\begin{theorem}[First Main Theorem]\label{fmt}
$X,L,\left\|\cdot\right\|,D$ as above. Then for any holomorphic map $f\colon \BD_{R}\to X$ with $f(\BD_{R})\nsubseteq D$ and any $0<r<R$, we have
 $$
 T(r,f,D,\left\|\cdot\right\|)=m(r,f,D,\left\|\cdot\right\|)+N(r,f,D)-m(1,f,D,\left\|\cdot\right\|).
 $$
\end{theorem}
\begin{proof}
By the Poincar\'{e}--Lelong formula, we have
$$
-2{\rm dd^{c}}\log\left\|s_{D}\circ f(z)\right\|=-\sum\limits_{w\in\BD_{R}}({\rm ord}_{w}f^{*}D)\delta_{w}+f^{*}c_{1}(\CO_{X}(D),\left\|\cdot\right\|),
$$
where $\delta_{w}$ is the Dirac current supported on $w$. Integrating over $\BD_{r}$, we get
$$
-2\int_{\BD_{r}}{\rm dd^{c}}\log\left\|s_{D}\circ f(z)\right\|=-\sum\limits_{w\in\BD_{r}}({\rm ord}_{w}f^{*}D)+\int_{\BD_{r}}f^{*}c_{1}(\CO_{X}(D),\left\|\cdot\right\|).
$$
Hence, we get
\begin{align*}
-N(r,f,D)+T(r,f,D,\left\|\cdot\right\|)&=-2\int_{1}^{r}\frac{{\rm d}t}{t}\int_{\BD_{t}}{\rm dd^{c}}\log\left\|s_{D}\circ f(z)\right\|\\
&= -2\int_{1}^{r}\frac{{\rm d}t}{t}\int_{\partial\BD_{t}}{\rm d^{c}}\log\left\|s_{D}\circ f(z)\right\|\\
&=-2\cdot\frac{1}{2}\int_{1}^{r}\frac{{\rm d}t}{t}\int_{0}^{2\pi}t\frac{\partial}{\partial t}\log\left\|s_{D}\circ f(t{\rm e}^{{\rm i}\theta})\right\|\frac{{\rm d}\theta}{2\pi}\\
&=m(r,f,D,\left\|\cdot\right\|)-m(1,f,D,\left\|\cdot\right\|).
\end{align*}
\end{proof}
The next lemma considers the limit behaviour of the three functions, which will be used in our proof in \S\ref{5.4}.
\begin{lemma}\label{limits}
$X,L,\left\|\cdot\right\|,D$ as above. Let $\{R_{n}\}_{n\geq1}$ be an increasing sequence of positive real numbers or $+\infty$. Let $R_{\infty}$ be the limit of $R_{n}$. Let $f_{n}\colon \BD_{R_{n}}\to X\ (1\leq n\leq\infty)$ be a sequence of holomorphic maps. Here we set $\BD_{+\infty}=\CC$. Assume that $f_{n}$ converges uniformly to $f_{\infty}$ on any compact subset of $\BD_{R_{\infty}}$. Then for any $0<r<R_{\infty}$, $T(r,f_{n},L,\left\|\cdot\right\|)$ tends to $T(r,f_{\infty},L,\left\|\cdot\right\|)$ as $n$ tends to infinity. Moreover, if $f_{\infty}(z)\notin D$ for all $z\in\partial\BD_{r}$, we have $($when $n$ tends to infinity$)$
$$
N(r,f_{n},D),\ m(r,f_{n},D,\left\|\cdot\right\|)
$$
tends to
$$
N(r,f_{\infty},D),\ m(r,f_{\infty},D,\left\|\cdot\right\|)
$$
respectively.

Furthermore, if $D$ is replaced by a closed subscheme $Z\,($of codimension possibly greater than $1)$ in $X$, then for any $0<r<R_{\infty}$, we have $N(r,f_{n},Z)\leq 2N(r,f_{\infty},Z)$ for $n$ large enough $($depending on $r)$.
\end{lemma}
\begin{proof}
The statement for the characteristic function can be deduced directly by definition and the uniform convergency of $f_{n}$ on $\ol{\BD}_{r}$. We mainly consider the other two functions.

We first consider the proximity function.

Since $f_{\infty}(z)\notin D$ for all $z\in\partial\BD_{r}$, we can choose an open neighbourhood $U$ of\\ $f_{\infty}(\partial\BD_{r})$ whose closure $\ol{U}$ (in the Euclidean topology) is disjoint from $D$. Then the function $\log\left\|s_{D}(x)\right\|$ is a real-valued continuous function on $\ol{U}$. By the compactness of $\ol{U}$, it is uniformly continuous. That is, if we fix a distance function $d\colon \ol{U}\times\ol{U}\to\RR_{\geq0}$ on $\ol{U}$, then for any $\varepsilon>0$, there is a $\delta>0$ such that for all $x,y\in\ol{U}$ with $d(x,x')<\delta$, we have $|\log\left\|s_{D}(x)\right\|-\log\left\|s_{D}(x')\right\||<\varepsilon$.

As $f_{n}$ converges uniformly to $f_{\infty}$ on any compact subset of $\BD_{R_{\infty}}$, for any $\delta>0$, when $n$ large enough (depending on $\delta$), we have $f_{n}(y)\in U$ and $d(f_{n}(y),f_{\infty}(y))<\delta$ for all $y\in\partial\BD_{r}$. Thus if we choose $\delta$ as the one in the previous paragraph, we have $|\log\left\|s_{D}(f_{n}(y))\right\|-\log\left\|s_{D}(f_{\infty}(y))\right\||<\varepsilon$ for $n$ large enough (depending on $\delta$) and any $y\in\partial\BD_{r}$. Now by definition of the proximity function, we see that $|m(r,f_{n},D,\left\|\cdot\right\|)-m(r,f_{\infty},D,\left\|\cdot\right\|)|<\varepsilon$ for $n$ large enough. Since $\varepsilon>0$ is arbitrary, $m(r,f_{n},D,\left\|\cdot\right\|)$ tends to $m(r,f_{\infty},D,\left\|\cdot\right\|)$ as $n\to\infty$.

Next we prove the lemma for the counting function $N$. By definition we have
\begin{align*}
N(r,f_{\infty},D)&=\int_{1}^{r}(\sum\limits_{z\in\BD_{t}}{\rm ord}_{z}f^{*}D)\frac{{\rm d}t}{t}\\
&=\sum\limits_{z\in\BD_{1}}{\rm ord}_{z}f^{*}D\cdot\log r+\sum\limits_{1\leq|z|<r}{\rm ord}_{z}f^{*}D\cdot(\log r-\log|z|)\\
&=\sum\limits_{z\in\BD_{r}}{\rm ord}_{z}f^{*}D\cdot\min\{\log r,\log r-\log|z|\}.
\end{align*}
Fix an $\varepsilon>0$. For any $z\in\ol{\BD}_{r}$, choose an open neighbourhood $U_{z}$ of $z$ satisfying the following conditions:
\begin{enumerate}[label=(\arabic*)]
 \item\label{limit1}  For any $z\in\ol{\BD}_{r}$ and $z'\in U_{z}\backslash\{z\}$. $f_{\infty}(z')\notin D$.
 \item\label{limit2} If $f_{\infty}(z)\in D$ (which implies that $z\in\BD_{r}$ by our assumption), $|z|>1$ and $z'\in U_{z}$, then $|z'|>1$ and
 $$
 (1-\varepsilon)(\log r-\log|z|)\leq\log r-\log|z'|\leq (1+\varepsilon)(\log r-\log|z|)
 $$
 \item\label{limit3} $U_{z}\subseteq\BD_{r}$ if $z\in\BD_{r}$.
\end{enumerate}
Now $U_{z}(z\in\ol{\BD}_{r})$ covers the compact set $\ol{\BD}_{r}$, so it has a finite sub-cover. Explicitly, assume $\ol{\BD}_{r}=\cup_{j=1}^{m}U_{z_{j}}$, where $m$ is a positive integer and $z_{j}\in\ol{\BD}_{r}$.

By Rouch\'{e}'s theorem, for $n$ large enough, we have $\sum_{z\in U_{z_{j}}}{\rm ord}_{z}f_{n}^{*}D=\sum_{z\in U_{z_{j}}}{\rm ord}_{z}f_{\infty}^{*}D$ for all $1\leq j\leq m$. By assumption \ref{limit1}, we have $\sum_{z\in U_{z_{j}}}{\rm ord}_{z}f_{\infty}^{*}D={\rm ord}_{z_{j}}f_{\infty}^{*}D$. Thus by assumption \ref{limit2},
\begin{align*}
\sum\limits_{z\in U_{z_{j}}}{\rm ord}_{z}f_{n}^{*}D\cdot\min\{\log r,\log r-\log|z|\}&\leq (1+\varepsilon)\sum\limits_{z\in U_{z_{j}}}{\rm ord}_{z}f_{n}^{*}D\cdot\min\{\log r,\log r-\log|z_{j}|\}\\
&=(1+\varepsilon){\rm ord}_{z_{j}}f_{\infty}^{*}D\cdot\min\{\log r,\log r-\log|z_{j}|\}
\end{align*}
for $n$ large enough. Similarly, we have
$$
\sum\limits_{z\in U_{z_{j}}}{\rm ord}_{z}f_{n}^{*}D\cdot\min\{\log r,\log r-\log|z|\}\geq (1-\varepsilon){\rm ord}_{z_{j}}f_{\infty}^{*}D\cdot\min\{\log r,\log r-\log|z_{j}|\}
$$
for $n$ large enough.

By Rouch\'{e}'s theorem again, we see that for any $1\leq j<k\leq m$ and any $n$ large enough, $f_{n}(z)\notin D$ for $z\in U_{z_{j}}\cap U_{z_{k}}$. In other words, when $n$ large enough, for any $z\in\BD_{r}$ satisfying $f_{n}(z)\in D$, there is exactly one $1\leq j\leq m$ such that $z\in U_{z_{j}}$. Now we have
\begin{align*}
N(r,f_{n},D)&=\sum\limits_{z\in \BD_{r}}{\rm ord}_{z}f_{n}^{*}D\cdot\min\{\log r,\log r-\log|z|\}\\
&=\sum\limits_{j=1}^{m}\sum\limits_{z\in U_{z_{j}}}{\rm ord}_{z}f_{n}^{*}D\cdot\min\{\log r,\log r-\log|z|\}\\
&\leq\sum\limits_{j=1}^{m}(1+\varepsilon){\rm ord}_{z_{j}}f_{\infty}^{*}D\cdot\min\{\log r,\log r-\log|z_{j}|\}\\
&=(1+\varepsilon)\sum\limits_{z\in\BD_{r}}{\rm ord}_{z}f_{\infty}^{*}D\cdot\min\{\log r,\log r-\log|z|\}\\
&=(1+\varepsilon)N(r,f_{\infty},D).
\end{align*}
for $n$ large enough. Similarly we have
$$
N(r,f_{n},D)\geq (1-\varepsilon)N(r,f_{\infty},D)
$$
for $n$ large enough. As $\varepsilon>0$ is arbitrary, we conclude that $N(r,f_{n},D)$ tends to $N(r,f_{\infty},D)$ as $n\to\infty$.

Finally, if $D$ is replaced by an arbitrary closed subscheme $Z$ of $X$, all of the arguments above also work, except for the equality $\sum_{z\in U_{z_{j}}}{\rm ord}_{z}f_{n}^{*}D=\sum_{z\in U_{z_{j}}}{\rm ord}_{z}f_{\infty}^{*}D$ replaced by the inequality $\sum_{z\in U_{z_{j}}}{\rm ord}_{z}f_{n}^{*}Z\leq\sum_{z\in U_{z_{j}}}{\rm ord}_{z}f_{\infty}^{*}Z$. In fact, take $Z=D_{1}\cap\cdots\cap D_{l}$ for some effective Cartier divisors $D_{1},\dots,D_{l}$. Then by the argument above we see that $\sum_{z\in U_{z_{j}}}{\rm ord}_{z}f_{n}^{*}D_{i}=\sum_{z\in U_{z_{j}}}{\rm ord}_{z}f_{\infty}^{*}D_{i}$ for all $1\leq i\leq l$ and $n$ large enough. Therefore we have
\begin{align*}
\sum_{z\in U_{z_{j}}}{\rm ord}_{z}f_{n}^{*}Z&=\sum_{z\in U_{z_{j}}}\min\{{\rm ord}_{z}f_{n}^{*}D_{1},\dots,{\rm ord}_{z}f_{n}^{*}D_{l}\}\\
&\leq \sum_{z\in U_{z_{j}}}{\rm ord}_{z}f_{n}^{*}D_{i}\\
&=\sum_{z\in U_{z_{j}}}{\rm ord}_{z}f_{\infty}^{*}D_{i}={\rm ord}_{z_{j}}f_{\infty}^{*}D_{i}.
\end{align*}
for all $1\leq i\leq l$ and $n$ large enough. This means that
$$
\sum_{z\in U_{z_{j}}}{\rm ord}_{z}f_{n}^{*}Z\leq\min\{{\rm ord}_{z_{j}}f_{\infty}^{*}D_{1},\dots,{\rm ord}_{z_{j}}f_{\infty}^{*}D_{l}\}={\rm ord}_{z_{j}}f_{\infty}^{*}Z=\sum_{z\in U_{z_{j}}}{\rm ord}_{z}f_{\infty}^{*}Z.
$$
for $n$ large enough. Now imitate the previous argument (choose $0<\varepsilon<1$ to ensure that $1+\varepsilon<2$), we can deduce that $N(r,f_{n},Z)\leq 2N(r,f_{\infty},Z)$ for $n$ large enough.
\end{proof}
The next theorem is due to Yamanoi, which is crucial to our proof in \S\ref{5.4}. Part (1) of the theorem is called the Second Main Theorem (for abelian varieties).
\begin{theorem}\label{Nevanlinna}
Let $A$ be an abelian variety over $\CC$ and let $L$ be an ample line bundle on $A$. Let $f\colon \CC\to A$ be a holomorphic curve such that the image of $f$ in $A$ is Zariski dense.

\begin{enumerate}[label=\normalfont(\arabic*),leftmargin=*]
\item {\rm (Second Main Theorem, \cite[Theorem 3.1.1]{zbMATH02144170})} If $D$ is a reduced effective divisor on $A$, then
$$
T(r,f,D)\leq N^{(1)}(r,f,D)+\varepsilon T(r,f,L)\quad\quad \big|\big|_{\varepsilon}\quad \forall \varepsilon>0.
$$
Here $\big|\big|$ means that the estimate holds for $r>0$ outside some exceptional set of finite Lebesgue measure. $T(r,f,D), T(r,f,L)$ are with respect to any smooth Hermitian metrics on $\CO_{A}(D), L$ respectively. Similarly for \emph{(2)} below.

\item {\rm (\cite[Corollary 2.5.8]{zbMATH02144170})} If $Z\subseteq A$ is a closed subscheme whose support has codimension greater than 1, then
$$
N(r,f,Z)\leq\varepsilon T(r,f,L)\quad\quad \big|\big|_{\varepsilon}\quad \forall \varepsilon>0.
$$
\end{enumerate}
\end{theorem}
Actually, we will use Theorem \ref{Nevanlinna} only for linear entire curves. In this case there is a simplified proof in \cite[pp. 255-257]{zbMATH06448080}.
\subsection{A key lemma}\label{5.2}
In this subsection we prove the following key lemma which will be used in \S\ref{5.5}.
\begin{lemma}[Lemma \ref{key1}]\label{key}
Let $S$ be a variety over $\CC$. Let $\pi\colon X\to S$ be a morphism of finite type between separated $\CC$-schemes. Fix a point $s_{\infty}\in S(\CC)$ and a sequence of points $s_{n}\in S(\CC)$ tending to $s_{\infty}$ in the Euclidean topology. Then there is a Zariski open dense subset $X_{s_{\infty}}^{\circ}\subset X_{s_{\infty}}$ such that for every point $x_{\infty}\in X_{s_{\infty}}^{\circ}(\CC)$ and every sequence of points $x_{n}\in X_{s_{n}}$ tending to $x_{\infty}$, if there exists a sequence of tangent vectors $v_{n}\in T_{x_{n}}(X_{s_{n}, {\rm red}})$ tending to a tangent vector $v_{\infty}\in T_{x_{\infty}}(X_{s_{\infty}})$ under the Euclidean topology of the tangent bundle $TX$, then we have $v_{\infty}\in T_{x_{\infty}}(X_{s_{\infty}, {\rm red}})$.
\end{lemma}

\begin{example}
Take $S=\BA^{2}$ and $X=V(xu+yv+u^{2}v^{2})\subseteq \BA^{4}$. The morphism $\pi\colon X\to S$ is given by $(x,y,u,v)\mapsto (x,y)$. Then the fibre of $\pi$ at the point $s_{\infty}=(0,0)$ is non-reduced and the tangent space $T_{x_{\infty}}(X_{s_{\infty}})$ for every point $x_{\infty}\in X_{s_{\infty}}$ has dimension 2. On the other hand, $X_{s_{\infty},{\rm red}}$ is isomorphic to ${\rm Spec}\ \CC[u,v]/(uv)$, so the tangent space $T_{x_{\infty}}(X_{s_{\infty}, {\rm red}})$ for any point $x_{\infty}\in X_{s_{\infty}}\backslash\{(0,0,0,0)\}$ has dimension 1, which is the same as the dimension of the fibre.

In general, for a point $(x,y,u,v)\in X(\CC)$, we have
$$
T_{(x,y,u,v)}(X_{(x,y)})\simeq\{(u^{(1)},v^{(1)})\in\BA^{2}\big| (x+2uv^{2})\cdot u^{(1)}+(y+2u^{2}v)\cdot v^{(1)}=0\}.
$$
 It is not hard to see that when $(x,y)\neq (0,0)$, the fibre $X_{(x,y)}$ is reduced. Now fix a sequence of points $s_{n}=(x_{n}, y_{n})\to s_{\infty}=(0,0)$. For convenience we assume that $x_{n}y_{n}\neq0$ for all $n$. Then for any point on $X_{s_{\infty}}$ of the form $(0,0,0, v_{\infty})$ ($v_{\infty}\neq 0$), suppose there is a sequence of points $(x_{n}, y_{n}, u_{n}, v_{n})$ tending to $(0,0,0,v_{\infty})$ and tangent vectors $(u_{n}^{(1)},v_{n}^{(1)})\in T_{(x_{n},y_{n},u_{n},v_{n})}(X_{s_{n}})$ tending to $(u_{\infty}^{(1)},v_{\infty}^{(1)})\in T_{(0,0,0,v_{\infty})}(X_{s_{\infty}})$, then we have

 \begin{align}\label{lizi}
 (x_{n}+2u_{n}v_{n}^{2})\cdot u_{n}^{(1)}+(y_{n}+2u_{n}^{2}v_{n})\cdot v_{n}^{(1)}=0,\ n\geq1
 \end{align}
 Moreover, since $v_{n}\to v_{\infty}\neq0$ and
 \begin{align}\label{lizi1}
 x_{n}u_{n}+y_{n}v_{n}+u_{n}^{2}v_{n}^{2}=0,\ n\geq1
 \end{align}
 Divide both sides of the equation by $u_{n}$ and then take $n\to\infty$, we see that\\
$\lim_{n\to\infty}(y_{n}/u_{n})=0$. This means that $y_{n}+2u_{n}^{2}v_{n}=o(2u_{n}v_{n}^{2})$ as $n\to\infty$. Here for two sequences $\{a_{n}\}_{n\geq1}, \{b_{n}\}_{n\geq1}$ of complex numbers, we say $a_{n}=o(b_{n})$ as $n\to\infty$ if $\lim_{n\to\infty}(a_{n}/b_{n})=0$.

 After taking a subsequence we may assume that the limit $A=\lim_{n\to\infty}(x_{n}/2u_{n}v_{n}^{2})$ exists. If $A\neq-1$, divide both sides of the equation (\ref{lizi}) by $2u_{n}v_{n}^{2}$ and then take $n\to\infty$, we get that
 $$
 (A+1)u_{\infty}^{(1)}=0
 $$
 Therefore $u_{\infty}^{(1)}=0$. This is exactly the equation of the tangent space $T_{(0,0,0,v_{\infty})}(X_{s_{\infty},{\rm red}})$.

 If $A=-1$, then we have $x_{n}=-2u_{n}v_{n}^{2}+o(u_{n})$ as $n\to\infty$. Substitute this to the equation (\ref{lizi1}), we see that $y_{n}=u_{n}^{2}v_{n}+o(u_{n}^{2})$. Thus we get that
 $$
 \lim\limits_{n\to\infty}\frac{x_{n}^{2}}{y_{n}}=\lim\limits_{n\to\infty}\frac{4u_{n}^{2}v_{n}^{4}+o(u_{n}^{2})}{u_{n}^{2}v_{n}+o(u_{n}^{2})}=4v_{\infty}^{3}.
 $$
 Note that we have fixed $(x_{n},y_{n})$. Therefore, as long as $v_{\infty}$ doesn't satisfy the equation above, there must be $A\neq-1$ and $(u_{\infty}^{(1)},v_{\infty}^{(1)})\in T_{(0,0,0,v_{\infty})}(X_{s_{\infty},{\rm red}})$.
\end{example}
\medskip

\noindent\textbf{Proof of Lemma \ref{key}}\medskip

\noindent Now we begin to prove Lemma \ref{key}. If $X_{s_{n}}$ is empty for all but finitely many $n\geq1$, the lemma is trivial. Thus after taking a subsequence of $\{s_{n}\}_{n\geq1}$, we assume $X_{s_{n}}$ is nonempty for all $n\geq1$. Replacing $S$ by an irreducible component of Zariski closure of $\cup_{n\geq1}\{s_{n}\}$ in $S$ if necessary, we may assume that the points $\{s_{n}\}_{n\geq1}$ are Zariski dense in $S$. This implies that $\pi\colon X\to S$ is dominant. Replacing $X$ by $X_{\rm red}$, we may further assume that $X$ is reduced. Now by the generic reducedness (cf. \cite[\href{https://stacks.math.columbia.edu/tag/054Z}{Lemma 054Z}, \href{https://stacks.math.columbia.edu/tag/020I}{Lemma 020I}, \href{https://stacks.math.columbia.edu/tag/0578}{Lemma 0578}]{stacks-project}), there exists a Zariski open dense subset $S_{0}$ of $S$ such that $X_{S_{0}}\to S_{0}$ has reduced fibres. Thus after taking a subsequence of $\{s_{n}\}_{n\geq1}$, we may assume that $s_{n}\in S_{0}$ for all $n\geq1$. In this case $X_{s_{n}}$ is reduced for all $n\geq1$. \medskip

\noindent\textbf{Step 1.} The first step is to reduce to the case when $\dim S=1$. Let $TX$ be the tangent bundle of $X$ and $p\colon TX\to X$ be the natural projection. Let $TS$ be the tangent bundle of $S$ with a zero section $0_{S}\colon S\to TS$. Then the morphism $\pi\colon X\to S$ induces a morphism of the tangent bundles ${\rm d}\pi\colon TX\to TS$. Let $Y$ be the inverse image of $0_{S}(S)\subseteq S$ under ${\rm d}\pi$. In other words,
$$
Y(\CC)\colonequals \{(x,v)\in TX\big| x\in X(\CC), v\in T_{x}(X_{\pi(x)})\}.
$$
By composing $p\big|_{Y}\colon Y\to X$ and $\pi\colon X\to S$, we get a (dominant) morphism $\pi\circ(p\big|_{Y})\colon Y\to S$. By the Raynaud--Gruson flattening theorem in \cite[Th\'{e}or\`{e}me 5.2.2]{zbMATH03360302}, there exists a proper birational morphism $g\colon S'\to S$ such that the strict transform $Y'$ of $Y$ under $g$ is flat over $S'$. Let $g_{Y}\colon Y'\to Y$ be the natural morphism.

Since $\{s_{n}\}$ is Zariski dense in $S$ and $g$ is birational, after taking a subsequence, we may assume that all $s_{n}$ lie in an open subset $U\subseteq S$ such that $g^{-1}(U)\to U$ and $Y'_{g^{-1}(U)}\to Y_{U}$ are isomorphisms. We set $s'_{n}=g^{-1}(s_{n})(n\geq1)\in S'$. Since $g$ is proper, the inverse image of any compact subset of $S(\CC)$ in $S'(\CC)$ is compact. This means that after taking a subsequence, we may assume that $s'_{n}$ converging to a point $s'_{\infty}\in S'(\CC)$ in the Euclidean topology.  In particular, we have $g(s'_{\infty})=s_{\infty}$.

Now for any point $y'_{\infty}\in Y'_{s'_{\infty}}(\CC)$ and any Euclidean open neighbourhood $U_{y'_{\infty}}$ of $y'_{\infty}$ in $Y'(\CC)$, by \cite[Theorem 2.12]{zbMATH03522318}, the flatness of $Y'\to S'$ implies that the image of $U_{y'_{\infty}}$ in $S'(\CC)$ is open. In particular, it contains an (Euclidean) open neighbourhood of $s'_{\infty}$ in $S'(\CC)$. This means that $y'_{\infty}$ is a limit point of the fibres $Y'_{s'_{n}}(\CC)(n\geq 1)$. In other words, there exists a sequence of points $y'_{n}\in Y'_{s'_{n}}(\CC)$ tending to $y'_{\infty}$ under Euclidean topology. Since $y'_{\infty}$ is an arbitrary point on $Y'_{s_{\infty}}(\CC)$, we have shown that every point of $Y'_{s'_{\infty}}(\CC)$ is a limit point of $Y'_{s'_{n}}(\CC)$. In short, we conclude that  \emph{$Y'_{s'_{\infty}}$ is exactly the `limit' of $\{Y'_{s'_{n}}\}_{n\geq1}$}.

For a point $y_{\infty}\in Y_{s_{\infty}}(\CC)$, first assume $y_{\infty}$ is of the form $\lim_{n\to\infty}y_{n}$ where $y_{n}\in Y_{s_{n}}(\CC)$. Then since for all $n\geq1$, $Y'_{s'_{n}}$ is isomorphic to $Y_{s_{n}}$ under $g_{Y}\colon Y'\to Y$, there is a unique $y'_{n}\in Y'_{s'_{n}}(\CC)$ corresponds to $y_{n}$ under $g_{Y}$. Using again the properness of $g_{Y}\colon Y'\to Y$, $\{y'_{n}\}_{n\geq1}$ has a limit point $y'_{\infty}\in Y'_{s'_{\infty}}(\CC)$. This means that $y_{\infty}=\lim_{n\to\infty}g_{Y}(y'_{n})=g_{Y}(\lim_{n\to\infty}y'_{n})=g_{Y}(y'_{\infty})$. In particular, $y_{\infty}$ lies in the image of $Y'_{s'_{\infty}}(\CC)$.

Conversely, assume that $y_{\infty}=g_{Y}(y'_{\infty})$ for some point $y'_{\infty}\in Y'_{s'_{\infty}}(\CC)$. Then by the argument above, $y'_{\infty}$ is of the form $\lim_{n\to\infty}y'_{n}$ for some $y'_{n}\in Y'_{s'_{n}}(\CC)$. Denote $y_{n}=g_{Y}(y'_{n})\in Y_{s_{n}}(\CC)$. We then have $y_{\infty}=g_{Y}(y'_{\infty})=g_{Y}(\lim_{n\to\infty}y'_{n})=\lim_{n\to\infty}g_{Y}(y'_{n})=\lim_{n\to\infty}y_{n}$. Combine this with the previous paragraph, we conclude that \emph{$g_{Y}(Y'_{s'_{\infty}})$ is exactly the limit of $\{Y_{s_{n}}\}_{n\geq1}$}.

In particular, the limit of $\{Y_{s_{n}}\}_{n\geq1}$ depends only on the limit point $s'_{\infty}=\lim_{n\to\infty}g^{-1}(s_{n})$! Now choose any curve $C'\subseteq S'$ passing through $s'_{\infty}$ which does not lie in the exceptional locus of $g\colon S'\to S$. Denote $C=g(C')$, then $C$ is a curve in $S$ passing through $s_{\infty}$. For any sequence of points $t_{n}\in S(\CC)$ with $t'_{n}=g^{-1}(t_{n})$ tending to $s'_{\infty}$, the argument above shows that \emph{$g_{Y}(Y'_{s'_{\infty}})$ is exactly the limit of $\{Y_{t_{n}}\}_{n\geq1}$}. Therefore, the limit of $\{Y_{s_{n}}\}_{n\geq1}$ and the limit of $\{Y_{t_{n}}\}_{n\geq1}$ coincide. Restrict $\pi\colon X\to S$ to $\pi_{C}\colon X_{C}\to C$, we then reduce the lemma to the case when $\dim S=1$.\medskip

\noindent\textbf{Step 2.} Now we assume $\dim S=1$. The next step is to reduce to the case when $X,S$ are normal.

In fact, Let $\widetilde{X},\widetilde{S}$ be the normalization of $X, S$ and let $\widetilde\pi\colon \widetilde{X}\to\widetilde{S}$ be the morphism induced from $\pi\colon X\to S$. Let $Z'\subseteq\widetilde{X}$ be the non-isomorphic locus of $p\colon \widetilde{X}\to X$. Denote by $Z\subseteq p(Z') $ the union of the irreducible components of $p(Z')$ which does not contained in the fibre $X_{s_{\infty}}$. Take $U_{1}$ to be the complement of $Z$ in $X_{s_{\infty}}$. Then since $\dim S=1$, $U_{1}$ is Zariski dense in $X_{s_{\infty}}$. Since $\{s_{n}\}$ is dense in $S$, we may assume $s_{n}$ is a smooth point in $S$ for all $n$. Take $s'_{n}$ to be the (unique) inverse image of $s_{n}$ in $\widetilde{S}$. After taking a subsequence, we assume that $s'_{n}$ converges to a point $s'_{\infty}\in \widetilde{S}$ in the Euclidean topology.

Now suppose that the lemma holds for $\widetilde\pi\colon \widetilde{X}\to\widetilde{S}$ and sequence $\{s_{n}'\}_{n\geq1}\subseteq \widetilde{S}(\CC)$. This means that there exists a Zariski open dense  subset $\widetilde{X}^{\circ}_{s'_{\infty}}\subseteq \widetilde{X}_{s'_{\infty}}$ satisfying the conditions in the lemma. Shrinking $\widetilde{X}^{\circ}_{s'_{\infty}}$ if necessary, we assume $p^{-1}(p(\widetilde{X}^{\circ}_{s'_{\infty}}))=\widetilde{X}^{\circ}_{s'_{\infty}}$. Then $p(\widetilde{X}^{\circ}_{s'_{\infty}})$ contains a Zariski dense open subset $U_{2}$ of $X_{s_{\infty}}$. Take $X^{\circ}_{s_{\infty}}=U_{1}\cap U_{2}$. Then $X^{\circ}_{s_{\infty}}$ is Zariski dense and open in $X_{s_{\infty}}$. For any point $x_{\infty}\in X^{\circ}_{s_{\infty}}(\CC)$, and any sequence $x_{n}\in X_{s_{n}}(\CC)$ tending to $x_{\infty}$, by the definition of $U_{1}$ we see that $x_{n}\notin p(Z)$ for $n$ sufficiently large. Thus for large $n$, we take $x'_{n}\in X'_{s'_{n}}(\CC)$ to be the (unique) inverse image of $x_{n}$ in $X'$. Note that $p\colon \widetilde{X}\to X$ is an isomorphism outside $Z'$, so we have an isomorphism ${\rm d}p\colon T_{x'_{n}}(\widetilde{X}_{s'_{n}})\to T_{x_{n}}(X_{s_{n}})$. For any sequence of tangent vectors $v_{n}\in T_{x_{n}}(X_{s_{n}})$ tending to $v_{\infty}\in T_{x_{\infty}}(X_{s_{\infty}})$, let $v'_{n}\in T_{x'_{n}}(\widetilde{X}_{s'_{n}})$ be the (unique) inverse image of $v_{n}$ under ${\rm d}p$. We may assume that $v_{\infty}\neq0$, otherwise the lemma is trivial. Then we can assume $v_{n}\neq0$ for all $n\geq1$. Let $ \ol{v}_{n},\ol{v}'_{n}$ be the image of $v_{n},v'_{n}$ in $\BP T_{x_{n}}X,\BP T_{x'_{n}}\widetilde{X}$ respectively. As $\BP T\widetilde{X}$ is compact, we may assume that $\{(x'_{n},\ol{v}'_{n})\}_{n\geq1}$ tending to an element $(x'_{\infty},\ol{v}'_{\infty})\in \BP T\widetilde{X}$. Choose some real numbers $a_{n}$ such that $(x'_{n},a_{n}v'_{n})$ has a subsequence tending to an element $(x'_{\infty},v'_{\infty})\in T\widetilde{X}$ with $v'_{\infty}\neq0$. This implies that $(x_{n},a_{n}v_{n})$ has a subsequence tending to $(x_{\infty},({\rm d}p)(v'_{\infty}))\in TX$. Therefore $\ol{v}_{n}$ tends to the image of $\ol{v}'_{\infty}$ in $\BP T_{x_{\infty}}X$, which means that $v_{\infty}$ and $({\rm d}p)(v'_{\infty})$ are proportional. Write $v_{\infty}=c\cdot ({\rm d}p)(v'_{\infty})$ for some $c\in\RR$.

  By our assumption, $x'_{\infty}\in \widetilde{X}^{\circ}_{s'_{\infty}}$ and therefore $v'_{\infty}\in T_{x'_{\infty}}(\widetilde{X}_{s'_{\infty},{\rm red}})$. Thus $v_{\infty}=c\cdot({\rm d}p)(v'_{\infty})\in T_{x_{\infty}}(X_{s_{\infty},{\rm red}})$. This implies that the lemma also holds for $\pi\colon X\to S$ and sequence $\{s_{n}\}_{n\geq1}$. In conclusion, we can assume $X, S$ are normal.\medskip

\noindent\textbf{Step 3.}
Now we assume $\dim S=1$ and $X, S$ are normal. In particular $S$ is regular. We also assume that $X$ is irreducible (hence integral). Now since $\pi\colon X\to S$ is dominant, $\pi$ is flat. By \cite[\href{https://stacks.math.columbia.edu/tag/0BRQ}{Lemma 0BRQ}]{stacks-project}, there exists a finite surjective morphism $T\to S$ such that in the diagram

\begin{displaymath}
  \xymatrix{Y \ar[r]_{\nu} \ar[dr]_{g} & X\times_S T \ar[r] \ar[d] & X \ar[d]_{\pi}\\  \ar@{}[r]|{\;}& T \ar[r] & S }
\end{displaymath}
the morphism $g$ is smooth at all generic points of fibres. Here $\nu$ is the normalization map. In particular, each fibre of $g$ is generically reduced.

Take $t_{n}\in T(\CC)$ whose image in $S$ is $s_{n}$. Since $T\to S$ is finite, after taking a subsequence of $\{t_{n}\}_{n\geq1}$ we may assume that $t_{n}$ tends to a point $t_{\infty}$ in the Euclidean topology. In particular $t_{\infty}$ maps to $s_{\infty}$ under $T\to S$. Let $Y^{\circ}_{t_{\infty}}\subseteq Y_{t_{\infty}}$ be a Zariski open dense subset of $Y_{t_{\infty}}$ which is reduced. Then for any point $y\in Y_{t_{\infty}}^{\circ}(\CC)$, we have $T_{y}( Y_{t_{\infty},{\rm red}})=T_{y}(Y_{t_{\infty}})$. This implies that the lemma holds for $g\colon Y\to T$ and sequence $\{t_{n}\}_{n\geq1}$. By the argument in \textbf{Step 2}, we see that the lemma also holds for $X\times_{S}T\to T$ and sequence $\{t_{n}\}_{n\geq1}$. Denote $X'=X\times_{S} T$ and take $X'^{\circ}_{t_{\infty}}\subseteq X'_{t_{\infty}}$ to be a Zariski open dense subset of $X'_{t_{\infty}}$ satisfying the lemma. Let $X_{s_{\infty}}^{\circ}$ be the image of $X'^{\circ}_{t_{\infty}}$ in $X_{s_{\infty}}$ under the natural isomorphism $X'_{t_{\infty}}\to X_{s_{\infty}}$.

Now for any sequence of vectors $\{(x_{n},v_{n})\in TX_{s_{n}}\}_{n\geq1}$ tending to a tangent vector $(x_{\infty},v_{\infty})\in TX_{s_{\infty}}$, there is a unique sequence of vectors $\{(x_{n}',v_{n}')\in TX'_{t_{n}}\}_{n\geq1}$ corresponding to $\{(x_{n},v_{n})\in TX_{s_{n}}\}_{n\geq1}$ under the natural identification $X'_{t_{n}}\simeq X_{s_{n}} (n\geq1)$. We may assume that $v_{\infty}\neq0$, otherwise the lemma is trivial. Then we can assume that $v_{n}\neq0$ for all $n\geq1$. Let $\ol{v}_{n},\ol{v}'_{n}$ be the image of $v_{n},v'_{n}$ in $\BP T_{x'_{n}}X',\BP T_{x_{n}}X$, respectively. Since $\BP TX'$ is compact, after taking a subsequence we can assume that $(x_{n}',\ol{v}'_{n})$ tending to an element $(x'_{\infty},\ol{v}'_{\infty})\in \BP T_{t_{\infty}}(X'_{t_{\infty}})$. This element maps to $(x_{\infty},\ol{v}_{\infty})$ under the tangent map. Choose some suitable real numbers $a_{n}$ such that $\{(x'_{n},a_{n}v'_{n})\}_{n\geq1}$ has a subsequence tending to an element $(x'_{\infty},v'_{\infty})\in TX'$ with $v'_{\infty}\neq0$. This implies that $(x_{n},a_{n}v_{n})$ has a subsequence tending to the image of $(x'_{\infty},v'_{\infty})$ in $TX$. Therefore $\ol{v}_{n}$ tends to the image of $\ol{v}'_{\infty}$ in $\BP T_{x_{\infty}}X$, which means that $v_{\infty}$ and the image of $v'_{\infty}$ in $T_{x_{\infty}}X$ are proportional.

 Thus if $x_{\infty}\in X^{\circ}_{s_{\infty}}$, we have $x'_{\infty}\in X'^{\circ}_{t_{\infty}}$ and therefore $v_{\infty}'\in T_{x_{\infty}'}(X'_{t_{\infty},{\rm red}})$ by the choice of $X'^{\circ}_{t_{\infty}}$. Hence $v_{\infty}\in {\rm Im}(T_{x'_{\infty}}(X'_{t_{\infty},{\rm red}})\to T_{x_{\infty}}(X_{s_{\infty},{\rm red}}))\subseteq T_{x_{\infty}}(X_{s_{\infty},{\rm red}})$.  This finishes the proof of the lemma. \qed

\begin{remark}\label{zilie}
From the Step 1 of the proof, we actually get the following fact: There is a subsequence $\{s_{n_{k}}\}_{k\geq1}$ of $\{s_{n}\}_{n\geq1}$ such that for any limit point $(x_{\infty},v_{\infty})$ of $\{T(X_{s_{n_{k}},{\rm red}})\}_{n\geq1}$ in $T(X_{x_{\infty}})$, $(x_{\infty},v_{\infty})$ is of the form

$$
(x_{\infty},v_{\infty})=\lim\limits_{k\to\infty}(x_{n_{k}},v_{n_{k}}),\ {\rm where}\  (x_{n_{k}},v_{n_{k}})\in T(X_{s_{n_{k}},{\rm red}}).
$$
In other words, in the equation expressing $(x_{\infty},v_{\infty})$ as a limit, we do not need to take a subsequence of $\{s_{n_{k}}\}_{k\geq1}$ anymore.

In fact, as in the first sentence in the proof, we assume that $X$ is reduced and $\{s_{n}\}_{n\geq}$ is Zariski dense in $S$.  Let $g\colon S'\to S$ be the same birational morphism as in Step 1. Let $U\subseteq S$ be the Zariski open dense subset of $S$ such that $g^{-1}(U)\to U$ is an isomorphism. Take $\{s_{n_{k}}\}_{k\geq1}$ to be a subsequence of $\{s_{n}\}_{n\geq1}$ such that $s_{n_{k}}\in U$ for all $k\geq1$ and the inverse image $s'_{n_{k}}$ of $s_{n_{k}}$ in $S'$ tends to a point $s'_{\infty}$ when $k\to\infty$. This is what the second paragraph of Step 1 did. By the argument in Step 1, we see that this subsequence $\{s_{n_{k}}\}_{k\geq1}$ satisfies the condition above.
 \end{remark}
\subsection{Step 1: algebraic-geometric reductions}\label{5.3}
Before proving Proposition \ref{generic1}, we first prove a simple lemma concerning the positivity of divisors on abelian varieties.

Recall that for a generically finite proper map $f\colon X\to Y$ between $K$-varieties, there is a push-forward homomorphism $f_{*}\colon {\rm Div}(X)\to {\rm Div}(Y)$ between the groups of Weil divisors. More precisely, for any prime Weil divisor $D$ of $X$, the image $D'=f(D)$ is a closed subvariety of $Y$. If $\dim D'<\dim D$, set $f_{*}(D)=0$. If $\dim D'=\dim D$, then the function field $K(D)$ of $D$ is a finite extension of the function field $K(D')$ of $D'$. Set $f_{*}(D)=[K(D):K(D')]\cdot D'$. This extends linearly to a homomorphism $f_{*}\colon {\rm Div}(X)\to {\rm Div}(Y)$. By \cite[Theorem 1.4]{zbMATH01027930}, $f_{*}$ descends to a homomorphism $f_{*}\colon{\rm Cl}(X)\to{\rm Cl}(Y)$ between Weil divisor class groups. Composing by the natural homomorphism ${\rm Cacl}(X)\to{\rm Cl}(X)$ from the Cartier divisor class group, we can define push-forward for all Cartier divisors.
\begin{lemma}\label{pos}
Let $X,A$ be two projective varieties over a field $K$ with $A$ an abelian variety. Let $f\colon X\to A$ be a generically finite surjective $K$-morphism. Let $R$ be a big Cartier divisor on $X$. Then $f_{*}(R)$ is an ample Cartier divisor on $A$.
\end{lemma}
\begin{proof}
As $A$ is smooth, all Weil divisors on $A$ is Cartier. In particular, $f_{*}(R)$ is a Cartier divisor on $A$. It remains to show the ampleness. By \cite[Proposition 1.4]{zbMATH00785040}, it suffices to show that for any curve $C\subseteq A$, $f_{*}(R)\cdot C>0$. We first assume $R$ is ample and effective. After a suitable translation, we may assume that $f(R)$ and $C$ intersect properly. Then $C$ does not intersect any irreducible component of $f(R)$ of codimension greater than 1. If $f(R)\cap C=\emptyset$, for any curve $C'\subseteq f^{-1}(C)$, we have $R\cap C'=\emptyset$, which contradicts to the ampleness of $R$. Therefore $C\cdot f_{*}(R)>0$. This proves the case when $R$ is ample and effective.

For general $R$, there exists a positive integer $m$ such that $mR$ is linearly equivalent to $R_{1}+R_{2}$, where $R_{1}$ is ample and effective and $R_{2}$ is effective. By the argument above, $f_{*}(R_{1})$ is ample on $A$. Moreover, $f_{*}(R_{2})$ is effective by definition, so it is also nef by the group structure of $A$ (cf. \cite[Example 1.4.7]{zbMATH02134816}). Therefore $f_{*}(R_{1})+f_{*}(R_{2})$ is ample on $A$. By \cite[Theorem 1.4]{zbMATH01027930}, $f_{*}(R)$ is linearly equivalent to $f_{*}(R_{1})+f_{*}(R_{2})$, and thus it is also ample.
\end{proof}
\begin{remark}\label{redample}
Since any effective divisor on an abelian variety is nef (cf. \cite[Example 1.4.7]{zbMATH02134816}), we can easily prove that an effective divisor $D$ on an abelian variety is ample if and only if $D_{\rm red}$ is ample. Moreover, if $D_{1},D_{2}$ are two effective Cartier divisors on an abelian variety with ${\rm Supp}(D_{1})\subseteq{\rm Supp}(D_{2})$ and $D_{1}$ ample, then $D_{2}$ is also ample. This will be used in \S\ref{5.4}.
\end{remark}
\medskip\medskip

\noindent\textbf{Basic setups}\medskip

\noindent Now we begin to prove Proposition \ref{generic1} by contradiction. After taking normalization, we may assume that $X$ is normal. Assume to the contrary that the height of $\{x_{n}\}_{n\geq1}$ is unbounded. After taking a subsequence, we can assume that the height of $\{x_{n}\}_{n\geq1}$ tends to infinity as $n\to\infty$.

First we repeat the following operations: Take $A_{1}=A,A_{0}={\rm Spec}\, K$ at the beginning. Then we have a finite map $f\colon X\to A_{1}\times A_{0}\times T_{K}$. In each step, if there is a nontrivial quotient abelian variety $A'$ of $A_{1}$ (i.e. a surjective homomorphism $A_{1}\to A'$) such that the height of the image of $\{x_{n}\}_{n\geq1}$ in $A'$ under the natural composition $X\to A_{1}\times A_{0}\times T_{K}\to A_{1}\to A'$ does not tend to infinity, then after taking a subsequence of $\{x_{n}\}_{n\geq1}$ we can assume that the image of $\{x_{n}\}_{n\geq1}$ in $A'$ has bounded height. Note that by our assumption in the previous paragraph, the height of $\{x_{n}\}_{n\geq1}$ still tends to infinity. Choose an isogeny $A_{1}\overset{j}\to{\rm ker}(A_{1}\to A')\times A'$, compose $f$ with the map $j\times{\rm id}\colon A_{1}\times A_{0}\times T_{K}\to{\rm ker}(A_{1}\to A')\times A'\times A_{0}\times T_{K}$, and change $A_{1},A_{0}$ to ${\rm ker}(A_{1}\to A'), A'\times A_{0}$ respectively, we get a finite map $X\to A_{1}\times A_{0}\times T_{K}$. This ends a step. Since the dimension of $A_{1}$ decreases strictly, the operations terminate after finitely many steps. Then we get a finite map $X\to A_{1}\times A_{0}\times T_{K}$ such that the image of $\{x_{n}\}_{n\geq1}$ in any nontrivial quotient abelian variety of $A_{1}$ has height tending to infinity, while the image of $\{x_{n}\}_{n\geq1}$ in $A_{0}(K)$ has bounded height. Note that by the assumption in Proposition \ref{generic1}, the image of $\{x_{n}\}_{n\geq1}$ in $T_{K}(K)$ is contained in $T(\CC)$, and thus also has bounded height. Combining these, we see that the abelian variety $A_{1}$ must be nontrivial.

By Lang--N\'{e}ron theorem (cf. \cite[Theorem 3.4]{arXiv:2305.14789}), $V(A_{0},K)=(A_{0}(K)/A_{0}^{(K/\CC)}(\CC))_{\RR}$ is a finite-dimensional $\RR$-vector space, and the canonical height on $A_{0}(K)$ (with respect to an ample line bundle on $A_{0}$) induces a positive definite quadratic form on $V(A_{0},K)$. As a consequence, there are only finitely many elements in $V(A_{0},K)$ with a fixed bounded height. Hence the image of $\{x_{n}\}_{n\geq1}$ in $A_{0}(K)$ is contained in a finite union of translations of $A_{0}^{(K/\CC)}(\CC)$. As $\{x_{n}\}_{n\geq1}$ is generic in $X$ and $X$ is irreducible, the image of $X$ in $A_{0}$ must be contained in a translation of $(A_{0}^{(K/\CC)})_{K}$. After composing $X\to A_{1}\times A_{0}\times T_{K}$ with a translation on $A_{0}$ (with identity on other components), we may assume that the image of $X$ in $A_{0}$ is contained in $(A_{0}^{(K/\CC)})_{K}$. Let $T'$ be the Zariski closure of the image of $\cup_{n\geq1}\{x_{n}\}$ in $A_{0}^{(K/\CC)}\times T$, which is a complex variety. Then the image of $X$ in $(A_{0}^{(K/\CC)})_{K}\times T_{K}$ is exactly $T'_{K}$. In conclusion, we get a finite morphism $X\to A_{1}\times_{K} T'_{K}$, whose composition with the projection $A_{1}\times_{K} T'_{K}\to T'_{K}$ is surjective.

From now on we focus on the finite morphism $X\to A_{1}\times_{K} T'_{K}$. For simplicity, we replace the original $A$ by $A_{1}$, replace the original $T$ by $T'$, and replace the original $f$ by this map. In other words, we call this map $f\colon X\to A\times_{K} T_{K}$. After taking normalization for $T$, we may assume that $T$ is normal. Let $s_{n}, t_{n}$ be the image of $x_{n}$ in $A(K),T_{K}(K)$ respectively, then $t_{n}\in T(\CC)$ is constant. Since $\{x_{n}\}_{n\geq1}$ is generic in $X$ and $X\to T_{K}$ is surjective, $\{t_{n}\}_{n\geq1}$ is also generic in $T$. Moreover, by our assumption, the image of $\{s_{n}\}_{n\geq1}$ in any nontrivial quotient abelian variety of $A$ has height tending to infinity.

Take a finite field extension of $K$ if necessary, we may assume ${\rm End}_{K}(A)={\rm End}_{\ol{K}}(A_{\ol{K}})$.

Take integral models $\CX,\CA$ of $X, A$ over $B$. Then $\CA\times_{\CC}T$ is an integral model of $A\times T_{K}$ over $B$. Taking a resolution of singularity of $\CA$, we may assume that $\CA$ is smooth. Replacing $\CX$ by the normalization of the Zariski closure of ${\rm Im}(X\to X\times_{K}(A\times_{K} T_{K})\to\CX\times_{B}(\CA\times_{\CC}T))$, we may assume that $\CX$ is normal and $f$ extends to a $B$-morphism $\CX\to\CA\times_{\CC}T$. For simplicity, we still denote this morphism by $f$. Recall that $(s_{n},t_{n})=f(x_{n})\in A(K)\times T_{K}(K)$. By taking Zariski closure in $\CX$, we can view $x_{n}$ as a section $x_{n}\colon B\to\CX$ and similarly for $s_{n},t_{n}$. \medskip

\noindent\textbf{Non-surjective case}\medskip

\noindent We first consider the case when $f\colon X\to A\times_{K}T_{K}$ is not surjective. In this case, we can still use the arguments in \cite{arXiv:2308.08117} combining with a flattening process.

Namely, denote $Z=f(X)\subseteq A\times_{K}T_{K}$ and let $\CZ$ be the Zariski closure of $Z$ in $T\times_{\CC}\CA$. Then $f(x_{n})=(s_{n},t_{n})\in Z(K)\subseteq A(K)\times T_{K}(K)$. By taking Zariski closure in $\CZ$, we can view $(s_{n},t_{n})$ as a section $(s_{n},t_{n})\colon B\to\CZ$.

Consider the composition $\CZ\hookrightarrow\CA\times_{\CC}T\to T$, which is surjective by our assumption. By Raynaud--Gruson's flattening theorem (cf. \cite[Th\'{e}om\`{e}me 5.2.2]{zbMATH03360302}), there is a proper birational map $g\colon T'\to T$ such that the strict transform $\CZ'$ of $\CZ$ under $h$ is flat over $T'$. Let $g_{\CZ}\colon\CZ'\to \CZ$ be the natural morphism. For each $n$, take a point $t'_{n}\in T'(\CC)$ such that $g(t'_{n})=t_{n}$. Then we can lift the section $(s_{n},t_{n})\colon B\to\CZ$ to a section $B\to\CZ'$. View $\CZ'$ as a closed subscheme of $\CA\times_{\CC}T'$, we can denote this section by $(s_{n},t'_{n})$. As $T'(\CC)$ is compact, after taking a subsequence, we can assume that $t'_{n}$ tends to a point $t'_{\infty}\in T'(\CC)$. The flatness of $\CZ'\to T'$ implies that all fibres $\CZ'_{t'}$ has the same dimension. Since $f$ is not surjective, $Z$ is a proper closed subset of $A\times_{K}T_{K}$ and $\CZ$ is a proper closed subset of $\CA\times_{\CC}T$. Therefore $\dim \CZ'=\dim \CZ<\dim (\CA\times_{\CC}T)=\dim(\CA\times_{\CC}T')$. This implies that $\CZ'_{t'}$ is a proper closed subset of $\CA\times\{t'\}$ for all $t'\in T'$. In particular, $\CZ'_{t'_{\infty}}$ is a proper closed subset of $\CA\times\{t'_{\infty}\}$. We then deduce that $\CZ'_{b,t'_{\infty}}$ is a proper closed subset of $\CA_{b}\times\{t'_{\infty}\}$ for all but finitely many $b\in B(\CC)$. Here $\CZ'_{b,t'_{\infty}}$ is the fibre over the point $(b,t'_{\infty})$ under the natural map $\CZ'\to B\times_{\CC} T'$.

Now by our assumption, the image of $\{s_{n}\}_{n\geq1}$ in any nontrivial quotient abelian variety of $A$ has height tending to infinity. Thus by Corollary \ref{location}, for $b\in B(\CC)$ outside a subset of zero Lebesgue measure, there is a Zariski dense subset $U_{b}\subseteq \CA_{b}$ such that each $y\in U_{b}$ is a limit point of the sections $\{s_{n}(B)\}_{n\geq1}$ (in the Euclidean topology). Take $b, U_{b}$ as above such that $\CZ'_{b,t'_{\infty}}\subsetneqq\CA_{b}\times\{t'_{\infty}\}$. Then each point in $U_{b}\times\{t'_{\infty}\}$ is a limit point of sections $\{s_{n}(B)\times\{t'_{n}\}\}_{n\geq1}$, which is contained in $\CZ'$. However, by our assumption $U_{b}\times\{t'_{\infty}\}$, which is Zariski dense in $\CA_{b}\times\{t'_{\infty}\}$, is not contained in $\CZ'_{b,t'_{\infty}}$, a contradiction!

\begin{remark}
The arguments of this non-surjective case essentially implies our main theorem when $X$ is a subvariety of abelian variety, which is known by \cite{zbMATH00721811,zbMATH03832101, zbMATH00016118, zbMATH00224044, zbMATH00957009}.
\end{remark}

\medskip\medskip

\noindent\textbf{Surjective case}\medskip

\noindent Now we assume that $f\colon X\to A\times_{K}T_{K}$ is finite surjective. Recall that we have assumed $X,T$ are both normal. Let $R$ be the ramification divisor of $f\colon X\to A\times_{K} T_{K}$ (see \S\ref{1.3} for the definition) and let $D=f_{*}(R)$ be the branch divisor. Denote by $\CR,\CD$ the Zariski closure of $R,D$ in $\CX,\CA\times_{\CC}T$, respectively.

For $b\in B,t\in T$, denote by $\CX_{b,t}$ the fibre of the natural morphism $\CX\to B\times T$ over $(b,t)$. Similarly we can define $\CR_{b,t},\CD_{b,t}$. We will use this notation frequently later.

\medskip\medskip

\noindent\textbf{Reduce to the case when $\CR$ doesn't contain any ``$T$-vertical'' component}  \medskip

\noindent Consider the composition ${\rm pr}_{T}\circ f\colon \CX\to\CA\times_{\CC}T\to T$. Since $\CX$ is normal, the generic fibre of ${\rm pr}_{T}\circ f$ is normal. Using Stein factorization, we can assume that all the fibres of ${\rm pr}_{T}\circ f$ are geometrically connected. Therefore the generic fibre of ${\rm pr}_{T}\circ f$ is geometrically irreducible.

By \cite[Theorem 0.3]{zbMATH01443405}, there is a generically finite projective surjective morphism $T'\to T$ with $T'$ smooth and a projective strict modification $\CX'\to\CX\times_{T}T'$ such that $\CX'$ is normal and all fibres of the morphism $\CX'\to T'$ are reduced. Moreover, $\CX'\to T'$ is equidimensional. Here a strict modification $\CX'\to\CX\times_{T}T'$ is a birational morphism from $\CX'$ to the Zariski closure of $T'\times_{T}\eta$ in $T'\times_{T}\CX$, where $\eta$ is the generic point of $\CX$. Note that although $T'\times_{T}\CX$ may be reducible, $\CX'$ is irreducible.

The surjectivity of $T'\to T$ means that $x_{n}\colon B\to\CX$ can be lifted (not necessarily unique) to $B\to\CX\times_{T}T'$ for all $n\geq1$. By the genericness of $\{x_{n}\}_{n\geq1}$, after taking a subsequence we may assume that $x_{n}$ can be lifted to $x'_{n}\colon B\to\CX'$ for all $n\geq1$.

Now we have a generically finite morphism $\CX'\to\CA\times_{\CC}T'$, which we denote by $f'$. Let $X'$ be the generic fibre of $\CX'\to B$ and let $R'$ be the ramification divisor of the morphism $X'\to A\times_{K}T'_{K} $. Denote by $\CR'$ the Zariski closure of $R'$ in $\CX'$. Since $\CX'\to T'$ is equidimensional, for any $t'\in T'(\CC)$, $\CX'_{t'}$ and $\CA\times\{t'\}$ have the same dimension. Consider the morphism $f'_{t'}\colon \CX'_{t'}\to\CA\times\{t'\}$ induced by $f'$. For any irreducible component $Z$ of $\CX'_{t'}$, either $f'(Z)$ is a proper closed subset of $\CA\times\{t'\}$, or $Z\to\CA\times\{t'\}$ is generically finite. In the latter case, as $\CX'_{t'}$ is reduced, $Z\to\CA\times\{t'\}$ is generically \'{e}tale. In conclusion, we see that the image of the non-\'{e}tale locus of $f'_{t'}\colon \CX'_{t'}\to\CA\times\{t'\}$ in $\CA\times\{t'\}$ is a proper closed subset of $\CA\times\{t'\}$. Moreover, for any point $x'$ in the \'{e}tale locus of $f'_{t'}$, the tangent map $T_{x'}\CX'_{t'}\to T_{f'(x')}(\CA\times\{t'\})$ is an isomorphism. This implies that the tangent map $T_{x'}\CX'\to T_{f'(x')}(\CA\times_{\CC} T')$ is injective. In particular, $x'\notin\CR'$. Furthermore, since
$$
\dim\CX'\leq\dim_{\CC} T_{x'}\CX'\leq\dim_{\CC} T_{f'(x')}(\CA\times_{\CC} T')=\dim \CA\times_{\CC} T'=\dim\CX'
$$
we have $\dim\CX'=\dim_{\CC} T_{x'}\CX'$, and thus $x'$ lies in the smooth locus of $\CX'$. We conclude that the \'{e}tale locus of $f'_{t'}$ is contained in the smooth locus of $\CX'$.

Let $T_{1}\to T'$ be a finite morphism with $T_{1}$ smooth and (Brody) hyperbolic. In particular, $T_{1}$ does not contain rational curves. The existence of $T_{1}$ is essentially due to Hironaka; see \cite[Lemma 2.1]{arXiv:2307.16223} for a proof (The author only proved the case of surfaces there, but the general situation is completely similar). Let $\CX_{1}$ be the normalization of $\CX'\times_{T'}T_{1}$. Then we have a generically finite morphism $\CX_{1}\to\CA\times_{\CC}T_{1}$, which we denote by $f_{1}$. Let $X_{1}$ be the generic fibre of $\CX_{1}\to B$ and let $R_{1}$ be the ramification divisor of the morphism $X_{1}\to A\times_{K}(T_{1})_{K} $. Denote by $\CR_{1}$ the Zariski closure of $R_{1}$ in $\CX_{1}$. For any $t_{1}\in T_{1}(\CC)$, let $t'$ be the image of $t_{1}$ in $T'(\CC)$. Using the same argument as in the previous paragraph, we can deduce that the \'{e}tale locus of $(\CX'\times_{T'}T_{1})_{t_{1}}\to\CA\times\{t_{1}\}$ is contained in the smooth locus of $\CX'\times_{T'}T_{1}$. This means that the normalization $\CX_{1}\to \CX'\times_{T'}T_{1}$ is an isomorphism over this \'{e}tale locus. In particular, we see that $\CR_{1}$ is disjoint from the inverse image of this \'{e}tale locus in $\CX_{1}$. Moreover, as $(\CX'\times_{T'}T_{1})_{t_{1}}\simeq\CX'_{t'}$ is reduced and has the same dimension as $\CA\times\{t_{1}\}$, the same argument as in the previous paragraph shows that the image of the non-\'{e}tale locus of $(\CX'\times_{T'}T_{1})_{t_{1}}\to\CA\times\{t_{1}\}$ in $\CA\times\{t_{1}\}$ is a proper closed subset of $\CA\times\{t_{1}\}$. Combining all these facts, we conclude that $f_{1}(\CR_{1})$ does not contain $\CA\times\{t_{1}\}$ for any $t_{1}\in T_{1}$.

Furthermore, we can lift the sections $\{x'_{n}\}_{n\geq1}$ to the sections $\{x_{n,1}\}_{n\geq1}$ on $\CX_{1}$ for all but finitely many $n$.  Thus replacing $X',\CX'\to\CA\times_{\CC}T'$ by an irreducible component of $X_{1},\CX_{1}\to\CA\times_{\CC}T_{1}$, we may assume at the beginning that $T'$ is smooth and does not contain rational curves, $\CX'$ is normal and $f'(\CR')$ does not contain $\CA\times_{\CC}\{t'\}$ for any $t'\in T'$.

Using Stein factorization for $f'$, we get a birational morphism $\pi\colon \CX'\to\CX''$ and a finite morphism $f''\colon \CX''\to\CA\times_{\CC}T'$ such that $f'=f''\circ\pi$. Moreover, $\CX''$ is normal by \cite[Example 2.1.15]{zbMATH02134816}. Let $x''_{n}$ be the image of $x'_{n}$ in $\CX''$. Let $X''$ be the generic fibre of $\CX''\to B$. Let $R''$ be the ramification divisor of $X''\to A\times T'_{K}$ and let $\CR''$ be the Zariski closure of $R''$ in $\CX''$.

Finally, as $\CX',\CX''$ are normal, we have $\CR''=\pi_{*}(\CR')$ and hence $f''(\CR'')$ does not contain $\CA\times\{t'\}$ for any $t'\in T'(\CC)$.

From now on we focus on $\CX''\to\CA\times_{\CC}T'\to T'$ and the sections $\{x''_{n}\}_{n\geq1}$. In order to simplifying notations, we replace the original data
$$
(f\colon X\to A\times T_{K},R,f\colon \CX\to\CA\times_{\CC}T,\CR,\{x_{n}\}_{n\geq1})
$$
by
$$
(f''\colon X''\to A\times T'_{K},R'',f''\colon \CX''\to\CA\times_{\CC}T',\CR'',\{x''_{n}\}_{n\geq1}).
$$
This means that we assume the original data satisfy the following properties:

\begin{enumerate}[label=(\arabic*)]
\item $T$ is smooth and does not contain any rational curves.

\item $f\colon X\to A\times T_{K}$ is finite surjective and $\CX$ is normal.

\item $f(R)$ does not contain $A\times\{t\}$ for any $t\in T(\CC)$.

\end{enumerate}
Note that $A$ also does not contain rational curves, so the finiteness of $f$ implies that $X$ does not contain rational curves.

\medskip\medskip

\noindent\textbf{Take canonical model}
\medskip

\noindent Let $\pi\colon \widetilde{X}\to X$ be a resolution of singularity of $X$. Since $X$ is of general type, so is $\widetilde{X}$. By \cite[Corollary 1.1.1(3)]{zbMATH05775673}, the canonical ring $\oplus_{m\geq0}H^{0}(\widetilde{X},\CO_{\widetilde{X}}(mK_{\widetilde{X}}))$ of $\widetilde{X}$ is finitely generated. Let $X_{\rm can}={\rm Proj}\oplus_{m\geq0}H^{0}(\widetilde{X},\CO_{\widetilde{X}}(mK_{\widetilde{X}}))$ be the canonical model of $\widetilde{X}$ and let $\pi_{\rm can}\colon \widetilde{X}\dashrightarrow X_{\rm can}$ be the canonical rational map. Replace $\widetilde{X}$ by a resolution of singularity of $X_{\rm can}$ if necessary, we may assume that $\pi_{\rm can}$ is a morphism. By \cite[Theorem 1.15]{zbMATH06148846}, which is due to Reid, $X_{\rm can}$ is normal, projective, birational to $X$ and has canonical singularities. In particular, $X_{\rm can}$ has dlt singularities. Moreover, the canonical divisor $K_{X_{\rm can}}$ of $X_{\rm can}$ is $\QQ$-Cartier and ample. Since $X$ does not contain rational curves, by a theorem of Hacon-Mckernan (cf. \cite[Corollary 1.7]{zbMATH05166600}), there is a birational morphism $\varpi\colon X_{\rm can}\to X$ extending the rational map $\pi\circ(\pi_{\rm can})^{-1}\colon X_{\rm can}\dashrightarrow X$. Consider the composition $X_{\rm can}\overset{\varpi}\to X\overset{f}\to A\times_{K} T_{K}$ which is generically finite. Let $R_{\rm can}$ be the ramification divisor of this morphism. By Hurwitz's formula,
\begin{align}\label{Hur}
K_{X_{\rm can}}\sim(f\circ\varpi)^{*}K_{(A\times_{K} T_{K})}+R_{\rm can}.
\end{align}
As $K_{X_{\rm can}}$ is $\QQ$-Cartier and $T$ is smooth, we see from (\ref{Hur}) that $R_{\rm can}$ is also $\QQ$-Cartier. Moreover, since the restriction of $K_{(A\times_{K} T_{K})}$ to any fibre $A\times\{t\}$ is isomorphic to $K_{A\times\{t\}}$, which is trivial, we see that
$$
K_{{X_{\rm can}}}\big|_{(X_{\rm can})_{t}}\sim R_{\rm can}\big|_{(X_{\rm can})_{t}}.
$$
Since $K_{X_{\rm can}}$ is ample, we see that $R_{\rm can}$ is relatively ample over $T_{K}$ (hence it is also $\varpi$-ample).\medskip

\noindent Now we have the following lemma.

\begin{lemma}\label{exactly}
$R$ is exactly the image of $R_{\rm can}$ in $X$.
\end{lemma}
\begin{proof}
Since $\varpi$ is an isomorphism outside a Zariski closed subset of $X$ of codimension greater than 1, we see that $\varpi_{*}(R_{\rm can})$ contains $R$ and the complement $\varpi(R_{\rm can})\backslash R$ is of codimension greater than 1 in $X$. We need to show that this complement is actually empty.

Let $V=X\backslash R$ be the complement of $R$ in $X$. Then $R$ restricts to 0 on $V$. Thus the restriction of $R_{\rm can}$ to $\varpi^{-1}(V)$ is an effective exceptional divisor $E$ of $\varpi$ (that is, an effective divisor on $\varpi^{-1}(V)\subseteq X_{\rm can}$ whose image in $V\subseteq X$ has codimension greater than 1). Use the negativity lemma (cf. \cite[Lemma 3.39]{zbMATH01206984}) to the $\QQ$-Cartier divisor $-E$, we see that $-E$ is effective, which means that $E=0$. Therefore we have $R_{\rm can}\subseteq \varpi^{-1}(R)$.
\end{proof}

 Recall that $D=f_{*}(R)$ is the branch divisor of $f$. Then by our assumption, $D$ does not contain $A\times\{t\}$ for any $t\in T_{K}$. This means that $D\cap (A\times\{t\})$ is an effective divisor on $A\times\{t\}$ for any closed point $t\in T_{K}$. Now we have the following lemma.
\begin{lemma}\label{ample1}
For any closed point $t\in T_{K}$, $D\cap (A\times\{t\})$ is ample on $A\times\{t\}$.
\end{lemma}
\begin{proof}
For general point $t\in T_{K}$, $R_{\rm can}\cap(X_{\rm can})_{t}$ is of codimension 1 in $(X_{\rm can})_{t}$, which is an ample divisor on $(X_{\rm can})_{t}$ by the relative ampleness of $R_{\rm can}$ over $T$. Furthermore, the restriction map $(X_{\rm can})_{t}\to A\times\{t\}$ is generically finite for general $t$. Thus by Lemma \ref{pos}, we see that $(f\circ\varpi)_{*}(R_{\rm can}\cap(X_{\rm can})_{t})$ is ample on $A\times\{t\}$ for general $t\in T_{K}$. Note that the support of $(f\circ\varpi)_{*}(R_{\rm can}\cap(X_{\rm can})_{t})$ is contained in $D\cap (A\times\{t\})$. Thus by remark \ref{redample}, the restriction of $\CO_{A\times T_{K}}(D)$ on $A\times\{t\}$ is ample for general $t$. Finally by \cite[Remark 1.10(b)]{zbMATH00049143}, the restriction of $\CO_{A\times T_{K}}(D)$ on $A\times\{t\}$ is ample for every $t\in T_{K}$. In other words, $D\cap (A\times\{t\})$ is ample on $A\times\{t\}$.
\end{proof}
Take an integral model $\widetilde{\CX}$ of $\widetilde{X}$ over $B$ such that $\widetilde{X}\overset{\pi}\to X\overset{f}\to A\times T_{K}$ extend to $B$-morphisms $\widetilde{\CX}\to\CX\to A\times T_{K}$. For simplicity, we still denote these two morphisms by $\pi, f$. Taking a resolution of singularity if necessary, we can further assume that $\widetilde{\CX}$ is smooth.

Denote $\widetilde{R}=\pi_{\rm can}^{*}R_{\rm can}$ which is an effective $\QQ$-Cartier divisor on $\widetilde{X}$. Recall that we have defined $\CR,\CD$ to be the Zariski closure of $R,D$ in $\CX,\CA\times_{\CC}T$, respectively. Similarly, denote by $\widetilde{\CR}$ the Zariski closure of $\widetilde{R}$ in $\widetilde{\CX}$. Then the smoothness of $\widetilde{\CX}$ implies that $\widetilde{\CR}$ is $\QQ$-Cartier.

As $T(\CC)$ is compact, after taking a subsequence we may assume that $t_{n}$ tends to a point $t_{\infty}\in T(\CC)$ in the Euclidean topology.

\begin{lemma}\label{bigness}
There exists a positive rational number $c_{1}>0$ such that
$\widetilde{\CR}-c_{1}(f\circ\pi)^{*}\CD_{\rm red}$ is $\QQ$-linearly equivalent to $\widetilde{\CE}-\widetilde{\CV}_{T}-\widetilde{\CV}_{B}$, where $\widetilde{\CE}$ is an effective $\QQ$-Cartier divisor on $\widetilde{\CX}$ whose image in $\CA\times_{\CC}T$ does not contain $\CA\times\{t_{\infty}\}$ , $\widetilde{\CV}_{T}$ is a $\QQ$-Cartier divisor on $\widetilde{\CX}$ such that the image of ${\rm Supp}(\widetilde{\CV}_{T})$ in $T$ does not contain $t_{\infty}$ $($in particular, this implies that the image of ${\rm Supp}(\widetilde{\CV}_{T})$ in $T$ is of codimension $1)$, and $\widetilde{\CV}_{B}$ is a $\QQ$-Cartier divisor on $\widetilde{\CX}$ such that the image of ${\rm Supp}(\widetilde{\CV}_{B})$ in $B$ is finite.
\end{lemma}
\begin{proof}
Since $K_{X_{\rm can}}$ is ample on $X_{\rm can}$, there exists a positive rational number $c_{1}>0$ such that $K_{X_{\rm can}}-c_{1}(f\circ\varpi)^{*}D_{\rm red}$ is also ample. By (\ref{Hur}), we see that $R_{\rm can}-c_{1}(f\circ\varpi)^{*}D_{\rm red}$ is relatively ample over $T$. Fix an ample divisor $A_{T}$ on $T$ whose support does not contain $t_{\infty}$.  Now by \cite[Proposition 1.7.10]{zbMATH02134816}, $R_{\rm can}-c_{1}(f\circ\varpi)^{*}D_{\rm red}+n(f\circ\varpi)^{*}A_{T}$ is ample on $X_{\rm can}$ for some positive integer $n$. Thus it is $\QQ$-linearly equivalent to an effective $\QQ$-Cartier divisor $E$ on $X_{\rm can}$ such that $(f\circ\varpi)({\rm Supp}(E))$ does not contain $A\times \{t_{\infty}\}$.   Now we have
$$
R_{\rm can}-c_{1}(f\circ\varpi)^{*}D_{\rm red}\sim E-n(f\circ\varpi)^{*}A_{T}.
$$
Pullback everything to $\widetilde{X}$, we get
$$
\widetilde{R}-c_{1}(f\circ\pi)^{*}D_{\rm red}\sim \widetilde{E}-n(f\circ\pi)^{*}A_{T}
$$
where $\widetilde{E}=\pi_{\rm can}^{*}E$ is an effective $\QQ$-Cartier divisor on $\widetilde{X}$.

Taking the Zariski closure in $\widetilde{\CX}$ and using the smoothness of $\widetilde{\CX}$, we have
\begin{align}\label{xianxingdengjia}
\widetilde{\CR}-c_{1}(f\circ\pi)^{*}\CD_{\rm red}\sim\widetilde{\CE}-\widetilde{\CV}_{T}-\widetilde{\CV}_{B},
\end{align}
where $\widetilde{\CE},\widetilde{\CV}_{T}$ are the Zariski closure of $\widetilde{E},n(f\circ\pi)^{*}A_{T}$ in $\widetilde{\CX}$, respectively. And $\widetilde{\CV}_{B}$ is a $\QQ$-Cartier divisor on $\widetilde{\CX}$ such that the image of ${\rm Supp}(\widetilde{\CV}_{B})$ in $B$ is finite.

By our choice, we see that $c_{1},\widetilde{\CE},\widetilde{\CV}_{T},\widetilde{\CV}_{B}$ satisfy the conditions in the lemma.
\end{proof}

As $\pi\colon \widetilde{X}\to X$ is birational and $\{x_{n}\}_{n\geq1}$ is generic in $X$, after taking a subsequence, we may assume that $x_{n}$ can be uniquely lifted to a point $\tilde{x}_{n}\in \widetilde{X}(K)$ for all $n\geq1$. Similarly as before, we often view $\tilde{x}_{n}$ as a section $\tilde{x}_{n}\colon B\to\widetilde{\CX}$.
\subsection{Step 2: lower bound of transversal intersections}\label{5.4}
Now we come to the first estimate in the counting argument. We will give an upper bound (see Lemma \ref{subtle}) on the number of tangency points between the sections $\{(s_{n},t_{n})\}_{n\geq1}$ and $\widetilde{\CD}_{\rm red}$, or equivalently, a lower bound on their transversal intersection numbers. The main ingredients in the estimate are Theorem \ref{Nevanlinna}(1) and Lemma \ref{key}. All the estimates in this subsection are performed on $\CA\times_{\CC}T$ and do not involve $\widetilde{\CX}$.

Before estimating, we need to fix some notations. In order to present the proof in a logical order, part of these notations will only be used in the next subsection. Below we fix all the data and indicate which items will be needed later.\medskip\medskip

\noindent\textbf{Choice of all data}\medskip

\noindent Let $\CE'$ be an effective Cartier divisor on $\CA\times_{\CC}T$ such that
\begin{align}\label{CE'}
\widetilde{\CE}\subseteq (f\circ\pi)^{*}\CE'
\end{align}
Since $(f\circ\pi)({\rm Supp}(\widetilde{\CE}))$ does not contain $\CA\times\{t_{\infty}\}$, we can choose $\CE'$ such that ${\rm Supp}(\CE')$ does not contain $\CA\times\{t_{\infty}\}$. This will only be used in the next subsection.

Let $\widetilde{\CR}_{0}$ be the union of all irreducible components of $\widetilde{\CR}$ which dominate some irreducible component of $\CR$. View $\widetilde{\CR}_{0}$ as a $\QQ$-Cartier divisor on $\widetilde{\CX}$ such that for any irreducible component of $\widetilde{\CR}_{0}$, its multiplicity in $\widetilde{\CR}_{0}$ is exactly the same as its multiplicity in $\widetilde{\CR}$. Denote $\widetilde{\CR}_{1}=\widetilde{\CR}-\widetilde{\CR}_{0}$. Then $\widetilde{\CR}_{1}$ is an effective $\QQ$-Cartier divisor on $\widetilde{\CX}$ whose image in $\CX$ has codimension greater than 1.

Let $\CZ_{1}$ be the image of $\widetilde{\CR}_{1}$ in $\CA\times_{\CC}T$. Use Raynaud--Gruson's flattening theorem (cf. \cite[Th\'{e}or\`{e}me 5.2.2]{zbMATH03360302}) to the composition $\CZ_{1}\hookrightarrow\CA\times_{\CC}T\to T$, we get a proper birational map $g\colon T'\to T$ such that the strict transform $\CZ'_{1}$ of $\CZ_{1}$ under $g$ is flat over $T'$. Note that $\CZ'_{1}$ is a closed subscheme of $\CA\times_{\CC}T'$, and the inverse image of $\CZ_{1}$ in $\CA\times_{\CC}T'$ is the union of $\CZ'_{1}$ and some $T'$-vertical closed subscheme $\CV_{T'}$ of $\CA\times_{\CC}T'$ (i.e. the image of $\CV_{T'}$ in $T'$ is not Zariski dense). After removing finitely many points, we can assume that $t_{n}$ has a unique pre-image $t'_{n}$ in $T'$ which does not lie in ${\rm pr}_{T'}(\CV_{T'})$ for all $n\geq1$. Furthermore, we can assume $t'_{n}$ tends to $t'_{\infty}$ in the Euclidean topology. By our definition of $\widetilde{\CR}_{1},\CZ_{1},\CZ'_{1}$, we see that for any $t'\in T'$, $(\CZ'_{1})_{t'}$ has codimension greater than 1 in $\CA\times\{t'\}$. Let $Z_{1}$ be the image of $(\CZ'_{1})_{t'_{\infty}}\subseteq\CA\times\{t'_{\infty}\}$ in $\CA\times\{t_{\infty}\}$. We will only use $\CR_{0},\CR_{1}$ (and other related notations such as $Z_{1}$) in the next subsection.

Using Lemma \ref{key} for the morphism $\CD_{\red}\hookrightarrow \CA\times_{\CC}T\to T$ and the sequence $t_{n}\to t_{\infty}$ of points of $T$, we get a Zariski open dense subset $\CD_{t_{\infty}}^{\circ}\subseteq \CD_{t_{\infty}}$ satisfying the properties in Lemma \ref{key}. Furthermore, taking a subsequence of $\{t_{n}\}_{n\geq1}$ if necessary, we may assume that $\{t_{n}\}_{n\geq1}$ satisfies the conditions of remark \ref{zilie}. Denote $\CD_{b,t_{\infty}}^{\circ}\colonequals (\CD_{t_{\infty}}^{\circ})_{b}$.

Now choose a set $S_{0}\subseteq B(\CC)$ of Lebesgue measure 0 satisfying the following conditions:
\begin{enumerate}[label=(\arabic*)]
\item\label{item1} $S_{0}$ contains the set $S$ in Corollary \ref{location}, which is of Lebesgue measure 0.
\item\label{item2} $S_{0}$ contains all the points $b\in B(\CC)$ such that $(\CD_{\rm red})_{b}$ or $(\widetilde{\CR}_{\rm red})_{b}$ is non-reduced, which is finite by the generic reducedness  (cf. \cite[\href{https://stacks.math.columbia.edu/tag/054Z}{Lemma 054Z}, \href{https://stacks.math.columbia.edu/tag/020I}{Lemma 020I}, \href{https://stacks.math.columbia.edu/tag/0578}{Lemma 0578}]{stacks-project}).
\item\label{item3} $S_{0}$ contains all the points $b\in B(\CC)$ such that $\CD_{b,t_{\infty}}^{\circ}$ is not Zariski dense in $\CD_{b,t_{\infty}}$, which is also finite since $\CD_{t_{\infty}}^{\circ}$ is Zariski dense in $\CD_{t_{\infty}}$.
\item\label{item4} $S_{0}$ contains all the points $b\in B(\CC)$ such that either $\CA_{b}$ is not an abelian variety, or $\CD_{b,t_{\infty}}$ is not an ample divisor on $\CA_{b}\times\{t_{\infty}\}$, which is also finite by Corollary \ref{ample1}.
\item\label{item4.5} $S_{0}$ contains all the points $b\in B(\CC)$ such that $\CE'_{b,t_{\infty}}$ contains $\CA_{b}\times\{t_{\infty}\}$, which is also finite by the sentence after (\ref{CE'}).
\item\label{item5} $S_{0}$ contains the image of ${\rm Supp}(\widetilde{\CV}_{B})\subseteq\CX$ in $B$, which is also finite by the property of $\CV_{B}$ (see Lemma \ref{bigness}).
\item\label{item6} $S_{0}$ contains all the points $b\in B(\CC)$ such that $(Z_{1})_{b}$ is of codimension less or equal to 1 in $\CA_{b}\times\{t_{\infty}\}$, which is also finite since $Z_{1}$ is of codimension greater than 1 in $\CA\times\{t_{\infty}\}$.
\end{enumerate}
Now we use Theorem \ref{linear} to construct a limit entire curve. Take $\ell_{n}\to\infty$ such that $\ell^{-1}_{n}s_{n}$ converges to a nonzero element $s_{\infty}$ in $V(A,K)$. Let $G=G(s_{\infty})$ be the smallest abelian subvariety of $A$ such that $s_{\infty}$ lies in $G(K)_{\RR}$. Denote by $\CG$ the Zariski closure of $G$ in $\CA$.

Now choose any point $b\notin S_{0}$. By \ref{item4} above we have $\CD_{b,t_{\infty}}$ ample on the abelian variety $\CA_{b}\times\{t_{\infty}\}$, which implies that $(\CD_{b,t_{\infty}})_{\rm red}$ is also ample on $\CA_{b}\times\{t_{\infty}\}$ (see remark \ref{redample}). Consider the composition $(\CD_{b,t_{\infty}})_{\rm red}\hookrightarrow\CA_{b}\times\{t_{\infty}\}\twoheadrightarrow(\CA_{b}/\CG_{b})\times\{t_{\infty}\}$. The ampleness above implies that this is surjective. Thus the image of $\CD_{b,t_{\infty}}^{\circ}$ under this map is Zariski dense in $(\CA_{b}/\CG_{b})\times\{t_{\infty}\}$ by \ref{item3} above. Let $Z$ be the complement of $\CD_{t_{\infty}}^{\circ}$ in $\CD_{t_{\infty}}$. By Corollary \ref{location} and property \ref{item4.5}, \ref{item6} of $S_{0}$, we can choose a point $x\in(\CA_{b}\times\{t_{\infty}\})\backslash(\CD_{b,t_{\infty}}\cup\CE'_{b,t_{\infty}}\cup(Z_{1})_{b})$ satisfying the following conditions:
\begin{enumerate}[label=(\arabic*)]
\item\label{x1} $x$ is a limit point of $s_{n}(B)\times\{t_{n}\}$. More precisely, there exists a sequence $\{b_{n}\}_{n\geq1}\subseteq B(\CC)$ tending to $b$ such that $x$ is the limit of a subsequence of $\{(s_{n}(b_{n}),t_{n})\}_{n\geq1}$. After taking a subsequence, we can assume that $x$ is the limit of $\{(s_{n}(b_{n}),t_{n})\}_{n\geq1}$.
\item\label{x2} $(\CD_{b,t_{\infty}})_{\rm red}\cap (x+\CG_{b}\times\{t_{\infty}\})$ is a reduced divisor on $(x+\CG_{b}\times\{t_{\infty}\})$.
\item\label{x3} $Z_{b}\cap(x+\CG_{b}\times\{t_{\infty}\})$ has codimension greater than 1 in $(x+\CG_{b}\times\{t_{\infty}\})$.
\item\label{x4} $(Z_{1})_{b}\cap(x+\CG_{b}\times\{t_{\infty}\})$ has codimension greater than 1 in $(x+\CG_{b}\times\{t_{\infty}\})$.
\end{enumerate}

As $b\notin S_{0}$, by property \ref{item5} of $S_{0}$, we can take an open disc $\BD\subseteq B(\CC)$ with center $b$ such that $\BD$ does not intersect with the image of ${\rm Supp}(\widetilde{\CV}_{B})$ in $B$. After removing finitely many $n$, we may assume that $b_{n}\in\BD$ for all $n\geq1$. Apply Theorem \ref{linear} to $\{s_{n}\}_{n\geq1}$ and the limit point $x\in(\CA_{b}\times\{t_{\infty}\})\backslash (\CD_{b,t_{\infty}}\cup\CE'_{b,t_{\infty}}\cup (Z_{1})_{b})$ of $\{s_{n}(B)\}_{n\geq1}$. We obtain a linear entire curve $\varphi_{\infty}=\phi_{(x,\delta(s_{\infty},v_{b}))}\times\{t_{\infty}\}\colon \CC\to\CA_{b}\times\{t_{\infty}\}$, which is the unique limit of the re-parametrization $\{\varphi_{n}=\phi_{n}\times\{t_{n}\}\colon \BD_{r_{n}}\to\CA\times\{t_{n}\}\}_{n\geq1}$ of $\{(s_{n},t_{n})\colon B\to\CA\times\{t_{n}\}\}_{n\geq1}$. Explicitly, recall that $\varphi_{n}$ is defined by

$$
\varphi_{n}\colon \BD_{r_{n}}\to\CA\times\{t_{n}\},\quad\quad z\mapsto (s_{n}(b_{n}+\ell_{n}^{-1}z),t_{n}).
$$

\noindent Here $z$ is the standard coordinate of $\BD$ and the sum $b_{n}+\ell_{n}^{-1}z$ is taken in $\BD$. Recall also that $\{r_{n}\}_{n\geq1}$ is a sequence of positive numbers satisfying

$$
r_{n}/\ell_{n}<1-|b_{n}|,\quad r_{n}\to\infty.
$$

\noindent  Moreover, the Zariski closure of $\phi_{(x,\delta(s_{\infty},v_{b}))}(\CC)$ in $\CA_{b}\times\{t_{\infty}\}$ equals to $(x+\CG_{b}\times\{t_{\infty}\})$. Replacing $x$ by a small translation if necessary, we may assume that $\phi_{\infty}(\partial{\BD})\times\{t_{\infty}\}$ does not intersect with $\CD_{b,t_{\infty}}\cup\CE'_{b,t_{\infty}}$. Now let

$$
\psi_{n}\colon \BD_{r_{n}}\to\widetilde{\CX},\quad\quad z\mapsto \tilde{x}_{n}(b_{n}+\ell_{n}^{-1}z)
$$

\noindent be the lift of $\varphi_{n}$ on $\widetilde{\CX}$.

\medskip

Fix smooth Hermitian metrics on the line bundles
$$
\CO_{\widetilde{\CX}}(\widetilde{\CR}), \CO_{\CA\times_{\CC}T}(\CD_{\rm red}),\CO_{\widetilde{\CX}}(\widetilde{\CE}),\CO_{\widetilde{\CX}}(\widetilde{\CV}_{T}),\CO_{\widetilde{\CX}}(\widetilde{\CV}_{B}),\CO_{\CA\times_{\CC}T}(\CE').
$$
For line bundle $L=\CO_{\CA\times_{\CC}T}(\CD_{\rm red})$ or $L=\CO_{\CA\times_{\CC}T}(\CE')$ on $\CA\times_{\CC}T$ with the fixed metric $\left\|\cdot\right\|_{L}$ above, we choose the induced metric $(f\circ\pi)^{*}\left\|\cdot\right\|_{L}$ on $(f\circ\pi)^{*}L$. Moreover, for any $\ZZ$-linear combination of the line bundles above (including the two pullback line bundles), we can also fix a natural metric on this new line bundle induced by the metrics above.

For a line bundle $L$ which is a $\ZZ$-linear combination of the line bundles above (including the two pullback line bundles) with the fixed metric $\left\|\cdot\right\|_{L}$, a holomorphic map $\phi\colon \BD_{R}\to \widetilde{\CX}$ (or $\phi\colon \BD_{R}\to\CA\times_{\CC}T$, depending on $L$ is a line bundle on $\widetilde{\CX}$ or on $\CA\times_{\CC}T$), set
$$
T(r,\phi,L)\colonequals T(r,\phi,L,\left\|\cdot\right\|_{L}).
$$
If $F$ is an effective Cartier divisor associated to a section $s_{F}\in H^{0}(L)$, set $T(r,\phi,F)\colonequals T(r,\phi,L)$. Moreover, if $f(\BD_{R})\nsubseteq {\rm Supp}(F)$, set
$$
m(r,\phi,F)\colonequals m(r,\phi,F,\left\|\cdot\right\|_{L}).
$$
Finally, the above definitions can be extended by $\QQ$-linearity to any $\QQ$-linear combination of the line bundles above (including the two pullback line bundles).

Now fix an $0<\varepsilon<\frac{1}{2}$. We will finally take $\varepsilon$ small enough to get a contradiction. We also fix a smooth Hermitian metric on the line bundle $\CO_{\CA_{b}\times\{t_{\infty}\}}((\CD_{b,t_{\infty}})_{\rm red})$ on $\CA_{b}\times\{t_{\infty}\}$ and define the characteristic functions, proximity functions of $(\CD_{b,t_{\infty}})_{\rm red}$ to be the corresponding functions with respect to this metric.\medskip\medskip

\noindent\textbf{Estimations on the limit entire curve $\varphi_{\infty}$}\medskip

\noindent We first do some estimates on the limit entire curve $\varphi_{\infty}$. The main ingredient is  Theorem \ref{Nevanlinna}(1). The philosophy is that, $N^{(1)}$ serves as the main part of the characteristic function $T\approx N^{(1)}+(N-N^{(1)})+m$, while $N-N^{(1)}$ and $m$ are negligible compared to $N^{(1)}$. This corresponds to inequalities (\ref{3epsilon}) and (\ref{m}) below, which will be used respectively when estimating the upper bound and the lower bound of $N-N^{(1)}$.\medskip

By assumption \ref{x2} for $x$, $(\CD_{b,t_{\infty}})_{\rm red}\cap (x+\CG_{b}\times\{t_{\infty}\})$ is a reduced divisor on $(x+\CG_{b}\times\{t_{\infty}\})$. Using Theorem \ref{Nevanlinna}(1) and the ampleness of $\CD_{b,t_{\infty}}$ on $\CA_{b}\times\{t_{\infty}\}$ (assumption \ref{item4} for $S_{0}$), we see that there exists a subset $\Sigma_{\varepsilon}\subseteq \RR_{>0}$ of finite Lebesgue measure such that for any $r\in\RR_{>0}\backslash\Sigma_{\varepsilon}$,
\begin{align}\label{sigma}
 T(r,\varphi_{\infty},(\CD_{b,t_{\infty}})_{\rm red})\leq N^{(1)}(r,\varphi_{\infty},(\CD_{b,t_{\infty}})_{\rm red})+\varepsilon T(r,\varphi_{\infty},(\CD_{b,t_{\infty}})_{\rm red}).
\end{align}
As $\varphi_{\infty}(\CC)$ contains $x$, it is not contained in $\CD$. Thus by Theorem \ref{fmt},
\begin{align}\label{fmtyy}
T(r,\varphi_{\infty},(\CD_{b,t_{\infty}})_{\rm red})= N(r,\varphi_{\infty},(\CD_{b,t_{\infty}})_{\rm red})+m(r,\varphi_{\infty},(\CD_{b,t_{\infty}})_{\rm red})-m(1,\varphi_{\infty},(\CD_{b,t_{\infty}})_{\rm red}).
\end{align}
Since $m(r,\varphi_{\infty},(\CD_{b,t_{\infty}})_{\rm red})$ is bounded below for all $r>0$, $m(1,\varphi_{\infty},(\CD_{b,t_{\infty}})_{\rm red})$ is a constant, we see that there is a constant $c_{2}>0$ such that
\begin{align}\label{c2}
 T(r,\varphi_{\infty},(\CD_{b,t_{\infty}})_{\rm red})\geq N(r,\varphi_{\infty},(\CD_{b,t_{\infty}})_{\rm red})-c_{2}
\end{align}
Therefore for any fixed $\varepsilon>0$ and any $r\in\RR_{>0}\backslash\Sigma_{\varepsilon}$, combine (\ref{sigma}) and (\ref{c2}), we get
\begin{align*}
N(r,\varphi_{\infty},(\CD_{b,t_{\infty}})_{\rm red})-N^{(1)}(r,\varphi_{\infty},(\CD_{b,t_{\infty}})_{\rm red})\leq c_{2}+\frac{\varepsilon}{1-\varepsilon}N^{(1)}(r,\varphi_{\infty},(\CD_{b,t_{\infty}})_{\rm red}).
 \end{align*}
Moreover, for fix $0<\varepsilon<\frac{1}{2}$ above, $\frac{\varepsilon}{1-\varepsilon}<2\varepsilon$ and $N^{(1)}(r,\varphi_{\infty},(\CD_{b,t_{\infty}})_{\rm red}) \geq\frac{c_{2}}{\varepsilon}$ when $r$ large enough. Thus by enlarging $\Sigma_{\varepsilon}$ we may assume that for all $r\in\RR_{>0}\backslash\Sigma_{\varepsilon}$,
\begin{align}\label{3epsilon}
N(r,\varphi_{\infty},(\CD_{b,t_{\infty}})_{\rm red})-N^{(1)}(r,\varphi_{\infty},(\CD_{b,t_{\infty}})_{\rm red})\leq 3\varepsilon N^{(1)}(r,\varphi_{\infty},(\CD_{b,t_{\infty}})_{\rm red}).
\end{align}
On the other hand, combine (\ref{sigma}) and (\ref{fmtyy}), we have for all $r\in\RR_{>0}\backslash\Sigma_{\varepsilon}$,
\begin{align}
m(r,\varphi_{\infty},(\CD_{b,t_{\infty}})_{\rm red})&=T(r,\varphi_{\infty},(\CD_{b,t_{\infty}})_{\rm red})-N(r,\varphi_{\infty},(\CD_{b,t_{\infty}})_{\rm red})+m(1,\varphi_{\infty},(\CD_{b,t_{\infty}})_{\rm red})\notag\\
&\leq N^{(1)}(r,\varphi_{\infty},(\CD_{b,t_{\infty}})_{\rm red})+\varepsilon T(r,\varphi_{\infty},(\CD_{b,t_{\infty}})_{\rm red})-N(r,\varphi_{\infty},(\CD_{b,t_{\infty}})_{\rm red})\notag\\
&+m(1,\varphi_{\infty},(\CD_{b,t_{\infty}})_{\rm red})\notag\\
&\leq \varepsilon T(r,\varphi_{\infty},(\CD_{b,t_{\infty}})_{\rm red})+m(1,\varphi_{\infty},(\CD_{b,t_{\infty}})_{\rm red})\notag
\end{align}
For fixed $\varepsilon$, when $r$ large enough, we have $\varepsilon T(r,\varphi_{\infty},(\CD_{b,t_{\infty}})_{\rm red})\geq m(1,\varphi_{\infty},(\CD_{b,t_{\infty}})_{\rm red})$. Thus by enlarging $\Sigma_{\varepsilon}$, we can assume that
\begin{align}\label{m}
m(r,\varphi_{\infty},(\CD_{b,t_{\infty}})_{\rm red})\leq 2\varepsilon T(r,\varphi_{\infty},(\CD_{b,t_{\infty}})_{\rm red})
\end{align}
holds for all $r\in\RR_{>0}\backslash\Sigma_{\varepsilon}$.\medskip\medskip

\noindent\textbf{Upper bound of $N(r,\varphi_{n},\CD_{\rm red})-N^{(1)}(r,\varphi_{n},\CD_{\rm red})$}\medskip

\noindent Now we establish the main lemma in this subsection, which gives an upper bound of $N(r,\varphi_{n},\CD_{\rm red})-N^{(1)}(r,\varphi_{n},\CD_{\rm red})$. This can also be understood as a lower bound for the transversal intersection numbers of $\varphi_{n}$ and $\CD_{\rm red}$.

\begin{lemma}\label{subtle}
For fixed $\varepsilon>0$, we can enlarge the set $\Sigma_{\varepsilon}$ of finite Lebesgue measure above such that for any $r\in\RR_{>0}\backslash\Sigma_{\varepsilon}$, the inequality
\begin{align}\label{upperbound}
N(r,\varphi_{n},\CD_{\rm red})-N^{(1)}(r,\varphi_{n},\CD_{\rm red})\leq\varepsilon T(r,\varphi_{\infty},(\CD_{b,t_{\infty}})_{\rm red})
\end{align}
holds for $n$ large enough $($depending on $\varepsilon$ and $r)$.
\end{lemma}

The key to the proof of Lemma \ref{subtle} is the inequality (\ref{3epsilon}) together with two limiting arguments: one is Lemma \ref{key}, and the other adapts the argument of Lemma \ref{limits}. The philosophy is that, if the limit entire curve $\varphi_{\infty}$ has many transversal intersections with $(\CD_{b,t_{\infty}})_{\rm red}$, then each $\varphi_{n}$ also has many transversal intersections with $\CD_{\rm red}$.\medskip\medskip

The first limiting argument is contained in the following lemma. Recall that $\CD_{t_{\infty}}^{\circ}\subseteq \CD_{t_{\infty}}$ is a Zariski dense open subset of $\CD_{t_{\infty}}$ satisfying the conditions in Lemma \ref{key} (applying to $\CD\hookrightarrow\CA\times_{\CC}T\to T$ and the sequence $\{t_{n}\}_{n\geq1}$). By property \ref{item3} of the set $S_{0}\subseteq B(\CC)$, we see that $\CD_{b,t_{\infty}}^{\circ}$ is a Zariski dense open subset of $\CD_{b,t_{\infty}}$.

\begin{lemma}\label{sublemma}
For each $z\in\CC$, if $\varphi_{\infty}(z)\in \CD_{b,t_{\infty}}^{\circ}$ and ${\rm ord}_{z}\varphi_{\infty}^{*}(\CD_{b,t_{\infty}})_{\rm red}=1$, then there is an open neighbourhood $V_{z}$ of $z$ such that for $n$ large enough, ${\rm ord}_{z'}\varphi_{n}^{*}\CD_{\rm red}\leq 1$ for all $z'\in V_{z}$.
\end{lemma}
\begin{proof}
Denote $y=\varphi_{\infty}(z)\in \CD_{b,t_{\infty}}^{\circ}$. We first prove that there is an open neighbourhood $U_{y}$ of $y$ in $\CA\times_{\CC}T$ such that for all $n$ large enough and $|z'|< 2|z|$, if $\varphi_{n}(z')\in U_{y}$, then ${\rm ord}_{z'}\varphi_{n}^{*}\CD_{\rm red}\leq 1$.

In fact, if this is not the case, then there is a $z_{n}\in\BD_{2|z|}$ such that $\varphi_{n}(z_{n})$ tends to $y$ and ${\rm ord}_{z_{n}}\varphi_{n}^{*}\CD_{\rm red}>1$ for all sufficiently large $n$. Denote $b'_{n}=b_{n}+\ell_{n}^{-1}z_{n}\in B(\CC)$ and $r'_{n}=r_{n}-|z_{n}|$. Then $y=\varphi_{\infty}(z)$ is the limit point of $(s_{n}(b'_{n}),t_{n})\in\CA\times\{t_{n}\}$. Consider the re-parametrization
$$
\varphi'_{n}\colon \BD_{r'_{n}}\to\CA\times\{t_{n}\},\quad\quad z\mapsto (s_{n}(b'_{n}+\ell_{n}^{-1}z),t_{n}).
$$
Note that $r'_{n}$ still tends to infinity as $|z_{n}|\leq 2|z|$ is bounded. By Theorem \ref{linear}, $\varphi'_{n}$ converges (uniformly on any compact subset of $\CC$) to a linear entire curve $\varphi'_{\infty}$ on $\CA_{b}\times\{t_{\infty}\}$ which is a translation (from $x$ to $y$) of $\varphi_{\infty}$.

The assumption ${\rm ord}_{w_{n}}\varphi_{n}^{*}\CD_{\rm red}>1$ means that $({\rm d}\varphi'_{n})_{0}(v_{\rm st})\in T_{\varphi'_{n}(0)}\CD_{\rm red}$ for all $n$ (large enough). Here $v_{\rm st}=\frac{\rm d}{{\rm d}z}$ is the tangent vector of $\BD_{r'_{n}}$ at 0 under the standard coordinate $z$. By the uniform converging property, $({\rm d}\varphi'_{n})_{0}(v_{\rm st})$ tends to $({\rm d}\varphi'_{\infty})_{0}(v_{\rm st})$. Since $y\in \CD_{b,t_{\infty}}^{\circ}$, by Lemma \ref{key}, we see that $({\rm d}\varphi'_{\infty})_{0}(v_{\rm st})\in T_{\varphi'_{\infty}(0)}(\CD_{b,t_{\infty}})_{\rm red}$. Or equivalently, $({\rm d}\varphi_{\infty})_{z}(v_{\rm st})\in T_{\varphi_{\infty}(z)}(\CD_{b,t_{\infty}})_{\rm red}$, which contradicts to the assumption that ${\rm ord}_{z}\varphi_{\infty}^{*}(\CD_{b,t_{\infty}})_{\rm red}=1$.

Now take a small neighbourhood $V_{z}$ of $z$ contained in $\BD_{2|z|}$ such that $\varphi_{\infty}(V_{z})\subseteq U_{y}$. Use the uniformly converge again, we see that for $n$ large enough, $\varphi_{n}(V_{z})\subseteq U_{y}$. Therefore by the argument above, ${\rm ord}_{z'}\varphi_{n}^{*}\CD_{\rm red}\leq 1$ holds for all $z'\in V_{z}$ and sufficiently large $n$. This finishes the proof of the lemma.
\end{proof}

Now we are ready to prove Lemma \ref{subtle}. The argument is similar to that of Lemma \ref{limits}, with slightly more elaborate estimates.

\begin{proof}[Proof of Lemma $\ref{subtle}$]
For fixed $r>0$ such that $\varphi_{\infty}(\partial\BD_{r})\cap\CD=\emptyset$ (we can enlarge $\Sigma_{\varepsilon}$ so that each $r\in\RR_{>0}\backslash\Sigma_{\varepsilon}$ satisfies this property), choose an open neighbourhood $W_{z}$ of $z$ (depending on $r$) for each $z\in\overline{\BD}_{r}$ satisfying the following properties:
\begin{enumerate}[label=(\arabic*)]
\item\label{V1} For any $z$ and any $z'\in W_{z}\backslash\{z\}$, we have $\varphi_{\infty}(z')\notin \CD$.
\item\label{V2} If $\varphi_{\infty}(z)\in \CD$, $|z|>1$ and $z'\in W_{z}$, then $|z'|>1$ and
$$
  \log r-\log|z'| \leq 2(\log r-\log|z|)
$$
\item\label{V2.5} If $|z|<1$, then $W_{z}\subseteq \BD_{1}$.
\item\label{V3} If $\varphi_{\infty}(z)\in \CD_{b,t_{\infty}}^{\circ}$ and ${\rm ord}_{z}\varphi_{\infty}^{*}(\CD_{b,t_{\infty}})_{\rm red}=1$, then $W_{z}\subseteq V_{z}$, where $V_{z}$ is in Lemma \ref{sublemma}.
\end{enumerate}
As $\overline{\BD}_{r}$ is compact, there is a finite number of $W_{z}$ covering it. Explicitly, assume $\overline{\BD}_{r}=\cup_{j=1}^{m}W_{z_{j}}$, where $m$ is a positive integer and $z_{j}\in \BD_{r}$.

By definition, we have
\begin{align}\label{bydefinition}
N(r,\varphi_{n},\CD_{\rm red})-N^{(1)}(r,\varphi_{n},\CD_{\rm red})&=\int_{1}^{r}\sum\limits_{z\in\BD_{t}}({\rm ord}_{z}\varphi_{n}^{*}\CD_{\rm red}-\min\{1,{\rm ord}_{z}\varphi_{n}^{*}\CD_{\rm red}\})\frac{{\rm d}t}{t} \notag \\
&=\sum\limits_{z\in\BD_{1}}({\rm ord}_{z}\varphi_{n}^{*}\CD_{\rm red}-\min\{1,{\rm ord}_{z}\varphi_{n}^{*}\CD_{\rm red}\})\log r \notag \\
&+\sum\limits_{1\leq|z|<r}({\rm ord}_{z}\varphi_{n}^{*}\CD_{\rm red}-\min\{1,{\rm ord}_{z}\varphi_{n}^{*}\CD_{\rm red}\})(\log r-\log|z|)\notag\\
&=\sum\limits_{z\in\BD_{r}}({\rm ord}_{z}\varphi_{n}^{*}\CD_{\rm red}-\min\{1,{\rm ord}_{z}\varphi_{n}^{*}\CD_{\rm red}\})\cdot\notag\\
&\min\{\log r,\log r-\log\left|z\right|\}
\end{align}
For $1\leq j\leq m$, define
\begin{align*}
S_{j}\colonequals \sum\limits_{z\in\BD_{r}\cap W_{z_{j}}}({\rm ord}_{z}\varphi_{n}^{*}\CD_{\rm red}-\min\{1,{\rm ord}_{z}\varphi_{n}^{*}\CD_{\rm red}\})\min\{\log r,\log r-\log\left|z\right|\}
\end{align*}

By Rouch\'{e}'s theorem and assumption \ref{V1}, we know that for $n$ large enough and $z\in\BD_{r}$, if $\varphi_{n}(z)\in\CD$, then $z$ lies in exactly one $W_{z_{j}}$. In this case we have (by (\ref{bydefinition}))
\begin{align}\label{sigmaj}
N(r,\varphi_{n},\CD_{\rm red})-N^{(1)}(r,\varphi_{n},\CD_{\rm red})&=\sum\limits_{1\leq j\leq m}S_{j}.
\end{align}

By Rouch\'{e}'s theorem again, for $n$ large enough, we have\\
$\sum_{z\in W_{z_{j}}}{\rm ord}_{z}\varphi_{n}^{*}\CD_{\rm red}=\sum_{z\in W_{z_{j}}}{\rm ord}_{z}\varphi_{\infty}^{*}\CD_{\rm red}$  for each $1\leq j\leq m$. By assumption \ref{V1} for $W_{z}$, we see that $\sum_{z\in W_{z_{j}}}{\rm ord}_{z}\varphi_{\infty}^{*}\CD_{\rm red}={\rm ord}_{z_{j}}\varphi_{\infty}^{*}\CD_{\rm red}$.

Now by our assumption $\varphi_{\infty}(\partial\BD_{1})\cap\CD=\emptyset$, we see that if $|z_{j}|=1$, then $S_{j}=0$ for $n$ large enough. If $|z_{j}|\neq1$, by assumption \ref{V2} and \ref{V2.5}, we have
\begin{align}\label{inanycase}
S_{j}&=\sum\limits_{z\in W_{z_{j}}}({\rm ord}_{z}\varphi_{n}^{*}\CD_{\rm red}-\min\{1,{\rm ord}_{z}\varphi_{n}^{*}\CD_{\rm red}\})\min\{\log r,\log r-\log\left|z\right|\}\notag\\
&\leq 2\sum\limits_{z\in W_{z_{j}}}({\rm ord}_{z}\varphi_{n}^{*}\CD_{\rm red})\min\{\log r,\log r-\log\left|z_{j}\right|\}\notag\\
&=2({\rm ord}_{z_{j}}\varphi_{\infty}^{*}\CD_{\rm red})\min\{\log r,\log r-\log\left|z_{j}\right|\}
\end{align}

Now for $n$ large enough and every $1\leq j\leq m$, if $\varphi_{\infty}(z_{j})\in \CD_{b,t_{\infty}}^{\circ}$  and ${\rm ord}_{z_{j}}\varphi_{\infty}^{*}(\CD_{b,t_{\infty}})_{\rm red}=1$, we see from Lemma \ref{sublemma} and assumption \ref{V3} that ${\rm ord}_{z}\varphi_{n}^{*}\CD_{\rm red}-\min\{1,{\rm ord}_{z}\varphi_{n}^{*}\CD_{\rm red}\}=0$ for all $z\in W_{z_{j}}$. Thus we have $S_{j}=0$. If ${\rm ord}_{z_{j}}\varphi_{\infty}^{*}(\CD_{b,t_{\infty}})_{\rm red}=0$, then ${\rm ord}_{z_{j}}\varphi_{\infty}^{*}\CD_{\rm red}=0$ so by the previous paragraph we also get $S_{j}=0$ for $n$ large enough.

Let $J_{1}$ be the set of all $j$ satisfying ${\rm ord}_{z_{j}}\varphi_{\infty}^{*}(\CD_{b,t_{\infty}})_{\rm red}>1$, and let $J_{2}$ be the set of all $j$ satisfying ${\rm ord}_{z_{j}}\varphi_{\infty}^{*}(\CD_{b,t_{\infty}})_{\rm red}=1$ and $\varphi_{\infty}(z_{j})\in\CD_{t_{\infty}}\backslash \CD_{t_{\infty}}^{\circ}=Z$. Then the argument in the previous paragraph shows that
\begin{align}\label{J}
\sum\limits_{1\leq j\leq m}S_{j}=\sum\limits_{j\in J_{1}}S_{j}+\sum\limits_{j\in J_{2}}S_{j}.
\end{align}
for $n$ large enough.

Choose a positive integer $C>0$ such that $(\CD_{\rm red})_{b,t_{\infty}}\subseteq C(\CD_{b,t_{\infty}})_{\rm red}$. Then we have\\
${\rm ord}_{z}\varphi_{\infty}^{*}\CD_{\rm red}\leq C{\rm ord}_{z}\varphi_{\infty}^{*}(\CD_{b,t_{\infty}})_{\rm red}$. Moreover, if $z=z_{j}$ for some $j\in J_{2}$, then\\
${\rm ord}_{z}\varphi_{\infty}^{*}\CD_{\rm red}\leq C{\rm ord}_{z}\varphi_{\infty}^{*}(\CD_{b,t_{\infty}})_{\rm red}=C\leq C{\rm ord}_{z}\varphi_{\infty}^{*}Z$. Hence by (\ref{inanycase}) we get
\begin{align}\label{2C4}
\sum\limits_{j\in J_{1}}S_{j}+\sum\limits_{j\in J_{2}}S_{j}&\leq 2C\sum\limits_{j\in J_{1}}({\rm ord}_{z_{j}}\varphi_{\infty}^{*}(\CD_{b,t_{\infty}})_{\rm red})\min\{\log r,\log r-\log|z_{j}|\}\notag\\
&+2C\sum\limits_{j\in J_{2}}({\rm ord}_{z_{j}}\varphi_{\infty}^{*}Z)\min\{\log r,\log r-\log|z_{j}|\}\notag\\
&\leq 2C\left(2(N(r,\varphi_{\infty},(\CD_{b,t_{\infty}})_{\rm red})-N^{(1)}(r,\varphi_{\infty},(\CD_{b,t_{\infty}})_{\rm red}))+N(r,\varphi_{\infty},Z)\right)
\end{align}
for $n$ large enough. Here we used the fact that ${\rm ord}_{z_{j}}\varphi_{\infty}^{*}(\CD_{b,t_{\infty}})_{\rm red}\leq 2({\rm ord}_{z_{j}}\varphi_{\infty}^{*}(\CD_{b,t_{\infty}})_{\rm red}-\min\{1,{\rm ord}_{z_{j}}\varphi_{\infty}^{*}(\CD_{b,t_{\infty}})_{\rm red}\})$ for $j\in J_{1}$.

By (\ref{3epsilon}), with $\varepsilon$ replaced by $\frac{\varepsilon}{48C}$, we get
$$
N(r,\varphi_{\infty},(\CD_{b,t_{\infty}})_{\rm red})-N^{(1)}(r,\varphi_{\infty},(\CD_{b,t_{\infty}})_{\rm red})\leq \frac{\varepsilon}{16C}N^{(1)}(r,\varphi_{\infty},(\CD_{b,t_{\infty}})_{\rm red})
$$
holds for all $r>0$ outside a set of finite Lebesgue measure. Enlarge $\Sigma_{\varepsilon}$ if necessary, we may assume that this holds for all $r\in\RR_{>0}\backslash\Sigma_{\varepsilon}$.

By Theorem \ref{fmt} we easily see that $N(r,\varphi_{\infty},(\CD_{b,t_{\infty}})_{\rm red})\leq 2T(r,\varphi_{\infty},(\CD_{b,t_{\infty}})_{\rm red})$ for $r$ sufficiently large. Thus we can enlarge $\Sigma_{\varepsilon}$ so that this holds for all $r\in\RR_{>0}\backslash\Sigma_{\varepsilon}$. Therefore we get
\begin{align}\label{16}
N(r,\varphi_{\infty},(\CD_{b,t_{\infty}})_{\rm red})-N^{(1)}(r,\varphi_{\infty},(\CD_{b,t_{\infty}})_{\rm red})\leq \frac{\varepsilon}{8C}T(r,\varphi_{\infty},(\CD_{b,t_{\infty}})_{\rm red}).
\end{align}

Denote $W=Z_{b}\cap(x+\CG_{b}\times\{t_{\infty}\})$. Then by our choice of $x$ (see property \ref{x3}), $W$ has codimension at least 2 in $(x+\CG_{b}\times\{t_{\infty}\})$. By Theorem \ref{Nevanlinna}(2), we may enlarge $\Sigma_{\varepsilon}$ so that for any $r\in\RR_{>0}\backslash\Sigma_{\varepsilon}$,
\begin{align*}
N(r,\varphi_{\infty},W)&\leq \frac{\varepsilon}{4C} T(r,\varphi_{\infty},(x+\CG_{b}\times\{t_{\infty}\})\cap(\CD_{b,t_{\infty}})_{\rm red}) \notag\\
&=\frac{\varepsilon}{4C} T(r,\varphi_{\infty},(\CD_{b,t_{\infty}})_{\rm red}).
\end{align*}
Thus we have $N(r,\varphi_{\infty},Z)=N(r,\varphi_{\infty},W)\leq \frac{\varepsilon}{4C} T(r,\varphi_{\infty},(\CD_{b,t_{\infty}})_{\rm red})$. Combine this with (\ref{sigmaj}), (\ref{J}), (\ref{2C4}), (\ref{16}), we get
$$
N(r,\varphi_{n},\CD_{\rm red})-N^{(1)}(r,\varphi_{n},\CD_{\rm red})\leq \varepsilon T(r,\varphi_{\infty},(\CD_{b,t_{\infty}})_{\rm red}).
$$
for any $r\in\RR_{>0}\backslash\Sigma_{\varepsilon}$ and $n$ large enough (depending on $r$). This finishes the proof of Lemma \ref{subtle}.
\end{proof}

\subsection{Step 3: upper bound of transversal intersections}\label{5.5}
On the other hand, we can estimate a lower bound for $N(r,\varphi_{n},\CD_{\rm red})-N^{(1)}(r,\varphi_{n},\CD_{\rm red})$ (or equivalently, an upper bound for the transversal intersection numbers of $\varphi_{n}$ with $\CD_{\rm red}$). The precise estimate proceeds as follows.

\begin{lemma}\label{xiajie-}
There is a constant $c>0$ and a set $\Sigma\subseteq\RR_{>0}$ of finite Lebesgue measure such that for all $r\in\RR_{>0}\backslash\Sigma$, the inequality
\begin{align}\label{xiajie}
N(r,\varphi_{n},\CD_{\rm red})-N^{(1)}(r,\varphi_{n},\CD_{\rm red})\geq c\cdot T(r,\varphi_{\infty},(\CD_{b,t_{\infty}})_{\rm red})
\end{align}
holds for $n$ large enough $($depending on $r)$.
\end{lemma}

Proposition \ref{generic1} follows immediately from Lemma \ref{subtle} and Lemma \ref{xiajie-}. In fact, take $0<\varepsilon<\min\{\frac{1}{2},c\}$. Let $\Sigma_{\varepsilon}$ be the set of finite Lebesgue measure in Lemma \ref{subtle}. Take $r\in\RR_{>0}\backslash (\Sigma_{\varepsilon}\cup\Sigma)$ large enough such that $T(r,\varphi_{\infty},(\CD_{b,t_{\infty}})_{\rm red})>0$. Then for $n$ large enough (depending on $\varepsilon$ and $r$) we get a contradiction from (\ref{upperbound}) and (\ref{xiajie}). Therefore, it remains to prove Lemma \ref{xiajie-}.

The proof of Lemma \ref{xiajie-} consists of three parts. The first part is to obtain a lower bound for $N(r,\psi_{n},\widetilde{\CR})$ in terms of $T(r,\varphi_{\infty},(\CD_{b,t_{\infty}})_{\rm red})$ (see (\ref{12}) below), which essentially relies on the positivity of $R$ (Lemma \ref{bigness} in \S\ref{5.3}). The second part uses a tangency observation mentioned in \S\ref{idea} to derive a lower bound of $N(r,\varphi_{n},\CD_{\rm red})-N^{(1)}(r,\varphi_{n},\CD_{\rm red})$ in terms of $N(r,\psi_{n},\widetilde{\CR}_{0})$ (see Corollary \ref{c5-} below), here we recall that $\widetilde{\CR}_{0}$ is the union (counted with multiplicities) of all irreducible components of $\widetilde{\CR}$ which dominates an irreducible component of $\widetilde{\CR}$. We further recall that $\widetilde{\CR}=\widetilde{\CR}_{0}+\widetilde{\CR}_{1}$, and thus $N(r,\psi_{n},\widetilde{\CR})=N(r,\psi_{n},\widetilde{\CR}_{0})+N(r,\psi_{n},\widetilde{\CR}_{1})$. The last part is to estimate the remaining term $N(r,\psi_{n},\widetilde{\CR}_{1})$. We shall show that it is actually negligible (see Lemma \ref{NR1-} below), due to the fact that the image of $\widetilde{\CR}_{1}$ in $\CX$ has codimension greater than 1.\medskip\medskip

\noindent\textbf{Lower bound of $N(r,\psi_{n},\widetilde{\CR})$}\medskip

\noindent Now we estimate a lower bound of $N(r,\psi_{n},\widetilde{\CR})$. By (\ref{xianxingdengjia}) and Lemma \ref{property}, we have
\begin{align*}
T(r,\psi_{n},\widetilde{\CR}-c_{1}(f\circ\pi)^{*}\CD_{\rm red})&=T(r,\psi_{n},\widetilde{\CE}-\widetilde{\CV}_{T}-\widetilde{\CV}_{B})+O(1)\\
&=T(r,\psi_{n},\widetilde{\CE})-T(r,\psi_{n},\widetilde{\CV}_{T}+\widetilde{\CV}_{B})+O(1)\\
&=m(r,\psi_{n},\widetilde{\CE})+N(r,\psi_{n},\widetilde{\CE})-m(1,\psi_{n},\widetilde{\CE})-m(r,\psi_{n},\widetilde{\CV}_{T}+\widetilde{\CV}_{B})\\&- N(r,\psi_{n}, \widetilde{\CV}_{T}+\widetilde{\CV}_{B})+m(1,\psi_{n},\widetilde{\CV}_{T}+\widetilde{\CV}_{B})+O(1).
\end{align*}
Here the first two equalities follow from Lemma \ref{property}(1). In particular, the term $O(1)$ is independent of $r$ and $n$. The third equality follows from Theorem \ref{fmt}. Now $N(r,\psi_{n},\widetilde{\CE})\geq0$ and $m(r,\psi_{n},\widetilde{\CE}), m(1,\psi_{n},\widetilde{\CV}_{T}+\widetilde{\CV}_{B})$ are uniformly (for all $n$ and $r$) bounded below. By our choice of $\widetilde{\CV}_{T},\widetilde{\CV}_{B}$ (see Lemma \ref{bigness}), we have $N(r,\psi_{n}, \widetilde{\CV}_{T}+\widetilde{\CV}_{B})=0$ for all $r,n$ and $m(r,\psi_{n},\widetilde{\CV}_{T}+\widetilde{\CV}_{B})$ is uniformly (for all $n$ and $r$) bounded above. Finally by (\ref{CE'}),
$$
m(1,\psi_{n},\widetilde{\CE})\leq m(1,\varphi_{n},\CE')+O(1)
$$
Here the term $O(1)$ is independent of $n$. As $\varphi_{\infty}(\partial{\BD})$ does not intersect $\CE'$, by Lemma \ref{limits}, $m(1,\varphi_{n},\CE')$ converges to $m(1,\varphi_{\infty},\CE')$ as $n\to\infty$. Hence the term $m(1,\varphi_{n},\CE')$ is also bounded. In conclusion, we have
$$
T(r,\psi_{n},\widetilde{\CR}-c_{1}(f\circ\pi)^{*}\CD_{\rm red})\geq O(1)
$$
Here $O(1)$ is independent of $r$ and $n$.

Thus we have
\begin{align*}
m(r,\psi_{n},\widetilde{\CR})+N(r,\psi_{n},\widetilde{\CR})
&=T(r,\psi_{n},\widetilde{\CR})+m(1,\psi_{n},\widetilde{\CR})\notag\\
&\geq c_{1}T(r,\psi_{n},(f\circ\pi)^{*}\CD_{\rm red})+O(1)\notag\\
&= c_{1}T(r,\varphi_{n},\CD_{\rm red})+O(1)
\end{align*}
for all $r,n$. Here $O(1)$ is independent of $r$ and $n$. Now for all sufficiently large $r>0$, $T(r,\varphi_{\infty},\CD_{\rm red})>0$. Thus from Lemma \ref{limits} we see that $T(r,\varphi_{n},\CD_{\rm red})\geq\frac{1}{2}T(r,\varphi_{\infty},\CD_{\rm red})$ for $n$ large enough (depending on $r$). Moreover, we have $T(r,\varphi_{\infty},\CD_{\rm red})=T(r,\varphi_{\infty},(\CD_{\rm red})_{b,t_{\infty}})\geq T(r,\varphi_{\infty},(\CD_{b,t_{\infty}})_{\rm red})+O(1)$ where $O(1)$ is independent of $r$ and $n$. Combine all of these we see that for sufficiently large $r>0$,
\begin{align}\label{mN}
m(r,\psi_{n},\widetilde{\CR})+N(r,\psi_{n},\widetilde{\CR})\geq \frac{c_{1}}{2}T(r,\varphi_{\infty},(\CD_{b,t_{\infty}})_{\rm red})+O(1).
\end{align}
holds for $n$ large enough (depending on $r$).

Now by Lemma \ref{exactly} and the definition $\widetilde{R}=\pi_{\rm can}^{*}R_{\rm can}$ we see that $\widetilde{\CR}\subseteq c_{3}(f\circ\pi)^{*}\CD_{\rm red}$ for some positive integer $c_{3}$. Therefore we have
\begin{align}\label{c3}
m(r,\psi_{n},\widetilde{\CR})\leq c_{3}\cdot m(r,\varphi_{n},\CD_{\rm red})+O(1)
\end{align}
Here the term $O(1)$ is independent of $r$ and $n$.

As $\varphi_{\infty}(\CC)$ contains $x$, it is not contained in $\CD$. Thus by Lemma \ref{limits}, we see that $m(r,\varphi_{n},\CD_{\rm red})$ tends to $m(r,\varphi_{\infty},\CD_{\rm red})$ as $n\to\infty$. Hence for fixed $r>0$, $m(r,\varphi_{n},\CD_{\rm red})\leq m(r,\varphi_{\infty},\CD_{\rm red})+1$ for $n$ large enough (depending on $r$). Now choose a positive integer $c_{4}$ such that $(\CD_{\rm red})_{b,t_{\infty}}\subseteq c_{4}(\CD_{b,t_{\infty}})_{\rm red}$. Then for all $r>0$ and $n$ large enough (depending on $r$),
\begin{align}\label{c4}
m(r,\varphi_{n},\CD_{\rm red})&\leq m(r,\varphi_{\infty},\CD_{\rm red})+1\notag\\
&=m(r,\varphi_{\infty},(\CD_{\rm red})_{b,t_{\infty}})+1\notag\\
&\leq c_{4}m(r,\varphi_{\infty},(\CD_{b,t_{\infty}})_{\rm red})+O(1)
\end{align}
Here the term $O(1)$ is independent of $r$ and $n$. Now for any $0<\varepsilon<\frac{1}{2}$, combine (\ref{m}), (\ref{c3}), (\ref{c4}), we have for all $r\in\RR_{>0}\backslash\Sigma_{\varepsilon}$ and $n$ large enough (depending on $\varepsilon$ and $r$),
\begin{align}\label{c3c4}
m(r,\psi_{n},\widetilde{\CR})&\leq c_{3}c_{4}m(r,\varphi_{\infty},(\CD_{b,t_{\infty}})_{\rm red})+O(1)\notag\\
&\leq 2\varepsilon c_{3}c_{4}T(r,\varphi_{\infty},(\CD_{b,t_{\infty}})_{\rm red})+O(1)
\end{align}
Here the term $O(1)$ depends only on $\varepsilon$. In particular, it is independent of $r$ and $n$. Finally combine (\ref{mN}) and (\ref{c3c4}), we get for all $r\in\RR_{>0}\backslash\Sigma_{\varepsilon}$ and $n$ large enough (depending on $\varepsilon$ and $r$),
\begin{align}\label{O1}
N(r,\psi_{n},\widetilde{\CR})\geq \left(\frac{c_{1}}{2}-2\varepsilon c_{3}c_{4}\right)T(r,\varphi_{\infty},(\CD_{b,t_{\infty}})_{\rm red})-O(1)
\end{align}
Here the term $O(1)$ is independent of $r,n$. For fixed $0<\varepsilon<\frac{c_{1}}{2c_{3}c_{4}}$, when $r\in\RR_{>0}\backslash\Sigma_{\varepsilon}$ large enough, the term $O(1)$ in (\ref{O1}) is less than $\frac{1}{2}\left(\frac{c_{1}}{2}-2\varepsilon c_{3}c_{4}\right)T(r,\varphi_{\infty},(\CD_{b,t_{\infty}})_{\rm red})$. Thus we have
\begin{align}\label{12}
N(r,\psi_{n},\widetilde{\CR})\geq \frac{1}{2}\left(\frac{c_{1}}{2}-2\varepsilon c_{3}c_{4}\right)T(r,\varphi_{\infty},(\CD_{b,t_{\infty}})_{\rm red})
\end{align}
for $n$ large enough (depending on $\varepsilon$ and $r$).\medskip\medskip

\noindent\textbf{Tangency and a upper bound of $N(r,\psi_{n},\widetilde{\CR}_{0})$}\medskip

\noindent Next we establish a relation between $N(r,\psi_{n},\widetilde{\CR}_{0})$ and $N(r,\varphi_{n},\CD_{\rm red})-N^{(1)}(r,\varphi_{n},\CD_{\rm red})$, the main result is Corollary \ref{c5-}. As we have mentioned in \S\ref{idea}, an important point in this step is a tangency observation. Now we give the rigorous statement.\medskip

Recall that $\widetilde{\CR}_{0}\subseteq \widetilde{\CR}$ is the union of all irreducible components of $\widetilde{\CR}$ which dominates an irreducible component of $\CR$.

\begin{proposition}\label{tangent}
For any closed point $x\in\widetilde{\CR}_{0}$, the image of the tangent map $({\rm d}(f\circ\pi))_{x}\colon T_{x}(\widetilde{\CX})\to T_{(f\circ\pi)(x)}(\CA\times_{\CC} T)$ at $x$ is contained in $T_{(f\circ\pi)(x)}(\CD_{\rm red})$. Here we naturally view $T_{(f\circ\pi)(x)}(\CD_{\rm red})$ as a subspace of $T_{(f\circ\pi)(x)}(\CA\times_{\CC} T)$.
\end{proposition}
\begin{proof}
Denote $d+1$ to be the dimension of $\widetilde{\CX}$. Let $U\subseteq (\widetilde{\CR}_{0})_{\rm red}$ be the \'{e}tale locus of $((f\circ\pi)\big|_{\widetilde{\CR}_{0}})_{\rm red}\colon (\widetilde{\CR}_{0})_{\rm red}\to\CD_{\rm red} $. By our assumption, $U$ is Zariski dense in $\widetilde{\CR}_{0}$. When $x\in U$, we have an isomorphism ${\rm d}((f\circ\pi)\big|_{\widetilde{\CR}_{0}})_{\rm red}\colon  T_{x}((\widetilde{\CR}_{0})_{\rm red})\to T_{(f\circ\pi)(x)}(\CD_{\rm red})$. On the one hand, the image of ${\rm d}(f\circ\pi)\colon T_{x}(\widetilde{\CX})\to T_{(f\circ\pi)(x)}(\CA\times_{\CC} T)$ contains the image of ${\rm d}((f\circ\pi)\big|_{\widetilde{\CR}})_{\rm red}$, which has dimension at least $d$.   On the other hand, since $\CA\times_{\CC} T$ is smooth and ${\rm d}(f\circ\pi)$ at $x$ is not surjective, its image has dimension at most $\dim(T_{(f\circ\pi)(x)}(\CA\times_{\CC} T))-1=(d+1)-1=d$. Therefore its image is exactly $T_{(f\circ\pi)(x)}(\CD_{\rm red})$.

Now we consider the morphism on the tangent bundle ${\rm d}(f\circ\pi)\colon T\widetilde{\CX}\to T(\CA\times_{\CC}T)$. Since $\CD_{\rm red}$ is a closed subscheme of $\CA\times_{\CC}T$, $T\CD_{\rm red}$ is a closed subscheme of $T(\CA\times_{\CC}T)$. Thus the inverse image $Z=({\rm d}(f\circ\pi))^{-1}(T\CD_{\rm red})$ is a closed subset of $T\widetilde{\CX}$.  Let $p\colon T\widetilde{\CX}\to \widetilde{\CX}$ be the natural projection. The argument above shows that $Z$ contains
$p^{-1}(U)$. Hence $Z$ contains the Zariski closure of $p^{-1}(U)$ in $T\widetilde{\CX}$. As $\widetilde{\CX}$ is smooth, $p\colon T\widetilde{\CX}\to \widetilde{\CX}$ is a vector bundle over $\widetilde{\CX}$. This means that the Zariski closure of $p^{-1}(U)$ in $T\widetilde{\CX}$ is exactly $p^{-1}(\widetilde{\CR}_{0})$. In other words, for any point $x\in\widetilde{\CX}$, the image of the tangent map ${\rm d}(f\circ\pi)\colon T_{x}(\widetilde{\CX})\to T_{(f\circ\pi)(x)}(\CA\times_{\CC} T)$ is contained in $T_{(f\circ\pi)(x)}(\CD_{\rm red})$.
\end{proof}

\begin{corollary}\label{qie}
With the notations above. Denote $\tilde{x}\in \widetilde{X}(K)$ and $(f\circ\pi)(\tilde{x})=(s,t)\in A(K)\times T_{K}(K)$. View $\tilde{x}$ as a section $\tilde{x}\colon B\to\widetilde{\CX}$. If $\tilde{x}(b)\in\widetilde{\CR}_{0}$ for some $b\in B(\CC)$, then we have $(s,t)\colon B\to\CA\times_{\CC}T$ tangent to $\CD_{\rm red}$ at $(s(b),t(b))$. More precisely, the image of the tangent map ${\rm d}(s,t)\colon T_{b}B\to T_{(s(b),t(b))}(\CA\times_{\CC}T)$ at $b$ is contained in $T_{(s(b),t(b))}(\CD_{\rm red})$.
\end{corollary}
\begin{proof}
Note that $(s,t)\colon B\to \CA\times_{\CC}T$ is exactly the composition $B\overset{\tilde{x}}\to \widetilde{\CX}\overset{\pi}\to \CX\overset{f}\to \CA\times_{\CC}T$, so the image in the Corollary is contained in the image of $T_{\tilde{x}(b)}(\widetilde{\CX})\to T_{(f\circ\pi)(\tilde{x}(b))}(\CA\times_{\CC}T)$, which by Proposition \ref{tangent}, is contained in $T_{(s(b),t(b))}(\CD_{\rm red})$.
\end{proof}
Since $\pi\colon \widetilde{X}\to X$ is birational, after taking a subsequence of $\{x_{n}\}_{n\geq1}$, we may assume that $x_{n}$ has a lift $\tilde{x}_{n}$ in $\widetilde{X}$ for all $n$. Hence Corollary \ref{qie} applies for all $\tilde{x}_{n}$.

We can translate Corollary \ref{qie} into an inequality in Nevanlinna theory. Recall that $\psi_{n},\varphi_{n}$ are respectively re-parametrizations of $\tilde{x}_{n},(s_{n},t_{n})$ defined in the last subsection. We have the following corollary.

\begin{corollary}\label{c5-}
There is a constant $c_{5}>0$ such that for all $n\geq1$, the inequality
\begin{align}\label{N1NN1-}
N(r,\psi_{n},\widetilde{\CR}_{0})\leq c_{5}(N(r,\varphi_{n},\CD_{\rm red})-N^{(1)}(r,\varphi_{n},\CD_{\rm red}))
\end{align}
holds for all $r\geq1$.
\end{corollary}

\begin{proof}
Choose a positive integer $c_{0}>0$ such that $\widetilde{\CR}_{0}\subseteq c(\widetilde{\CR}_{0})_{\rm red}$. Then we have
\begin{align}\label{C3-}
N(r,\psi_{n},\widetilde{\CR}_{0})\leq c_{0}N(r,\psi_{n},(\widetilde{\CR}_{0})_{\rm red})
\end{align}

Corollary \ref{qie} just means that
\begin{align}\label{N1-}
N^{(1)}(r,\psi_{n},(\widetilde{\CR}_{0})_{\rm red})\leq N(r,\varphi_{n},\CD_{\rm red})-N^{(1)}(r,\varphi_{n},\CD_{\rm red})
\end{align}
As $(\widetilde{\CR}_{0})_{\rm red}\subseteq (f\circ\pi)^{*}\CD_{\rm red}$, we have (by definition)
\begin{align}\label{NN1-}
N(r,\psi_{n},(\widetilde{\CR}_{0})_{\rm red})-N^{(1)}(r,\psi_{n},(\widetilde{\CR}_{0})_{\rm red})\leq N(r,\varphi_{n},\CD_{\rm red})-N^{(1)}(r,\varphi_{n},\CD_{\rm red})
\end{align}
In fact, if $\psi_{n}(z)\notin\widetilde{\CR}_{0}$, then
$${\rm ord}_{z}\psi_{n}^{*}(\widetilde{\CR}_{0})_{\rm red}=0\leq {\rm ord}_{z}\varphi_{n}^{*}(\CD_{\rm red})-\min\{1,{\rm ord}_{z}\varphi_{n}^{*}(\CD_{\rm red})\},
$$
If $\psi_{n}(z)\in\widetilde{\CR}_{0}$, then
$${\rm ord}_{z}\psi_{n}^{*}(\widetilde{\CR}_{0})_{\rm red}\leq {\rm ord}_{z}\varphi_{n}^{*}(\CD_{\rm red}) \quad {\rm and}\quad
\min\{1,{\rm ord}_{z}\psi_{n}^{*}(\widetilde{\CR}_{0})_{\rm red}\}=\min\{1,{\rm ord}_{z}\varphi_{n}^{*}(\CD_{\rm red})\}=1.
$$
So we always have
$$
{\rm ord}_{z}\psi_{n}^{*}(\widetilde{\CR}_{0})_{\rm red}-\min\{1,{\rm ord}_{z}\psi_{n}^{*}(\widetilde{\CR}_{0})_{\rm red}\}\leq {\rm ord}_{z}\varphi_{n}^{*}(\CD_{\rm red})-\min\{1,{\rm ord}_{z}\varphi_{n}^{*}(\CD_{\rm red})\}.
$$
Hence by the definition of the counting function, we get (\ref{NN1-}).

Combine (\ref{C3-}), (\ref{N1-}) and (\ref{NN1-}) we see that
\begin{align*}
N(r,\psi_{n},\widetilde{\CR}_{0})&\leq c_{0}N(r,\psi_{n},(\widetilde{\CR}_{0})_{\rm red})\notag\\
&=c_{0}\left(N(r,\psi_{n},(\widetilde{\CR}_{0})_{\rm red})-N^{(1)}(r,\psi_{n},(\widetilde{\CR}_{0})_{\rm red})+N^{(1)}(r,\psi_{n},(\widetilde{\CR}_{0})_{\rm red})\right) \notag \\
&\leq 2c_{0}(N(r,\varphi_{n},\CD_{\rm red})-N^{(1)}(r,\varphi_{n},\CD_{\rm red})).
\end{align*}
Now (\ref{N1NN1-}) is true by taking $c_{5}=2c_{0}>0$.
\end{proof}
\medskip

\noindent\textbf{Upper bound of $N(r,\psi_{n},\widetilde{\CR}_{1})$}\medskip

\noindent Recall that $\widetilde{\CR}=\widetilde{\CR}_{0}+\widetilde{\CR}_{1}$, and thus $N(r,\psi_{n},\widetilde{\CR})=N(r,\psi_{n},\widetilde{\CR}_{0})+N(r,\psi_{n},\widetilde{\CR}_{1})$. We have established the estimates (\ref{12}), (\ref{N1NN1-}) for $N(r,\psi_{n},\widetilde{\CR}), N(r,\psi_{n},\widetilde{\CR}_{0})$ respectively. Now we estimate the remaining term $N(r,\psi_{n},\widetilde{\CR}_{1})$, which turns out to be negligible, due to the fact that the image of $\widetilde{\CR}_{1}$ in $\CX$ has codimension greater than 1.

\begin{lemma}\label{NR1-}
For fixed $\varepsilon>0$, we can enlarge the set $\Sigma_{\varepsilon}$ above such that for all $r\in\RR_{>0}\backslash\Sigma_{\varepsilon}$, the inequality
\begin{align}\label{R1-}
N(r,\psi_{n},\widetilde{\CR}_{1})\leq\varepsilon T(r,\varphi_{\infty},(\CD_{b,t_{\infty}})_{\rm red})
\end{align}
holds for $n$ large enough $($depending on $\varepsilon$ and $r)$.
\end{lemma}

\begin{proof}
Since $\CZ_{1}=(f\circ\pi)(\widetilde{\CR}_{1})$, there is a positive integer $C_{1}$ such that $\widetilde{\CR}_{1}\subseteq C_{1}(f\circ\pi)^{*}\CZ_{1}$. Then we have
\begin{align}\label{C1-}
N(r,\psi_{n},\widetilde{\CR}_{1})\leq C_{1}N(r,\varphi_{n},\CZ_{1}).
\end{align}

Recall from ``\textbf{Choice of all data}'' that we have a commutative diagram
\begin{displaymath}
\xymatrix{
\CZ'_{1}\ar@{^{(}->}[r]\ar[rd]_{\rm flat} & \CA\times_{\CC}T'\ar[r]^{{\rm id}_{\CA}\times g}\ar[d] & \CA\times_{\CC}T\ar[d] & \CZ_{1}\ar@{_{(}->}[l]\\
\empty & T'\ar[r]^{g} & T
}
\end{displaymath}
where $g\colon T'\to T$ is proper birational and $\CZ'_{1}$ is the strict transform of $\CZ_{1}$ under $g$. Thus $\CZ'_{1}$ is the pullback of $\CZ_{1}$ in $\CA\times_{\CC}T'$ minus some $T'$-vertical component $\CV_{T'}$. As $g$ is birational, for all but finitely many $n$, we can lift $\varphi_{n}\colon \BD_{r_{n}}\to \CA\times\{t_{n}\}$ uniquely to $\varphi'_{n}\colon \BD_{r_{n}}\to\CA\times\{t'_{n}\}$. Moreover, $t_{n}\notin {\rm pr}_{T'}(\CV_{T'})$ for all but finitely many $n$. In this case $N(r,\varphi'_{n},\CV_{T'})\equiv0$. Thus for all but finitely many $n$ and all $r>0$,
\begin{align}\label{=-}
N(r,\varphi_{n},\CZ_{1})=N(r,\varphi'_{n},\CZ'_{1})
\end{align}
Since $\{\varphi_{n}\}_{n\geq1}$ has a limit $\varphi_{\infty}$, $\{\varphi'_{n}\}_{n\geq1}$ also has a limit $\varphi'_{\infty}$, which is a lift of $\varphi_{\infty}$ in $\CA\times_{\CC}T'$. Recall that $Z_{1}$ is the image of $(\CZ'_{1})_{t'_{\infty}}$ in $\CA\times\{t_{\infty}\}$. By property \ref{x4} of the point $x$, we see that $\varphi_{\infty}(\CC)\nsubseteq Z_{1}$, and thus $\varphi'_{\infty}(\CC)\nsubseteq\CZ'_{1}$. By Lemma \ref{limits}, for any $r\in\RR_{>0}\backslash\Sigma_{\varepsilon}$,
\begin{align}\label{leq-}
N(r,\varphi'_{n},\CZ'_{1})\leq 2N(r,\varphi'_{\infty},\CZ'_{1})
\end{align}
holds for $n$ large enough.

Take a positive integer $C_{2}$ such that $(\CZ'_{1})_{t'_{\infty}}\subseteq C_{2}({\rm id}_{\CA}\times g)^{*}Z_{1}$. Then
\begin{align}\label{C2-}
N(r,\varphi'_{\infty},\CZ'_{1})=N(r,\varphi'_{\infty},(\CZ'_{1})_{t'_{\infty}})\leq C_{2}N(r,\varphi_{\infty},Z_{1})
\end{align}

By our assumption, the Zariski closure of $\varphi_{\infty}(\CC)$ in $\CA\times_{\CC}T$ is $(x+\CG_{b}\times\{t_{\infty}\})$, whose intersection with $(Z_{1})_{b}$ has codimension greater than 1 in $(x+\CG_{b}\times\{t_{\infty}\})$ (see property \ref{x4} of $x$). By Lemma \ref{Nevanlinna}(2), we can enlarge $\Sigma_{\varepsilon}$ such that for all $r\in\RR_{>0}\backslash\Sigma_{\varepsilon}$,
\begin{align}\label{R1zuihou-}
N(r,\varphi_{\infty},Z_{1})=N(r,\varphi_{\infty},(Z_{1})_{b})\leq \frac{\varepsilon}{2C_{1}C_{2}} T(r,\varphi_{\infty},(\CD_{b,t_{\infty}})_{\rm red})
\end{align}
Combine (\ref{C1-}), (\ref{=-}), (\ref{leq-}), (\ref{C2-}), (\ref{R1zuihou-}), we see that for all $r\in\RR_{>0}\backslash\Sigma_{\varepsilon}$, the inequality (\ref{R1-}) holds for $n$ large enough.
\end{proof}

\begin{proof}[Proof of Lemma $\ref{xiajie}$]
Recall that we have defined constants $c_{1},c_{2},c_{3},c_{4},c_{5}>0$ which are independent of $\varepsilon$. Now take $c=\frac{c_{1}}{8c_{5}}$ and choose $0<\varepsilon<\min\{\frac{1}{2}, \frac{c_{1}}{2c_{3}c_{4}}, \frac{c_{1}}{8(c_{3}c_{4}+1)}\}$. Take $r>1$ outside the set $\Sigma_{\varepsilon}$ such that $T(r,\varphi_{\infty},(\CD_{b,t_{\infty}})_{\rm red})>0$. Then we deduce from (\ref{12}), (\ref{N1NN1-}), (\ref{R1-}) that
\begin{align*}
N(r,\varphi_{n},\CD_{\rm red})-N^{(1)}(r,\varphi_{n},\CD_{\rm red})&\geq \frac{1}{c_{5}}\left(\frac{c_{1}}{4}-\varepsilon(c_{3}c_{4}+1)\right)T(r,\varphi_{\infty},(\CD_{b,t_{\infty}})_{\rm red})\\
&\geq c\cdot T(r,\varphi_{\infty},(\CD_{b,t_{\infty}})_{\rm red}).
\end{align*}
for $n$ large enough (depending on $\varepsilon$ and $r$). This completes the proof.
\end{proof}

\section{Proof of the main theorem}\label{3}
In this section we prove the main theorem \ref{main}. We first use Proposition \ref{generic1} to prove the following proposition.

\begin{proposition}\label{ht}
Let $K$ be a function field of transcendence degree \emph{1} over a field $k$ of characteristic \emph{0} such that $k$ is algebraically closed in $K$. Let $X$ be projective variety over $K$ with a finite morphism $f\colon X\to A$ to an abelian variety $A$ over $K$. Then the heights of the points in $(X\backslash{\rm Sp}(X))(K)$ are bounded.
\end{proposition}

\begin{proof}
First consider the case when $k=\CC$. We prove by contradiction. Assume to the contrary that the Proposition does not hold (for $k=\CC$), and $X$ is a counterexample which has the smallest dimension among all counterexamples. Let $\{x_{n}\}_{n\geq1}$ be a sequence of points of $(X\backslash{\rm Sp}(X))(K)$ with height tending to infinity as $n\to\infty$. Then we claim that $\{x_{n}\}_{n\geq1}$ is generic. In fact, for any infinite subsequence $\{y_{n}\}_{n\geq1}$ of $\{x_{n}\}_{n\geq1}$, let $Y$ be an irreducible component of the Zariski closure of $\cup_{n\geq1}\{y_{n}\}$ in $X$. Then $\{y_{n}\}_{n\geq1}$ has a subsequence $\{z_{n}\}_{n\geq1}$ contained in $Y$ which is Zariski dense in $Y$. By our assumption, $\{z_{n}\}_{n\geq1}$ has height tending to infinity. Moreover, since $z_{n}\notin {\rm Sp}(X)$ and ${\rm Sp}(Y)\subseteq {\rm Sp}(X)$, we see that $z_{n}\in (Y\setminus{\rm Sp}(Y))(K)$ for all $n\geq1$. Since $Y$ also has a finite morphism to $A$ via the composition $Y\hookrightarrow X\to A$, we see that $Y$ is also a counterexample of the Proposition. By the minimality of the dimension of $X$, we see that $\dim Y=\dim X$, which implies that $Y=X$. This proves our claim.

Now notice that ${\rm Sp}(X)\subsetneqq X$ by our assumption. Thus by \cite[Corollary 4.3]{arXiv:2305.14789}, which is essentially due to Ueno and Kawamata, $X$ is of general type. This means that $X$ and $\{x_{n}\}_{n\geq1}$ satisfy the conditions in Proposition \ref{generic1} (with $T$ taken to be trivial). Hence by our assumption the height of $\{x_{n}\}_{n\geq1}$ should be bounded, a contradiction!

For general $k$, we use the Lefschetz principle. In fact, it suffices to prove that for any infinite sequence $\{x_{n}\}_{n\geq1}$ of $(X\backslash{\rm Sp}(X))(K)$, the heights of $\{x_{n}\}_{n\geq1}$ is bounded. The datum $(K, X, A, f\colon X\to A)$ is defined over a finitely generated subfield of $k$ over $\QQ$, and so is every $x_{n}\in X(K)$. Then all these data with all $n\geq1$ are defined over a countable generated subfield $k_{0}$ of $k$ over $\QQ$. Fix an inclusion $k_{0}\hookrightarrow\CC$. By descent from $k$ to $k_{0}$ and base change from $k_{0}$ to $\CC$, we achieve the case that the base field is $\CC$. Note that the heights of the points $\{x_{n}\}_{n\geq1}$ do not change in this process, and taking special set of projective varieties is stable under base change by \cite[Lemma 4.2]{arXiv:2305.14789}. Thus the Proposition follows from the case when $k=\CC$.
\end{proof}

Now we are ready to prove Theorem \ref{albanese}. We first prove the following lemma.

\begin{lemma}\label{const}
Let $K$ be a function field of transcendence degree \emph{1} over a field $k$ of characteristic \emph{0} such that $k$ is algebraically closed in $K$. Let $X$ be a projective variety $X$ over $K$ which does not contain rational curves. Assume that $X(K)$ is Zariski dense in $X$. And there is a Zariski dense subset $U$ of $X(K)$ such that the height of the points in $U$ is bounded.

Then the normalization of $X$ is isomorphic to the base change $T_{K}=T\times_{k}K$ for a projective variety $T$ over $k$, and the complement of the image of the composition $T(k)\to (T_{K})(K)\to X(K)$ in $X(K)$ is not Zariski dense in $X$.
\end{lemma}
\begin{proof}
The proof is essential the same as in \cite[\S 2.2, \textbf{Constancy 1} and \S 4.2, \textbf{Constancy}]{arXiv:2305.14789}, we summarize it here. We first prove that $X/K$ is dominated by a constant family, and all the rational points come from this process. Namely, there is a smooth projective scheme $V$ over $k$, together with a dominant $K$-morphism
$$
\varphi\colon V_{K}\ \lra\ X
$$
such that the composition
$$
V(k)\ \lra\ V_{K}(K)\ \lra\ X(K)
$$
is surjective. Here we denote $V_{K}=V\times_{k}K$.
Let $K$ be the function field of a smooth projective curve $B$ over $k$, and let $\CX$ be an integral model of $\CX$ over $B$. Let $\CL$ be an ample line bundle on $\CX$. Every point $x\in X(K)$ extends to a section $\tilde{x}\subseteq\CX$ over $B$. By our assumption, there is a constant $c>0$ such that $\deg_{\CL}(\tilde{x})<c$ for every $x\in U$.

The key is that there is a quasi-projective scheme $S$ over $k$ parametrizing the set of sections $Y$ of $\CX\to B$ with $\deg_{\CL}(Y)<c$. By the moduli, we have a $k$-morphism $S\times_{k}B\to\CX$, which sends the fibre $s\times_{k}B$ for $s\in S(k)$ to the section of $\CX\to B$ corresponding to $s$. It follows that the image of the composition
$$
S(k)\ \lra\ (S\times_{k}B)(B)\ \lra\ \CX(B)\to X(K)
$$
contains $U$. Since $U$ is Zariski dense in $X$, the morphism $S\times_{k}B\to \CX$ is dominant. Taking a base change by ${\rm Spec}\, K\to B$, we obtain a dominant $K$-morphism $S\times_{k}K\to X$.

We can adjust the quasi-projective scheme $S$ to a smooth projective scheme over $k$. In fact, by taking a disjoint union of irreducible components, we can assume that $S$ is a disjoint union of quasi-projective varieties. Let $V'$ be a disjoint union of projective varieties containing $S$ as a dense and open subscheme. Let $V$ be a resolution of singularities of $V'$ by Hironaka's theorem. The morphism $S_{K}\to X$ gives a dominant rational map $V_{K}\dashrightarrow X$. It can be extended to a morphism $V_{K}\to X$ since $X$ does not contain any rational curve.

By construction, the composition
$$
V(k)\ \lra\ (V\times_{k}B)(B)\ \lra\ \CX(B)\ \lra\ X(K)
$$
has dense image. But $V$ is projective, so it is surjective and the morphism $V_{K}\to X$ is also surjective.

Note that the normalization of $X$ also satisfies the conditions in the lemma since the normalization map is finite. Therefore to prove the constancy of the normalization of $X$, we may assume that $X$ is itself normal. In this case, the proof is the same as \cite[\S 4.2, \textbf{Constancy}]{arXiv:2305.14789}, we omit it here. This finishes the proof of the first part of the lemma.
\end{proof}

\begin{remark}
A natural question is that if the condition ``$X$ does not contain rational curves'' can be removed in Lemma \ref{const} (with the conclusion replaced by ``$X$ is birational to a constant variety $T_{K}$''). We will use the lemma only for those $X$ that admit a finite map to some abelian variety, in this case $X$ does not contain rational curves.
\end{remark}

\noindent\textbf{Proof of Theorem \ref{albanese}}\medskip

\noindent After taking a normalization we may assume that $X$ is normal. First we reduce the transcendence degree of $K/k$ to 1 in Theorem \ref{albanese}. Assume that ${\rm trdeg}(K/k)>1$, we need to prove that the conclusion holds for the datum $(K/k,X)$.

Let $k_{1}$ be an intermediate field of $K/k$ such that $K/k_{1}$ has transcendence degree 1 and that $k_{1}$ is algebraically closed in $K$. As $X(K)$ is Zariski dense in $X$, it is not empty. In particular, we have an Albanese morphism $f\colon X\to{\rm Alb}(X)$ defined over $K$. Assume that Theorem \ref{albanese} holds for the datum $(K/k_{1},X)$. Then the Albanese variety of $X$ is defined over $k_{1}$, and $X$ is birational to $(X_{1})_{K}$ for some projective variety $X_{1}$ over $k_{1}$. Moreover, the complement of the
image of ${\rm Im}(X_{1}(k_{1})\to(X_{1})_{K}(K))$ in $X(K)$ (via a birational map from $(X_{1})_K$ to $X$) is not Zariski dense in $X$. Thus $X_{1}$ is of general type with $X(k_{1})$ Zariski dense in $X_{1}$. After taking a normalization of $X_{1}$, we can assume that $X_{1}$ is normal. Then $X_{1}$ is also of maximal Albanese dimension. In fact, after composing a birational map, the base change of the Albanese map of $X_{1}$ (over $k_{1}$) to $K$ is just the Albanese map of $X$. If in addition $f$ is finite, then there is a $K$-isomorphism $X_{1}\times_{k_{1}}K\to X$ for some projective variety $X_{1}$ over $k_{1}$. By \cite[Lemma 4.8]{arXiv:2305.14789}, the finite morphism $X_{1}\times_{k_{1}}K\to X\to {\rm Alb}(X)$ induces a finite $k_{1}$-morphism $X_{1}\to {\rm Alb}(X)^{(K/k_{1})}$. Therefore we can reduce Theorem \ref{albanese} to the datum $(k_{1}/k, X_{1})$. Repeat the process on $k_{1}/k$, we eventually get the case ${\rm trdeg}(K/k)=1$.

Now assume ${\rm trdeg}(K/k)=1$. Recall that we have an Albanese morphism $f\colon X\to{\rm Alb}(X)$ defined over $K$. By Stein factorization, $f$ can be decomposed as $X\to X_{0}\to{\rm Alb}(X)$ where $X\to X_{0}$ has connected fibres and $X_{0}\to{\rm Alb}(X)$ is finite. As $f\colon X\to{\rm Alb}(X)$ is generically finite onto its image, $X\to X_{0}$ is birational. In particular, $X_{0}$ is of general type and $X_{0}(K)$ is Zariski dense in $X_{0}$. Replacing $X$ by $X_{0}$, we may assume that $f$ is finite.

By Proposition \ref{ht}, the heights of the points in $(X\backslash{\rm Sp}(X))(K)$ are bounded. Since $X$ is of general type, by \cite[Corollary 4.3]{arXiv:2305.14789}, ${\rm Sp}(X)$ is a proper Zariski closed subset of $X$. Since $X(K)$ is Zariski dense in $X$, Theorem \ref{albanese} follows from Lemma \ref{const}. \qed \medskip\medskip

Finally, we can prove the main Theorem \ref{main}.\medskip

\noindent\textbf{Proof of Theorem \ref{main}}\medskip

\noindent Let $(X, Z)$ be as in Theorem \ref{main}. Let $Z_{1},\dots, Z_{r_{0}}$ be the irreducible components of $Z$. By assumption, $Z$ is the Zariski closure of $(X\backslash{\rm Sp}(X))(K)$ in $X$, so $Z_{i}(K)$ is Zariski dense in $Z_{i}$ for $i=1,\dots, r_{0}$. Moreover, each $Z_{i}$ is not contained in ${\rm Sp}(X)$ since otherwise $Z$ is contained in the union $\cup_{1\leq j\leq r_{0}, j\neq i}Z_{j}$ which is impossible. Therefore ${\rm Sp}(Z_{i})\subseteq {\rm Sp}(X)\cap Z_{i}$ is a proper closed subset of $Z_{i}$. By \cite[Corollary 4.3]{arXiv:2305.14789} which is essentially due to Ueno and Kawamata, $Z_{i}$ is of general type. Apply Theorem \ref{albanese} to $Z_{i}$ for $i=1,\dots,r_{0}$. As in part (1) of Theorem \ref{main}, we obtain $(T_{i},\rho_{i})$ for $i=1,\dots, r_{0}$. Moreover, the complement of ${\rm Im}(T_{i}(k)\to Z_{i}(K))$ in $Z_{i}(K)$ is not Zariski dense in $Z_{i}$.

If $\Sigma\colonequals ({\rm Sp}(X)\cap Z)(K)\cup(\cup_{1\leq i\leq r_{0}}{\rm Im}(T_{i}(k)\to Z(K)))$ not equals to $Z(K)$, denote $Z^{(1)}$ the Zariski closure of $Z(K)\setminus\Sigma$ in $X$. Note that every irreducible component of $Z^{(1)}$ is properly contained in an irreducible component of $Z$ and is not contained in ${\rm Sp}(X)$. Thus we can apply Theorem \ref{albanese} to the irreducible components $Z_{r_{0}+1}, Z_{r_{0}+2},\dots, Z_{r_{1}}$ of $Z^{(1)}$.  Repeat this process. Eventually, we have closed subvarieties $Z_{r_{0}+1}, Z_{r_{0}+2},\dots, Z_{r}$ of $Z$, with a $K$-isomorphism $\rho_{i}\colon T_{i}\times_{k}K\to Z'_{i}$ for a projective variety $T_{i}$ over $k$ for $i=r_{0}+1,\dots,r$, where $Z'_{i}$ is the normalization of $Z_{i}$, such that
$$
X(K)={\rm Sp}(X)(K)\cup\left(\bigcup\limits_{1\leq i\leq r}{\rm Im}(T_{i}(k)\to Z(K))\right).
$$
This finishes the proof of the main Theorem \ref{main}. \qed

\bibliographystyle{alpha}
\bibliography{refs}

\

\noindent \small{ Peking University, Beijing 100871, China}

\noindent \small{\it Email: gaoguoquan@pku.edu.cn}
\end{document}